\documentclass[11pt]{amsart}

\usepackage{amsmath,amsthm,amssymb,amsfonts,mathtools}
\usepackage{enumitem}
\usepackage{array,longtable,booktabs}
\usepackage{xcolor}

\definecolor{purple}{RGB}{128,0,128}
\usepackage[colorlinks=true,linkcolor=blue,citecolor=green,
  urlcolor=blue]{hyperref}
\hypersetup{
  pdftitle={Twisted Diophantine Approximation II: Uniform Theory},
  pdfauthor={Taehyeong Kim and Vasiliy Neckrasov},
  pdfkeywords={twisted Diophantine approximation, shrinking targets,
    parametric geometry of numbers, successive minima, Hausdorff measure}
}

\usepackage{tikz}

\setenumerate{leftmargin=*}
\setitemize{leftmargin=*}
\allowdisplaybreaks
\numberwithin{equation}{section}

\newtheorem{thm}{Theorem}[section]
\newtheorem{prop}[thm]{Proposition}
\newtheorem{lem}[thm]{Lemma}
\newtheorem{cor}[thm]{Corollary}
\newtheorem{ex}[thm]{Example}

\theoremstyle{definition}
\newtheorem{defn}[thm]{Definition}

\theoremstyle{remark}
\newtheorem{rem}[thm]{Remark}

\DeclareMathOperator{\rank}{rank}
\newcommand{\Mat}{\mathrm{Mat}}
\newcommand{\spanR}{\operatorname{span}_{\mathbb R}}

\newcommand{\R}{\mathbb R}
\newcommand{\Z}{\mathbb Z}
\newcommand{\Q}{\mathbb Q}
\newcommand{\T}{\mathbb T}
\newcommand{\N}{\mathbb N}

\newcommand{\cB}{\mathcal B}
\newcommand{\cG}{\mathcal G}
\newcommand{\cP}{\mathcal P}
\newcommand{\bq}{\mathbf q}
\newcommand{\bp}{\mathbf p}
\newcommand{\bb}{\mathbf b}
\newcommand{\bx}{\mathbf x}
\newcommand{\by}{\mathbf y}

\newcommand{\bv}{\mathbf v}

\newcommand{\q}{\mathbf q}
\newcommand{\p}{\mathbf p}
\newcommand{\x}{\mathbf x}
\newcommand{\one}{\mathbf 1}

\newcommand{\dist}{\operatorname{dist}}

\newcommand{\Sing}{{\mathrm{Sing}}}
\newcommand{\Dhat}{D}
\newcommand{\Emass}{\operatorname{Emass}}
\newcommand{\e}{\mathrm e}

\renewcommand{\le}{\leqslant}
\renewcommand{\ge}{\geqslant}

\title[Twisted Diophantine Approximation II]
{Twisted Diophantine Approximation II:\ Uniform Theory}
\author{Taehyeong Kim}
\address{Department of Mathematics, Brandeis University,
Waltham, MA 02453, USA}
\email{taehyeongkim@brandeis.edu}
\author{Vasiliy Neckrasov}
\address{Department of Mathematics, Brandeis University,
Waltham, MA 02453, USA}
\email{vneckrasov@brandeis.edu}
\date{}

\begin{document}

\begin{abstract}
 We develop a general method for uniform twisted metric Diophantine approximation with an arbitrary fixed real matrix. Using all successive minima of the associated diagonal lattice trajectory, we estimate lattice-point counts, accounting for clustering in every direction. Ratios of these counts give exact Hausdorff dimensions of sets of twisted $\psi$-Dirichlet vectors for a large class of functions $\psi$, including all except finitely many power functions {$T^{-\tau}$, $\tau>0$}. We obtain Hausdorff dimension formulae for the endpoint sets obtained by taking the intersection and union of the sets of twisted $c\psi$-Dirichlet vectors over
$c>0$, and prove that taking the intersections causes no dimension drop. Furthermore, we obtain Hausdorff dimensions of level sets of the uniform Diophantine exponent.

In addition, we prove zero--full laws for general Hausdorff measures of $\psi$-Dirichlet sets.
Applications recover the one-dimensional Kim--Liao formulae, strengthen Aggarwal's escape-of-mass upper bound, reproduce the metric criteria of Moshchevitin and the second author and give certain extensions of the above statements.
\end{abstract}

\maketitle

\tableofcontents

\part{Results}

\section{Introduction}\label{sec:introduction}

\subsection{Background and motivation}\label{sec:previous}

Let $m,n\in\N$, set $\R_+=(0,\infty)$, put $d=m+n$, and let
$\Mat_{m,n}$ denote the set of real $m\times n$ matrices. We use the
supremum norm, denoted by $|\cdot|$, in every finite-dimensional vector
space. For a non-increasing function $\psi:\R_+\to\R_+$ tending to zero, a pair
$(A,\bb)\in\Mat_{m,n}\times\R^m$ is called \emph{$\psi$-Dirichlet} if
the system
\begin{equation}\label{eq:1.2}
 |A\q-\p-\bb|\le\psi(T),\qquad |\q|\le T,
 \qquad (\p,\q)\in\Z^m\times\Z^n
\end{equation}
has a solution for every sufficiently large $T$.

Kleinbock and Wadleigh \cite[Theorem~1.6]{kwad} established a
zero-one law for the Lebesgue measure of the set of $\psi$-Dirichlet
pairs: in the notation of \eqref{eq:1.2}, this set has full or zero
measure according as
\[
 \sum\limits_{k=1}^{\infty} \frac{1}{k^{n+1}\psi(k)^m}
\]
converges or diverges. The first named author and Kim \cite[Theorem~1.3]{kimkim}
obtained the corresponding Hausdorff-measure laws for the
complementary sets of non-$\psi$-Dirichlet pairs, and deduced
dimension formulae for sublevel sets of the uniform Diophantine
exponent \cite[Corollary~1.5]{kimkim}.

The more classical inhomogeneous formulation fixes $\bb$ and varies
$A$. For every $\bb\notin\Z^m$, \cite[Theorem~1.4]{kimkim} proves the corresponding Hausdorff-measure law in this singly metric setting. {In particular, for the power functions $\psi_\tau(T):=T^{-\tau}$ with $0<\tau\le n/m$, their result gives}
\[
 \dim_H\{A\in\Mat_{m,n}:(A,\bb)\text{ is not {$\psi_\tau$-Dirichlet}}\}
 =mn-\frac{n-m\tau}{1+\tau}.
\]
\noindent See also \cite{Sch22} for some lower bounds of dimensions of related sets.

In \emph{twisted} inhomogeneous approximation, the matrix $A$ is
fixed and the shift $\bb$ is the metric variable. We identify
$\T^m=\R^m/\Z^m$ with $[0,1)^m$.
For a fixed $A$, put
\begin{equation}\label{eq:1.1}
 \Dhat_A(\psi):=
 \{\bb\in\T^m:(A,\bb)\text{ is $\psi$-Dirichlet}\}.
\end{equation}
The problem considered in this paper is the following:
\[
 \text{Given $A$ and $\psi$, determine }\dim_H\Dhat_A(\psi).
\]
This question was previously raised in \cite[\S 5.1]{kimkim}; see also \cite[\S 7.3]{kwad}. Results for almost every pair do not answer this question for an arbitrary fixed matrix. {Different denominator vectors can give nearby points in $\T^m$
when $A$ maps their differences close to integer vectors. As a result, many denominators can contribute to the same small region, and the
amount of clustering depends on $A$.}

There are already criteria for both extreme measure behaviours. Jarn\'ik's transference theorem \cite{Jar39}, in the formulation of \cite[Theorem~A]{MN25}, gives conditions under which every shift is uniformly approximable. A corresponding measure-zero result was proved by Moshchevitin and the second author (\cite[Theorem~3.1]{MN25}; see also \cite[Theorem~1.5]{N26} for the weighted version)\footnote{Earlier, Bugeaud and Laurent (\cite{BL05}; see also \cite{CGGMS20} for the weighted version) established the measure-zero result for exponents.}.  These criteria use homogeneous approximation to $A^\top$ and allow changes of constants in $\psi$; they do not determine the dimension in the null case. We state both, with extensions, in Propositions \ref{prop:transference-null} and~\ref{prop:transference-full}.

To the best of our knowledge, the only general exact formulae for nontrivial Hausdorff dimensions in the twisted uniform setting prior to the present work are those of Kim and Liao \cite{KL19} for irrational numbers (the case $m=n=1$) and power functions, away from the critical parameters. Their formulae use continued fractions. We recover their non-critical results in Theorem~\ref{cor:kim-liao}, and extend the explicit one-dimensional description to more general functions in
Theorem~\ref{thm:one-dimensional}. Another application strengthens Aggarwal's escape-of-mass upper bound \cite{Agg26} and gives an upper bound for all power functions; see Propositions~\ref{prop:aggarwal-special} and~\ref{prop:aggarwal-general}.

The purpose of this paper is to express the uniform dimension problem through homogeneous lattice data for an arbitrary fixed matrix. All successive minima explicitly enter the formula.
This is the uniform counterpart to the asymptotic problem studied in \cite{KN26}. 

\subsection{The main theorem}\label{subsec:main-theorem}

From now on, unless stated otherwise, $\psi:\R_+\to\R_+$ is continuous,
strictly decreasing, and tends to zero.
Changing $\psi$ on a bounded interval does not affect any of the sets
under consideration. We therefore make the harmless normalization
$\psi(T)\to\infty$ as $T\to0$. Thus $\psi$ and
$T\mapsto T/\psi(T)$ have inverses on $\R_+$.

Our main theorem requires a one-sided regularity condition which
excludes arbitrarily rapid decrease on a fixed multiplicative
scale. The corresponding two-sided condition was called
quasimultiplicativity in \cite[\S1.3]{N26}.

\begin{defn}\label{def:regular}
A decreasing function $\psi$ is called \emph{lower quasimultiplicative}
if there exist $0<\tau<\infty$ and $R_0>1$ such that
\begin{equation}\label{eq:1.3}
 R^{-\tau}\psi(T)\le\psi(RT)
\end{equation}
for every sufficiently large $T$ and every $R\ge R_0$.
\end{defn}

Lower quasimultiplicativity essentially means that the rate of decay of $\psi$ over all sufficiently large multiplicative intervals does not exceed that of some fixed power function. In particular, every power function $\psi_\tau(T)=T^{-\tau}$, with $\tau>0$, is lower quasimultiplicative.

By monotonicity, \eqref{eq:1.3} implies that, for every $R>1$,
\begin{equation}\label{eq:lower-qm-comparison}
 \max\{R,R_0\}^{-\tau}\psi(T)\le\psi(RT)\le\psi(T)
\end{equation}
for all sufficiently large $T$.

For lower quasimultiplicative $\psi$, define
\begin{equation}\label{eq:classical-endpoints}
 \Sing_A(\psi):=\bigcap_{c>0}\Dhat_A(c\psi),
 \qquad
  U_A(\psi):=\bigcup_{c>0}\Dhat_A(c\psi).
\end{equation}
Here $c\psi$ denotes the function $T\mapsto c\psi(T)$.
The first set consists of shifts for which arbitrarily small
constants are admissible; the second allows an arbitrary fixed
constant. We first compute the dimensions of these two endpoint
sets, and then identify when they determine the dimension of
$\Dhat_A(\psi)$ itself.

We now introduce the lattice data. Put
\begin{equation}\label{eq:1.6}
 u_A:=
 \begin{pmatrix}
 I_m&A\\
 0&I_n
 \end{pmatrix},
 \qquad
 a_t:=
 \begin{pmatrix}
 \e^{t/m}I_m&0\\
 0&\e^{-t/n}I_n
 \end{pmatrix}.
\end{equation}
Let $\lambda_i(\Lambda)$ denote the $i$-th successive minimum of a
lattice $\Lambda\subseteq\R^d$ relative to the unit cube, and write
\begin{equation*}\label{eq:1.7}
 \lambda_i(t):=\lambda_i(a_tu_A\Z^d),
 \qquad 1\le i\le d.
\end{equation*}
For {$R,Q>0$}, define
\begin{equation}\label{eq:2.1}
 {t_{R,Q}}:=\frac{mn}{d}\log\frac QR,
 \qquad
 \mathcal N_A(R,Q):=
 \prod_{i=1}^d
 \max\left\{
 \frac{Q^{n/d}R^{m/d}}{\lambda_i({t_{R,Q}})},1
 \right\}.
\end{equation}
This quantity records the homogeneous clustering at scale $(R,Q)$:
\begin{equation*}\label{eq:1.10}
 \#\left\{(\p,\q)\in\Z^m\times\Z^n:
 |A\q-\p|\le R,\ |\q|\le Q\right\}
 \asymp\mathcal N_A(R,Q).
\end{equation*}
The comparison is proved in Lemma~\ref{lem:profile-count}. In
particular, the data in \eqref{eq:2.1} are entirely homogeneous.

For $L>0$, let
\begin{equation}\label{eq:1.11}
 T_j:=\psi^{-1}(\e^{-jL}),\qquad j\ge0.
\end{equation}
Thus consecutive radii $\psi(T_j)$ differ by the fixed factor
$\e^{-L}$. 

\medskip 
For $L,c>0$, we define%
\begin{equation}\label{eq:4.19}
 \mathcal P_{A,\psi}(L,c):=
 \liminf_{N\to\infty}
 \frac1{NL}\sum_{j=0}^{N-1}
 \log
 \frac{\mathcal N_A(c\psi(T_j),T_{j+1})}
      {\mathcal N_A(c\psi(T_{j+1}),T_{j+1})}.
\end{equation}

The numerator and denominator use the same denominator bound and
two consecutive approximation radii. Their ratio governs the
number of smaller approximation balls that can be retained inside
a larger one; see Propositions~\ref{prop:upper} and~\ref{prop:lower}.
The normalization $NL=-\log\psi(T_N)$ converts their accumulated
growth into dimension.

We write $\mathcal H^s$ for $s$-dimensional Hausdorff measure and $\dim_H$ for Hausdorff dimension, using the quotient supremum metric on $\T^m$. We will also denote the $m$-dimensional Lebesgue measure by $\mathcal L^m$.

\begin{thm}\label{thm:main}
Let $A\in\Mat_{m,n}$ and let $\psi$ be continuous, strictly decreasing,
lower quasimultiplicative, and tend to zero.
\begin{enumerate}[label=\textup{(\roman*)}]
\item The limits
\begin{equation}\label{eq:pressure-minus}
 P_-(A,\psi):=
 \lim_{L\to\infty}\inf_{c>0}\mathcal P_{A,\psi}(L,c)
\end{equation}
and
\begin{equation*}\label{eq:pressure-plus}
 P_+(A,\psi):=
 \lim_{L\to\infty}\sup_{c>0}\mathcal P_{A,\psi}(L,c)
\end{equation*}
exist. One has
\begin{equation}\label{eq:1.15}
 \dim_H\Sing_A(\psi)
 =\inf_{c>0}\dim_H\Dhat_A(c\psi)
 =P_-(A,\psi)
\end{equation}
and
\begin{equation}\label{eq:1.16}
 \dim_H U_A(\psi)
 =\sup_{c>0}\dim_H\Dhat_A(c\psi)
 =P_+(A,\psi).
\end{equation}
\item The quantity $\dim_H\Dhat_A(c\psi)$ is independent of $c>0$
if and only if
\begin{equation}\label{eq:1.17}
 P_-(A,\psi)=P_+(A,\psi).
\end{equation}
Under this $c$-invariance condition,
\begin{equation}\label{eq:1.18}
 \begin{split}
 \dim_H\Dhat_A(\psi)
 &=\dim_H\Sing_A(\psi)=\dim_H U_A(\psi)\\
 &=P_-(A,\psi)=P_+(A,\psi)
 =\lim_{L\to\infty}\mathcal P_{A,\psi}(L,1).
 \end{split}
\end{equation}
\end{enumerate}
\end{thm}

\begin{rem}\label{rem:generic-c-invariance}
The value $\dim_H\Dhat_A(c\psi)$ is independent of $c$ (in which case Theorem~\ref{thm:main} gives the exact expression for $\dim_H\Dhat_A(\psi)$) generically along power families, in the following precise sense:
for every fixed $A$ and every lower quasimultiplicative $\psi$, the dimension $\dim_H\Dhat_A(c\psi^s)$ is independent of $c>0$ for all but countably many $s>0$. Indeed, if $0<s<t$, then
$c_2\psi(T)^t\le c_1\psi(T)^s$ eventually for any $c_1,c_2>0$.
Theorem~\ref{thm:main} therefore gives
\begin{equation}\label{eq:power-pressure-ordering}
 P_+(A,\psi^t)\le P_-(A,\psi^s).
\end{equation}
The nonempty intervals
$(P_-(A,\psi^s),P_+(A,\psi^s))$ are consequently pairwise disjoint,
so there are only countably many of them. Outside this exceptional
set, \eqref{eq:1.18} gives the exact dimension for every $c>0$.
The exceptional set here may depend on $A$ and $\psi$.
In Section~\ref{subsec:near-powers}, for functions similar to power functions, we give an explicit finite set containing all possible exceptional powers, which is {independent} of $A$.
\end{rem}

In particular, for any $c>0$ the theorem gives the bounds
\[
 P_-(A,\psi)\le\dim_H\Dhat_A(c\psi)\le P_+(A,\psi)
 \qquad \text{for any} \ c>0.
\]

We note that the equality $\dim_H U_A(\psi)
 =\sup_{c>0}\dim_H\Dhat_A(c\psi)$ in \eqref{eq:1.16} follows directly from the countable stability of Hausdorff dimension. 
The situation with the lower endpoint set $\Sing_A(\psi)$ is different: a separate argument is needed to prove that no dimension drop is happening and $\dim_H\Sing_A(\psi)
 =\inf_{c>0}\dim_H\Dhat_A(c\psi)$.
 
A related no-dimension-drop phenomenon occurs in homogeneous
approximation. Write $D(\psi)$ for the set of matrices for
which \eqref{eq:1.2}, with $\bb=\mathbf0$ and the additional
restriction $\q\ne\mathbf0$, is solvable for every sufficiently
large $T$. If $\operatorname{Sing}(m,n)$ denotes the set of
singular matrices and $(m,n)\ne(1,1)$, then
\[
 \dim_H\operatorname{Sing}(m,n)
 =\inf_{c>0}\dim_H D(cT^{-n/m}).
\]
This follows by combining the covering
estimate of \cite[Theorem~1.5]{KKLM17} with the dimension formula
of \cite[Theorem~3.1]{DFSU24}. We also note that, to the best of the authors' knowledge, no general explicit evaluation of the Hausdorff dimensions of the sets of $\psi$-singular $m \times n$ matrices is known in the homogeneous case.

 In section \ref{subsec:continuity} we suggest two alternative representations of the endpoint sets in which the union and intersection are interchanged, showing that the structure is more symmetric than it can be seen from definitions.

\subsection{On functions close to powers}\label{subsec:near-powers}

For power functions, the possible exceptional exponents in Remark
\ref{rem:generic-c-invariance} belong to the finite set
\begin{equation}\label{eq:power-exceptional-set}
 \mathcal E_{m,n}:=
 \left\{\frac ba:a,b\in\N,\ 1\le a\le m,\ 1\le b\le n\right\}.
\end{equation}
{Define}
\begin{equation}\label{eq:power-rank}
 r_{A,\tau}(T):=
 \#\left\{1\le i\le d:
 \lambda_i\!\left(\frac{mn(1+\tau)}d\log T\right)
 <T^{(n-m\tau)/d}\right\}.
\end{equation}
Here we used ${t_{T^{-\tau},T}}=\frac{mn(1+\tau)}d\log T$ in
\eqref{eq:2.1}. It can be useful to note that $r_{A,\tau}$ counts the number of factors in the definition of $\mathcal N_A(T^{-\tau},T)$ that exceed 1.

\begin{thm}\label{thm:near-powers}
Let $A\in\Mat_{m,n}$ and $\tau\in(0,\infty)\setminus\mathcal E_{m,n}$.
For every $c>0$,
\begin{equation}\label{eq:power-rank-dimension}
 \dim_H\Dhat_A(c\psi_\tau)
 =\liminf_{T\to\infty}
 \frac{\displaystyle\int_1^T r_{A,\tau}(S)\,\frac{dS}{S}
       -\log\mathcal N_A(T^{-\tau},T)}{(1+\tau)\log T}.
\end{equation}
In particular, this dimension is independent of $c$. Its common value
is continuous as a function of $\tau$ on each component of
$(0,\infty)\setminus\mathcal E_{m,n}$.

Moreover, if $\psi$ is continuous, strictly decreasing to zero, and
\begin{equation}\label{eq:limiting-power-order}
 \lim_{T\to\infty}\frac{\log\psi(T)^{-1}}{\log T}=\tau,
\end{equation}
then, for every $c>0$,
\begin{equation}\label{eq:near-power-dimension}
 \dim_H\Dhat_A(c\psi)=\dim_H\Dhat_A(\psi_\tau),
\end{equation}
so the same right-hand side of \eqref{eq:power-rank-dimension} gives
the dimension for $c\psi$. No lower quasimultiplicativity is required in
this last assertion.
\end{thm}

Many natural classes of functions satisfy \eqref{eq:limiting-power-order}: for instance, we can multiply power functions by any powers of logarithms (or repeated logarithms), and Theorem \ref{thm:near-powers} will apply.

\begin{rem}\label{rem:exceptional-powers} 
It is necessary to exclude the finite set of exceptional exponents from consideration in Theorem \ref{thm:near-powers}: for each of these exponents there is a matrix $A$ such that $\dim_H\Dhat_A(c\psi_\tau)$ is not $c$-invariant. Fix $\tau=b/a\in\mathcal E_{m,n}$, choose a badly approximable matrix $B\in\Mat_{a,b}$, and let $A\in\Mat_{m,n}$ have $B$ as its upper-left block and all other
entries zero. Then
\[
 \Dhat_A(c\psi_\tau)
 =\Dhat_B(c\psi_\tau)\times\{\mathbf0\}
 \subseteq\T^a\times\{\mathbf0\}.
\]
Bad approximability is preserved by transposition. Hence
Propositions~\ref{prop:trivial-shifts}
and~\ref{prop:transference-full}, applied to $B$, give constants
$0<c_0<c_1$ such that
\[
 \Dhat_A(c\psi_\tau)=
 \begin{cases}
 \langle B\Z^b \rangle\times\{\mathbf0\},
       &0<c<c_0,\\
 \T^a\times\{\mathbf0\},
       &c>c_1.
 \end{cases}
\]
Here $\langle \cdot \rangle$ denotes reduction modulo $\Z^a$. Their dimensions are respectively $0$ and $a$. 

We do not claim, however, that each exponent in $\mathcal E_{m,n}$ is exceptional for every matrix $A$.
\end{rem}

\smallskip

We are finishing this subsection with a result about the level sets of Diophantine exponents. Define the \emph{uniform Diophantine exponent} by
\begin{equation}\label{eq:uniform-shift-exponent}
 \widehat\omega(A,\bb):=
 \sup\{s>0:\bb\in\Dhat_A(\psi_s)\},
\end{equation}
where the supremum of the empty set is zero. In particular,
$\widehat\omega(A,\bb)=\infty$ for $\bb\in A\Z^n\bmod\Z^m$.
The following statement includes the critical exponents.

\begin{prop}\label{prop:uniform-exponent-levels}
For every $A\in\Mat_{m,n}$ and every $\tau>0$,
\begin{align}
 \dim_H\{\bb\in\T^m:\widehat\omega(A,\bb)\ge\tau\}
 &=\lim_{s\nearrow\tau}P_\pm(A,\psi_s) = \inf_{0<\sigma<\tau}\dim_H\Dhat_A(\psi_\sigma),
 \label{eq:exponent-superlevel-dimension}\\
 \dim_H\{\bb\in\T^m:\widehat\omega(A,\bb)>\tau\}
 &=\lim_{s\searrow\tau}P_\pm(A,\psi_s) = \sup_{\sigma>\tau}\dim_H\Dhat_A(\psi_\sigma).
 \label{eq:exponent-strict-superlevel-dimension}
\end{align}
All the involved limits exist; one can choose both $+$ and $-$ in the expression above. If the set $\{\bb:\widehat\omega(A,\bb)=\tau\}$ is nonempty,
then
\begin{equation}\label{eq:exact-exponent-dimension}
 \dim_H\{\bb\in\T^m:\widehat\omega(A,\bb)=\tau\}
 =\lim_{s\nearrow\tau}P_-(A,\psi_s).
\end{equation}
\end{prop}

\begin{cor}\label{cor:noncritical-exponent-levels}
If $\tau\notin\mathcal E_{m,n}$, then
\begin{equation}\label{eq:noncritical-exponent-dimension}
\begin{aligned}
 \dim_H\{\bb\in\T^m:\widehat\omega(A,\bb)\ge\tau\}
 &=\dim_H\{\bb\in\T^m:\widehat\omega(A,\bb)>\tau\}\\
 &=\dim_H\Dhat_A(\psi_\tau).
\end{aligned}
\end{equation}
The same equality holds for the exact level
$\{\bb:\widehat\omega(A,\bb)=\tau\}$ whenever it is nonempty.
Thus these dimensions are given by \eqref{eq:power-rank-dimension}.
\end{cor}

We again note that the homogeneous analog of this problem -- finding Hausdorff dimension of level sets of the uniform Diophantine exponent -- has not been completely settled (see \cite[Theorem 4.9, Theorem 4.10]{DFSU24} for some results in terms of template geometry). In particular, the monotonicity and continuity (away from critical exponents) of this dimension, which in twisted case follows for nonempty sets from Theorem \ref{thm:near-powers} and Proposition \ref{prop:uniform-exponent-levels}, remain open (\cite[Problem 5.3]{DFSU24}).

\subsection{Zero-full dichotomy for the Hausdorff measure}
\label{subsec:zero-full}

Theorem~\ref{thm:main} determines whether the $s$-dimensional Hausdorff measure $\mathcal H^s$ vanishes or is infinite on the two endpoint sets away from their critical
dimensions. In this section we present a more general zero-full law, which includes critical exponents and does not require lower quasimultiplicativity. For this statement it is
natural to introduce the following notation. For $c>0$, put
\begin{equation}\label{eq:rescaled-function}
 \widetilde\psi_c(T):=c\psi(T/c).
\end{equation}
For a general continuous, strictly decreasing function $\psi$
tending to zero, we put
\begin{equation}\label{eq:1.14}
 \Sing_A(\psi):=\bigcap_{c>0}\Dhat_A(\widetilde\psi_c),
 \qquad
  U_A(\psi):=\bigcup_{c>0}\Dhat_A(\widetilde\psi_c).
\end{equation}
The following lemma shows that these definitions agree with
\eqref{eq:classical-endpoints} whenever lower quasimultiplicativity
holds.

\begin{lem}\label{lem:rescaling-comparison}
Suppose that $\psi$ is lower quasimultiplicative. For every $c>0$
there are constants $a_c,b_c>0$ such that
\begin{equation}\label{eq:rescaling-comparison}
 a_c\psi(T)\le\widetilde\psi_c(T)\le b_c\psi(T)
\end{equation}
for all sufficiently large $T$. Consequently, the definitions in
\eqref{eq:classical-endpoints} and \eqref{eq:1.14} agree.
\end{lem}

\begin{proof}
Inequalities \eqref{eq:rescaling-comparison} follow from
\eqref{eq:lower-qm-comparison}, applied with
$R=\max\{c,c^{-1}\}$ when $c\ne1$; the case $c=1$ is immediate.
They give one inclusion in each comparison of the intersections
and unions in \eqref{eq:classical-endpoints} and \eqref{eq:1.14}.
For the reverse inclusions, observe that
\begin{equation}\label{eq:elementary-rescaling-comparison}
 \widetilde\psi_c(T)\le c\psi(T)\quad(0<c\le1),
 \qquad
 \widetilde\psi_c(T)\ge c\psi(T)\quad(c\ge1).
\end{equation}
\end{proof}

Before stating the main results of this section, we introduce the groups in which zero-full law for Hausdorff measures becomes meaningful. Put
\begin{equation*}\label{eq:orbit-closure}
 \Gamma_A:=A\Z^n\bmod\Z^m,
 \qquad H_A:=\overline{\Gamma_A}\subseteq\T^m,
\end{equation*}
and let $\mathcal L_A$ be normalized Haar measure on $H_A$.
By the Closed Subgroup Theorem, $H_A$ is a compact Lie group
\cite[Theorem~20.12]{Lee}. When $H_A=\T^m$, one has
$\mathcal L_A=\mathcal L^m$; more generally, if $\dim H_A>0$,
the restriction of $\mathcal H^{\dim H_A}$ to $H_A$ is a positive
constant multiple of $\mathcal L_A$.

Every point of $\Gamma_A$ belongs to $\Dhat_A(\psi)$ by using a
fixed denominator. Conversely, if $\bb\notin H_A$, then
$|A\q-\bb|_{\Z^m}\ge\dist(H_A,\bb)>0$, so $\psi(T)\to0$
excludes $\bb$ from $\Dhat_A(\psi)$. Here and hereafter, 
\begin{equation}\label{eq:distance_to_the_nearest_integer}
 |\bx|_{\Z^m}:=\min_{\p\in\Z^m}|\bx-\p|.
\end{equation}
Thus
\begin{equation*}\label{eq:ambient-chain}
 \Gamma_A\subseteq\Sing_A(\psi)\subseteq\Dhat_A(\widetilde\psi_c)
 \subseteq U_A(\psi)\subseteq H_A
 \qquad \text{for any} \ c>0.
\end{equation*}

Lastly, the zero-full theorem can be stated for the general Hausdorff measures. A \emph{dimension function} is a continuous nondecreasing function $f:\R_+\to\R_+$ such that $f(r)\to0$ as $r\to0$. We denote the corresponding Hausdorff measure by $\mathcal H^f$. Thus $f(r)=r^s$ gives $\mathcal H^s$. A set $E\subseteq H_A$ has \emph{locally full $\mathcal H^f$-measure} if
\[
 \mathcal H^f(E\cap B)=\mathcal H^f(B)
\]
for every nonempty relatively open ball $B\subseteq H_A$.

\begin{thm}\label{thm:zero-full}
Let $A\in\Mat_{m,n}$ and let $\psi$ be continuous, strictly decreasing,
and tend to zero.
\begin{enumerate}[label=\textup{(\roman*)}]
\item\label{enu:zero-full-amplification}
Let $f$ be a dimension function. There exists $h_0\in[0,+\infty]$
such that, for $b>0$,
\begin{equation}
\mathcal H^f\bigl(\Dhat_A(\widetilde\psi_b)\bigr) = \begin{cases}
 \mathcal H^f(H_A),& b>h_0,\\
 0,& 0<b<h_0.
\end{cases}
\end{equation}
In the first case, $\Dhat_A(\widetilde\psi_b)$ actually has locally full
$\mathcal H^f$-measure in $H_A$.
In particular,
\begin{equation}\label{eq:Haar-amplification}
 \mathcal L_A\bigl(\Dhat_A(\widetilde\psi_a)\bigr)>0
 \ \Rightarrow\
 \mathcal L_A\bigl(\Dhat_A(\widetilde\psi_b)\bigr)=1 \quad \text{for any} \ b > a.
\end{equation}
\item\label{enu:zero-full-exact}
Let $a>0$. If
\[
 \mathcal L_A\bigl(\Dhat_A(\widetilde\psi_a)\bigr)=1,
\]
then
\begin{equation}\label{eq:exact-amplification}
 \Dhat_A(\widetilde\psi_{2a})=H_A.
\end{equation}
\item\label{enu:zero-full-subgroups}
The sets $\Sing_A(\psi)$ and $ U_A(\psi)$ are Borel
subgroups of $H_A$. For every dimension function $f$, each has
either zero or locally full $\mathcal H^f$-measure in $H_A$.
\end{enumerate}
\end{thm}

In particular, the zero-full law for the endpoints sets still holds at each positive critical dimension in Theorem~\ref{thm:main}.

We note that the same zero-full law in $\Mat_{m,n}$ holds in the homogeneous case for the sets $\Sing(m,n)$ and, more generally, $ \Sing(\psi):=\bigcap_{c>0}\Dhat(\widetilde\psi_c)$. Indeed: these sets are invariant under translation by rational matrices, and the zero-full dichotomy directly follows from \cite[Theorem 1]{BDK05}.

\subsection{Continuity of dimension and endpoint sets representation}
\label{subsec:continuity}

In \eqref{eq:classical-endpoints}, $\Sing_A(\psi)$ is defined
by an intersection and $ U_A(\psi)$ by a union over
constant multiples of $\psi$. The following proposition
reverses these operations: $\Sing_A(\psi)$ is a union over
approximating functions $\vartheta=o(\psi)$, whereas
$ U_A(\psi)$ is an intersection over functions
$\varphi$ satisfying $\psi=o(\varphi)$. Here, $\psi=o(\varphi)$ is meant to denote that $\lim\limits_{T \to \infty} \frac{\psi(T)}{\varphi(T)} = 0$.
In the supremum, infimum, union and intersection below, all approximating functions are understood to be continuous, strictly decreasing, and tend to zero.

\begin{samepage}
\begin{prop}
\label{prop:endpoint-representations}
Under the assumptions of Theorem \ref{thm:main}, the following hold.
\begin{enumerate}[label=\textup{(\roman*)}]
\item\label{enu:upper-intersection}
\begin{equation}\label{eq:1.28}
  U_A(\psi)
 =\bigcap_{\psi=o(\varphi)}\Dhat_A(\varphi).
\end{equation}
The intersection in \eqref{eq:1.28} is unchanged if it is restricted to
lower quasimultiplicative functions.

\item\label{enu:lower-union}
\begin{equation}\label{eq:sing-union}
 \Sing_A(\psi)
 =\bigcup_{\vartheta=o(\psi)}\Dhat_A(\vartheta).
\end{equation}
The union in \eqref{eq:sing-union} is unchanged if it is restricted to
lower quasimultiplicative functions.
\end{enumerate}
\end{prop}
\end{samepage}

The following proposition complements the endpoint dimension formulae \eqref{eq:1.15} and \eqref{eq:1.16}: the dimension of $\Sing_A(\psi)$ can also be approximated arbitrarily closely from below by dimensions of $\Dhat_A(\vartheta)$ with $\vartheta=o(\psi)$, and that of $ U_A(\psi)$ from above by dimensions of $\Dhat_A(\varphi)$ with $\psi=o(\varphi)$. These conclusions do not follow from the set representations alone: intersections can lose dimension, and the union in \eqref{eq:sing-union} is indexed by an uncountable family. They express continuity of the endpoint dimensions with respect to asymptotic comparison of approximating functions.

\begin{prop}
\label{prop:dimension-continuity}
Under the assumptions of Theorem \ref{thm:main},
\begin{equation}\label{eq:dimension-continuity}
\begin{split}
 \dim_H\Sing_A(\psi)
 &=\sup_{\vartheta=o(\psi)}\dim_H\Dhat_A(\vartheta)
 \le \dim_H\Dhat_A(\psi)\\
 &\le \inf_{\psi=o(\varphi)}\dim_H\Dhat_A(\varphi)
 =\dim_H U_A(\psi).
\end{split}
\end{equation}
The supremum and the infimum in \eqref{eq:dimension-continuity} are
unchanged if they are restricted to lower quasimultiplicative functions.
\end{prop}

\subsection{Outline of the proof of Theorem~\ref{thm:main}}

Fix an increasing sequence $T_j\to\infty$, and put
\[
 \mathcal P_{T_j}:=
 \{B(\langle A\q\rangle,\psi(T_j)):
       \q\in\Z^n,\ |\q|\le T_j\},
 \qquad
 E_j:=\bigcup_{B\in\mathcal P_{T_j}}B,
\]
where $\langle\cdot\rangle$ denotes reduction modulo $\Z^m$ and
$B(\bb,r)$ is the closed ball of radius $r$ about $\bb\in\T^m$.
Let $\Dhat_{A,T_k}(\psi)$ consist of the shifts for which
\eqref{eq:1.2} is solvable for every $T\ge T_k$. Then
\[
 \Dhat_{A,T_k}(\psi)\subseteq\bigcap_{j\ge k}E_j,
 \qquad
 \Dhat_A(\psi)=\bigcup_{k\ge1}\Dhat_{A,T_k}(\psi).
\]

For the upper bound, fix $k$ and choose a maximal disjoint
subcollection (Vitali subcollection) $\mathcal G_{T_k}^{(k)}\subseteq\mathcal P_{T_k}$.
For each $Q\in\mathcal G_{T_j}^{(k)}$, put
\[
 \mathcal P_{T_j\to T_{j+1}}^{(k)}(Q)
 :=
 \{B\in\mathcal P_{T_{j+1}}:B\subseteq6Q\},
\]
and select a Vitali subcollection
$\mathcal G_{T_j\to T_{j+1}}^{(k)}(Q)$.
Their union over $Q$ defines $\mathcal G_{T_{j+1}}^{(k)}$.
By the Vitali covering argument, triples of the selected balls cover
each candidate union. Lemma~\ref{lem:upper-cover} shows that
$4\mathcal G_{T_j}^{(k)}$ covers $\Dhat_{A,T_k}(\psi)$.
These fixed dilations do not affect the dimension estimate.

To count the selected balls, lift them to $\R^{m+n}$. Lemma~\ref{lem:profile-count} expresses the resulting
lattice-point counts through all the successive minima, accounting
for clustering in every direction. Dividing the count in the parent
box by the count corresponding to one next-stage ball gives
Lemma~\ref{lem:upper-transition}:
\[
 \#\mathcal G_{T_j\to T_{j+1}}^{(k)}(Q)
 \ll
 \frac{\mathcal N_A(\psi(T_j),T_{j+1})}
      {\mathcal N_A(\psi(T_{j+1}),T_{j+1})}.
\]
Multiplying these bounds and comparing their logarithms with
$-\log\psi(T_\ell)$ proves the upper dimension estimate, recorded in Proposition~\ref{prop:upper}.

For the lower bound, start with a single ball and replace
$B\subseteq6Q$ by $B\subseteq Q$ in the definition of
$\mathcal P_{T_j\to T_{j+1}}^{(k)}(Q)$.
The available margins are now
$\psi(T_j)-\psi(T_{j+1})$ in the approximation coordinates and
$T_{j+1}-T_j$ in the denominator coordinates.
Under certain regularity conditions on the sequence $\{ T_j \}$ and function $\psi$, Lemma
\ref{lem:lower-transition} gives the same transition estimate as we had for the upper bound,
up to constants. The nested unions of the collections
$\mathcal G_{T_j}^{(k)}$ produce a compact set
$\mathbf G \subseteq\bigcap_{j\ge k}E_j$.

The condition $\psi(T_j)\le\rho\psi(T_{j+1})$ bridges the gaps between
mesh times: in Lemma~\ref{lem:limit-inclusion} we show that
\[
 \mathbf G\subseteq\Dhat_{A,T_k}(\rho\psi)
 \subseteq\Dhat_A(\rho\psi).
\]
It also controls intermediate radii in the mass-distribution
argument of Lemma~\ref{lem:tree}, which yields the lower bound recorded in Proposition~\ref{prop:lower}.

Theorem~\ref{thm:uniform-pressure} then follows by choosing suitable
parameters. We choose the equal-$\psi$ mesh \eqref{eq:1.11}, and together with the lower quasimultiplicativity of $\psi$ it gives us all the necessary regularity assumptions. The fixed transition constants from Propositions~\ref{prop:upper} and \ref{prop:lower} contribute only $O(L^{-1})$.
Taking the infimum and supremum over $c>0$ and letting $L\to\infty$, we identify  $P_-(A,\psi)$ and $P_+(A,\psi)$ with the infimum and supremum of $\dim_H\Dhat_A(c\psi)$.

The remaining nontrivial step is that passing to the intersection
defining $\Sing_A(\psi)$ causes no dimension drop in comparison to $\inf\limits_{c > 0} \dim_H\Dhat_A(c\psi)$.
We construct a function
$\vartheta=o(\psi)$ for which the lower construction gives the dimension of $\Dhat_A(\vartheta)$
arbitrarily close to $P_-(A,\psi)$.
This proves Proposition~\ref{prop:no-drop}.
For $ U_A(\psi)$, the corresponding equality follows from
countable stability and
$ U_A(\psi)=\bigcup_{N\ge1}\Dhat_A(N\psi)$,
completing the proof of Theorem~\ref{thm:main}.

\medskip
\noindent\textbf{Companion paper.}
The companion paper~\cite{KN26} studies asymptotic twisted approximation
for arbitrary fixed matrices. For positive non-increasing approximation
functions satisfying dyadic regularity, it establishes an exact zero--one
law for the Lebesgue measure of asymptotically approximable shifts.
Applications include an almost-sure zero--infinity law for inhomogeneous
$\psi$-Lagrange constants, separate Hausdorff-measure criteria, and
dimension results. The two papers form parts of the same project, with
the same ideas leading to complementary results in the asymptotic and
uniform settings.

\medskip
\noindent\textbf{Organization.}
Section~\ref{sec:applications} develops the applications of our main theorems: the one-dimensional dimension formulae in Section \ref{sec:one-dimensional}, bounds in terms of escape of mass in Section \ref{subsec:Aggarwal}, and results on Lebesgue measure and the structure of the approximation sets in Section \ref{subsec:transference}. Section~\ref{sec:CD} formulates the general upper and lower estimates. Section~\ref{sec:coverings} proves these estimates using lattice counting, covering arguments, and a Cantor construction. Section~\ref{sec:pressure} derives the uniform dimension estimates and proves that passing to the singular set causes no dimension drop, completing the proof of Theorem~\ref{thm:main}. Section~\ref{sec:near-power-proofs} proves the integral formula for powers and functions close to them, along with the dimension formulae for sets with prescribed uniform exponents. Section~\ref{sec:remaining-theorems} proves the zero--full dichotomy, the endpoint set identities, and the continuity statements. Section~\ref{subsec:one-dimensional} develops the continued-fraction description and proves the one-dimensional results. Section~\ref{sec:Aggarwal_result} proves the escape-of-mass estimates. Section \ref{sec:transference} proves the transference applications. Finally, Section~\ref{sec:conclusion} suggests some further research directions.

\medskip
\noindent\textbf{Conventions.}
For nonnegative quantities $X$ and $Y$, we write $X\ll Y$ if
$X\le CY$ for some constant $C>0$, write $X\gg Y$ if $Y\ll X$,
and write $X\asymp Y$ if both $X\ll Y$ and $Y\ll X$.
For an arbitrary real quantity $X$ and a positive quantity $Y$, we
write $X=O(Y)$ if $|X|\ll Y$, and write $X=o(Y)$ if $X/Y\to0$ in the
limit under consideration.
Since the dimensions $m,n$, the matrix $A$, and the function $\psi$
are fixed throughout, dependence on them is suppressed in
$\ll,\gg,\asymp$, $O(\cdot)$, and $o(\cdot)$.
Dependence on any additional parameters is recorded by subscripts.

\medskip
\noindent\textbf{Use of artificial intelligence.} OpenAI Codex was used as an auxiliary tool in checking selected calculations and arguments and in refining the exposition. The authors independently verified every mathematical argument, determined the final content and wording, and take full responsibility for the manuscript.

\medskip
\noindent\textbf{Acknowledgments.}
The authors are grateful to Dmitry Kleinbock, Uri Shapira, and Nikolay
Moshchevitin for helpful discussions.

\section{Applications}\label{sec:applications}

\subsection{Results in the one-dimensional case}\label{sec:one-dimensional}

Let $m=n=1$. Due to presence of continued fractions, the question about Hausdorff dimensions of sets $\Dhat_{\alpha}(\psi), \ \alpha \in \R$, was resolved for power functions $\psi_\tau$. 

Namely, let $\tau>0$. 
Let
$\alpha\in\R\setminus\Q$, and let $p_k/q_k$ be its convergents. We define the \emph{ordinary Diophantine exponent} of $\alpha$ via
\begin{equation*}\label{eq:5.89}
 \omega(\alpha):=
 \limsup_{k\to\infty}\frac{\log q_{k+1}}{\log q_k}
 =
 \limsup_{k\to\infty}
 \frac{\log|q_k\alpha|_{\Z}^{-1}}{\log q_k},
\end{equation*}
where $p_k/q_k$ are the convergents of $\alpha$; the notation $| \cdot |_{\Z}$ stands for the distance to the nearest integer and was introduced in \eqref{eq:distance_to_the_nearest_integer}. 
We will also replace the sets $\Dhat_{\alpha}(\psi)$ by the one-sided version
\begin{equation}\label{eq:D_alpha_N}
\Dhat_{\alpha, \N}(\psi):=\left\{b\in\T:
\begin{array}{l}
 \text{for every sufficiently large $T$, there is} \\ 1\le q\le T \ 
 \text{such that }|q\alpha-b|_{\Z}\le\psi(T)
 \end{array}
\right\}.
\end{equation}
The statement below is a slight modification of Theorem~1 and the non-critical case of Theorem~2 from \cite{KL19}.

\begin{thm}[Kim--Liao, \cite{KL19}]
\label{cor:kim-liao}
Let $\alpha \in \R$ be irrational and $c > 0$.
\begin{enumerate}[label=\textup{(\roman*)}]
\item If $0<\tau<1/\omega(\alpha)$, then
${\Dhat_{\alpha,\N}}(c\psi_\tau)=\T$.
\item If $\tau>\omega(\alpha)$, then
\[
 {\Dhat_{\alpha,\N}}(c\psi_\tau)=\{q\alpha:q \in \N \}.
\]
\item If $1/\omega(\alpha)<\tau<1$, fix $a > 0$ and let $\{n_i\}$ be the increasing
enumeration of the convergent denominators satisfying
\begin{equation*}\label{eq:5.91}
 n_i|n_i\alpha|_{\Z}^{\tau}<a.
\end{equation*}
Then
\begin{equation*}\label{eq:5.92}
 \dim_H{\Dhat_{\alpha,\N}}(c\psi_\tau)
 =
 \liminf_{i\to\infty}
 \frac{
 \displaystyle
 \log\left(
 n_i^{1+1/\tau}
 \prod_{j=1}^{i-1}
 n_j^{1/\tau}|n_j\alpha|_{\Z}
 \right)}
 {\displaystyle
 \log\left(n_i|n_i\alpha|_{\Z}^{-1}\right)}.
\end{equation*}
\item If $1<\tau<\omega(\alpha)$, fix $b > 0$ and let $\{n_i\}$ be the increasing
enumeration of the convergent denominators satisfying
\begin{equation*}\label{eq:5.93}
 n_i^\tau|n_i\alpha|_{\Z}<b.
\end{equation*}
Then
\begin{equation*}\label{eq:5.94}
 \dim_H{\Dhat_{\alpha,\N}}(c\psi_\tau)
 =
 \liminf_{i\to\infty}
 \frac{
 \displaystyle
 -\log\left(
 \prod_{j=1}^{i-1}
 n_j|n_j\alpha|_{\Z}^{1/\tau}
 \right)}
 {\displaystyle
 \log\left(n_i|n_i\alpha|_{\Z}^{-1}\right)}.
\end{equation*}
\end{enumerate}
\end{thm}

\smallskip

Specifically, one can obtain the original result of Kim-Liao by taking $c = 1$, $a = 1$ and $b = 2$. The remaining difference between Theorem \ref{cor:kim-liao} and the original result of Kim-Liao is that they used the strict inequality $|q\alpha-b|_{\Z}<\psi(T)$ in \eqref{eq:D_alpha_N}. Such a modification does not change the Hausdorff dimension of the set; see Remark \ref{rem:for-KL}.

\medskip

Using Theorem \ref{thm:main}, we will prove a generalization of Theorem \ref{cor:kim-liao} for functions $\psi$ under {lower quasimultiplicativity from Definition \ref{def:regular}}.

For every sufficiently large $k$, let $Y_k$ be the unique solution of
\begin{equation*}\label{eq:5.58}
 \frac{Y_k}{\psi(Y_k)}=\frac{q_k}{|q_k\alpha|_{\Z}}.
\end{equation*}

We prove the following theorem:

\begin{thm}
\label{thm:one-dimensional}
Suppose that $\psi$ is {continuous, strictly decreasing to zero,
and lower quasimultiplicative}.  Then the following statements hold.
\begin{enumerate}[label=\textup{(\roman*)}]
\item {Suppose that
\begin{equation}\label{eq:one-dimensional-slow-hypothesis}
 \limsup_{T\to\infty}\frac{\log\psi(T)^{-1}}{\log T}<1,
\end{equation}
}and let $\{k_i\}$ be the increasing
enumeration of all sufficiently large indices satisfying
\begin{equation}\label{eq:5.59}
 \psi\bigl(|q_{k_i}\alpha|_{\Z}^{-1}\bigr)<q_{k_i}^{-1}.
\end{equation}
If this sequence is finite, then
\begin{equation}\label{eq:5.60}
 \dim_H\Dhat_\alpha(\psi)=1.
\end{equation}
If it is infinite, then
\begin{equation}\label{eq:5.61}
 \dim_H\Dhat_\alpha(\psi)
 =
 \liminf_{i\to\infty}
 \frac{
 \displaystyle
 \log\left(
 q_{k_i}\prod_{j=1}^{i-1}
 q_{k_j}\psi(Y_{k_j})
 \right)}
 {\log\psi(Y_{k_i})^{-1}}.
\end{equation}

\item {Suppose that
\begin{equation}\label{eq:one-dimensional-fast-hypothesis}
 \liminf_{T\to\infty}\frac{\log\psi(T)^{-1}}{\log T}>1,
\end{equation}
}and let $\{k_i\}$ be the increasing
enumeration of all sufficiently large indices satisfying
\begin{equation}\label{eq:5.62}
 \psi(q_{k_i})>|q_{k_i}\alpha|_{\Z}.
\end{equation}
If this sequence is finite, then
\begin{equation}\label{eq:5.63}
 \dim_H\Dhat_\alpha(\psi)=0.
\end{equation}
If it is infinite, then
\begin{equation}\label{eq:5.64}
 \dim_H\Dhat_\alpha(\psi)
 =
 \liminf_{i\to\infty}
 \frac{
 \displaystyle
 \log\left(
 \prod_{j=1}^{i-1}
 \frac{\psi(Y_{k_j})}{|q_{k_j}\alpha|_{\Z}}
 \right)}
 {\log\psi(Y_{k_i})^{-1}}.
\end{equation}
\end{enumerate}
In either regime, for every fixed $c>0$,
\begin{equation}\label{eq:5.65}
 \dim_H\Dhat_\alpha(c\psi)
 =
 \dim_H\Dhat_\alpha(\widetilde\psi_c)
 =
 \dim_H\Dhat_\alpha(\psi).
\end{equation}
\end{thm}
\smallskip

\begin{rem}\label{rem:for-KL}
Due to the last assertion of Remark \ref{rem:lower-Q}, the above theorem also holds if one replaces the set $\Dhat_\alpha(\psi)$ by $\Dhat_{\alpha, \N}(\psi)$ defined in \eqref{eq:D_alpha_N}.
Due to the $c$-invariance \eqref{eq:5.65} and trivial inclusion property 
$$
\Dhat_{\alpha, \N}(\widetilde{\psi}_c) \subseteq \Dhat_{\alpha, \N}(\psi)
$$
which holds for $c < 1$, one can also use strict inequalities in the definition \eqref{eq:D_alpha_N}.
\end{rem}

\subsection{Upper dimension estimates via escape of mass}
\label{subsec:Aggarwal}

The \emph{lower escape of mass} of the trajectory $\{ a_tu_A\Z^d \}_{t \ge 0}$ is defined by
\begin{equation*}
 \underline{\Emass}(A):=
 \lim_{\delta\searrow0}\liminf_{R\to\infty}
 \frac1R
 \mathcal L^1\left(
 \left\{t\in[0,R]:
 \lambda_1(t)\le\delta
 \right\}\right).
\end{equation*}

Aggarwal  \cite[Corollary~2.10]{Agg26} proved that the Hausdorff dimension of $\Sing_A(\psi_{\frac nm})$ has an upper bound $m\underline{\Emass}(A)$. We improve this result by the factor $(d-1)/d$.
\begin{prop}
\label{prop:aggarwal-special}
For every $A\in\Mat_{m,n}$,
\begin{equation}
 \dim_H\Sing_A(\psi_{\frac nm})
 \le
 m{\frac{d-1}{d}}\,\underline{\Emass}(A).
\end{equation}
\end{prop}

Since the lower escape of mass takes values between zero and one, an immediate corollary of our result is that the dimension of $\Sing_A(\psi_{\frac nm})$ is smaller than full at least by the factor of $1 - \frac{1}{m+n}$. This gives one more parallel with the classical homogeneous case: it is known \cite[Theorem~3.1]{DFSU24} that, if $(m,n) \ne (1,1)$, the Hausdorff dimension of singular $m \times n$ matrices is equal to $mn \left( 1 - \frac{1}{m+n} \right)$, that is, smaller than full by exactly same factor. 

Of course, in twisted case we can only have the upper bound of this form: the actual dimension can be smaller (e.g. when the lower escape of mass is smaller than one), and the exact value of this dimension can be established, using the full minima profile, by Theorem \ref{thm:main}. We can, however, claim that Proposition \ref{prop:aggarwal-special} cannot be improved in the following sense: it is the best bound one can get from considering the lower escape of mass alone.
\begin{prop}\label{prop:aggarwal-sharpness}
For every $0\le\eta\le1$, there exists $A\in\Mat_{m,n}$ such that
\[
 \underline{\Emass}(A)=\eta
 \qquad \text{and} \qquad
 \dim_H\Sing_A(\psi_{n/m})=\eta\,\frac{m(d-1)}d.
\]
\end{prop}

We construct examples for Proposition \ref{prop:aggarwal-sharpness} using parametric geometry of numbers. In dimension one, it is also not hard to construct such a number explicitly via the continued fraction expansion.

\begin{ex}\label{rem:aggarwal-sharp-example}
Let $\alpha=[0;2,3,4,\ldots]$. Then
\[
 \underline{\Emass}(\alpha)=1,
 \qquad
 \dim_H\Sing_\alpha(\psi_1)
 =\dim_H\Dhat_\alpha(c\psi_1)=\frac12
 \quad\text{for every }c>0.
\]
\end{ex}

For $\delta>0$, introduce the two-parameter version of the quantity $r_{A,\tau}$ (defined by \eqref{eq:power-rank}):
\begin{equation}\label{eq:power-rank-cutoff}
 r_{A,\tau}(T;\delta):=
 \#\left\{1\le i\le d:
 \lambda_i\!\left(\frac{mn(1+\tau)}d\log T\right)
 <\delta T^{(n-m\tau)/d}\right\}.
\end{equation}
The one-parameter function defined in \eqref{eq:power-rank} is the
case $\delta=1$, that is, $r_{A,\tau}(T)=r_{A,\tau}(T;1)$.

We deduce Proposition~\ref{prop:aggarwal-special} by comparing escape
of mass with the right-hand side of the following more general result.

\begin{prop}
\label{prop:aggarwal-general}
For every $A\in\Mat_{m,n}$ and every $\tau>0$,
\begin{equation}\label{eq:5.31}
 \dim_H\Sing_A(\psi_\tau)
 \le
 \lim_{\delta\searrow0}\liminf_{T\to\infty}
 \frac{\displaystyle\int_1^T r_{A,\tau}(S;\delta)\,\frac{dS}{S}
       -\log\mathcal N_A(T^{-\tau},T)}{(1+\tau)\log T}.
\end{equation}
\end{prop}

\subsection{On Lebesgue measure and structure of the sets
\texorpdfstring{$\Dhat_A(\psi)$}{D(A, psi)}}
\label{subsec:transference}

We collect a measure consequence of Theorem \ref{thm:zero-full}, a
consequence of inhomogeneous transference, and an elementary criterion
for trivial shifts.

Let $\varphi:\R_+\to\R_+$ be continuous and strictly decreasing, and
suppose that $\varphi(S)\to0$ as $S\to\infty$. As with $\psi$, we make
the harmless normalization $\varphi(S)\to\infty$ as $S\to0$. We denote
the transposed matrix to $A$ by $A^{\top}$. Say that $A^\top$ \emph{is
$\varphi$-approximable} if
\begin{equation*}\label{eq:5.1}
 |A^\top\p-\q|\le\varphi(S),\qquad |\p|\le S
\end{equation*}
has a solution
$(\p,\q)\in(\Z^m\setminus\{\mathbf0\})\times\Z^n$ for an unbounded set
of parameters $S$.  The functions $\varphi$ and $\psi$ are called
\emph{dual} if
\begin{equation}\label{eq:5.2}
 \psi(T)=\frac1{\varphi^{-1}(1/T)},
 \qquad\text{or, equivalently,}\qquad
 \varphi(S)=\frac1{\psi^{-1}(1/S)}.
\end{equation}
Recall from
\eqref{eq:rescaled-function} that $\widetilde\psi_\delta(T)=\delta\psi(T/\delta)$.

Similarly, say that $A$ is \emph{$\varphi$-approximable} if
\begin{equation*}\label{eq:5.50}
 |A\q-\p|\le\varphi(S),\qquad |\q|\le S
\end{equation*}
has a solution $(\p,\q)\in\Z^m\times(\Z^n\setminus\{\mathbf0\})$ for
an unbounded set of parameters $S$.

Put
\begin{equation*}\label{eq:transference-constant}
 \kappa_d:=2^{1-d}(d!)^2.
\end{equation*}

The following proposition corresponds to the full-space case of
\cite[Theorem~3.1]{MN25}, in the present normalization.

\begin{prop}\label{prop:transference-null}
Let $\psi$ and $\varphi$ be dual. If $A^\top$ is
$\varphi$-approximable, then, for every $0<\varepsilon<1/(2d)$,
\begin{equation}\label{eq:transference-null}
 \mathcal L^m\bigl(\Dhat_A(\widetilde\psi_\varepsilon)\bigr)=0.
\end{equation}
\end{prop}

The following proposition is a modification (up to absolute constants) Theorem~A of \cite{MN25}, written in the
present normalization. The additional assertion is an extension in
which the denominator vectors are required to be positive.

\begin{prop}\label{prop:transference-full}
Suppose that $\psi$ and $\varphi$ are dual and that $A^\top$ is not
$\varphi$-approximable. Then
\begin{equation}\label{eq:transference-full}
 \Dhat_A(\widetilde\psi_{\kappa_d})=\T^m.
\end{equation}
Moreover, under the same assumption,
\[
 \Dhat_{A,\N^n}(\widetilde\psi_{3\kappa_d})=\T^m,
\]
where $\Dhat_{A,\mathcal Q}(\psi)$ is defined in
Remark~\ref{rem:lower-Q}.
\end{prop}

The following proposition is a structural strengthening, in the
present normalization, of \cite[Theorem~3.8\textup{(a)}]{MN25}. It
identifies exactly the trivial shifts and allows an arbitrary prescribed
set of denominator vectors.

\begin{prop}\label{prop:trivial-shifts}
If $A$ is not $\varphi$-approximable, then, for every $0<c<1/2$,
\begin{equation}\label{eq:5.51}
 \Dhat_A(\widetilde\varphi_c)=\Gamma_A.
\end{equation}
Moreover, under the same assumption, for every
$\mathcal Q\subseteq\Z^n$ one has
\[
 \Dhat_{A,\mathcal Q}(\widetilde\varphi_c)
 =A\mathcal Q\bmod\Z^m.
\]
\end{prop}

\section*{Glossary of Notation}
\label{sec:glossary}

Principal symbols and terminology are grouped by topic below.
References lead to the full definitions.

\begingroup
\small
\setlength{\tabcolsep}{5pt}
\setlength{\LTpre}{7pt}
\setlength{\LTpost}{8pt}
\renewcommand{\arraystretch}{1.05}
\newcommand{\glossarygroup}[1]{%
  \addlinespace[7pt]
  \multicolumn{3}{@{}l}{\scshape #1}\\*[4pt]}
\begin{longtable}{@{}>{\raggedright\arraybackslash}p{0.25\textwidth}%
 >{\raggedright\arraybackslash}p{\dimexpr0.55\textwidth-2\tabcolsep\relax}%
 >{\raggedright\arraybackslash}p{\dimexpr0.20\textwidth-2\tabcolsep\relax}@{}}
\toprule
\textsc{Notation} & \textsc{Meaning} & \textsc{Reference} \\
\midrule
\endfirsthead
\toprule
\textsc{Notation} & \textsc{Meaning} & \textsc{Reference} \\
\midrule
\endhead
\bottomrule
\endfoot
\glossarygroup{Approximation sets and terminology}
$d$ & $d=m+n$ & \S\ref{sec:previous} \\[3pt]
$\Mat_{m,n}$ & Real $m\times n$ matrices &
\S\ref{sec:previous} \\[3pt]
$\T^m$ & $\T^m=\R^m/\Z^m$ & \S\ref{sec:previous} \\[3pt]
$\langle\cdot\rangle$ & Residue class modulo $\Z^m$ &
\S\ref{sec:previous} \\[3pt]
$|\cdot|$ & Supremum norm: $|\bx|=\max_i|x_i|$ &
\S\ref{sec:previous} \\[3pt]
$|\cdot|_{\Z^m}$ & $|\bx|_{\Z^m}=\min_{\p\in\Z^m}|\bx-\p|$ &
\eqref{eq:distance_to_the_nearest_integer} \\[3pt]
$B(\bx,r)$ & $B(\bx,r)=\{\by:|\by-\bx|\le r\}$ &
\S \ref{sec:upperbd}  \\[3pt]
$\lambda B$ & $\lambda B(\bx,r)=B(\bx,\lambda r)$, $\lambda>0$ &
\S \ref{sec:upperbd} \\[3pt]
$\ll,\gg,\asymp,O,o$ & Comparison and asymptotic notation. &
\S \ref{sec:introduction}, Conventions \\[3pt]
$\widetilde\psi_c$ & $\widetilde{\psi}_c(T)=c\psi(T/c)$, $c>0$ &
\eqref{eq:rescaled-function} \\[3pt]
$\psi_\tau$ & $\psi_\tau(T)=T^{-\tau}$, $\tau>0$ &
\S\ref{sec:previous} \\[3pt]
$\Dhat_A(\psi)$ & Twisted $\psi$-Dirichlet vectors &
\eqref{eq:1.1} \\[3pt]
$\Dhat_{A,\mathcal Q}(\psi)$ & Twisted $\psi$-Dirichlet vectors with
denominators restricted to $\mathcal Q\subseteq\Z^n$ &
Rem.~\ref{rem:lower-Q} \\[3pt]
$\Sing_A(\psi)$, $U_A(\psi)$ & Intersection and union, respectively, of
$\Dhat_A(\widetilde\psi_c)$ over $c>0$ &
\eqref{eq:1.14}, Lem.~\ref{lem:rescaling-comparison} \\[3pt]
Lower quasi-\newline multiplicativity & Decreasing $\psi$ for which
there exist $\tau>0$ and $R_0>1$ such that
$R^{-\tau}\psi(T)\le\psi(RT)$ for all sufficiently large $T$ and every
$R\ge R_0$. & Def.~\ref{def:regular} \\[3pt]
$\mathcal E_{m,n}$ & Finite set of possible exceptional power exponents &
\eqref{eq:power-exceptional-set} \\[3pt]
$\widehat\omega(A,\bb)$ & Uniform Diophantine exponent &
\eqref{eq:uniform-shift-exponent} \\[3pt]
\glossarygroup{Lattice counts and dimension formulae}
$u_A$ & $u_A=\begin{pmatrix}I_m&A\\0&I_n\end{pmatrix}$ &
\eqref{eq:1.6} \\[5pt]
$a_t$ & $a_t=\operatorname{diag}(\e^{t/m}I_m,\e^{-t/n}I_n)$ &
\eqref{eq:1.6} \\[3pt]
$\Lambda_A$ & $\Lambda_A=u_A\Z^d$ &
\S\ref{sec:near-power-proofs} \\[3pt]
$\lambda_i(E,\Lambda)$ & $i$-th successive minimum of $\Lambda$ with respect to $E$ &
\S\ref{ssec:counting} \\[3pt]
$\lambda_i(\Lambda)$ & $\lambda_i(\Lambda)=\lambda_i([-1,1]^d,\Lambda)$ &
\S\ref{subsec:main-theorem} \\[3pt]
$\lambda_i(t)$ & Successive minima $\lambda_i(a_tu_A\Z^d)$ &
\S\ref{subsec:main-theorem} \\[3pt]
${t_{R,Q}}$ & $t_{R,Q}=\frac{mn}{d}\log(Q/R)$ &
\eqref{eq:2.1} \\[3pt]
$N_A(R,Q)$ & Number of integer pairs with
$|A\q-\p|\le R$, $|\q|\le Q$ & \S\ref{ssec:counting} \\[3pt]
$\mathcal N_A(R,Q)$ & $\displaystyle\prod_{i=1}^d
\max\left\{\frac{Q^{n/d}R^{m/d}}{\lambda_i(t_{R,Q})},1\right\}$ &
\eqref{eq:2.1}, Lem.~\ref{lem:profile-count} \\[5pt]
$T_j$ & $T_j=\psi^{-1}(\e^{-jL})$ &
\eqref{eq:1.11} \\[3pt]
$\mathcal P_{A,\psi}(L,c)$ & $\displaystyle
 \liminf_{N\to\infty}\frac1{NL}\sum_{j=0}^{N-1}\log
 \frac{\mathcal N_A(c\psi(T_j),T_{j+1})}
      {\mathcal N_A(c\psi(T_{j+1}),T_{j+1})}$ &
\eqref{eq:4.19} \\[5pt]
$P_-(A,\psi)$ & $\displaystyle\lim_{L\to\infty}\inf_{c>0}\mathcal P_{A,\psi}(L,c)$ &
Thm.~\ref{thm:main}(i) \\[5pt]
$P_+(A,\psi)$ & $\displaystyle\lim_{L\to\infty}\sup_{c>0}\mathcal P_{A,\psi}(L,c)$ &
Thm.~\ref{thm:main}(i) \\[5pt]
$r_{A,\tau}(T;\delta)$ & $\displaystyle
\#\left\{\begin{gathered}
1\le i\le d:\\[-2pt]
\lambda_i(t_{T^{-\tau},T})<\delta T^{(n-m\tau)/d}
\end{gathered}\right\}$ &
\eqref{eq:power-rank-cutoff} \\[5pt]
$r_{A,\tau}(T)$ & $r_{A,\tau}(T)=r_{A,\tau}(T;1)$ &
\eqref{eq:power-rank} \\[3pt]
$\underline{\Emass}(A)$ & Lower escape of mass of $a_tu_A\Z^d$ &
\S\ref{subsec:Aggarwal} \\[3pt]
$a_{x,y}$ & $a_{x,y}=\operatorname{diag}(\e^yI_m,\e^{-x}I_n)$ &
\S\ref{sec:near-power-proofs} \\[3pt]
$\omega_\Gamma$ & $\omega_\Gamma=\bv_1\wedge\cdots\wedge\bv_r$ for a basis
$(\bv_1,\ldots,\bv_r)$ of $\Gamma$, with supremum norm in the standard
exterior basis. & \S\ref{sec:near-power-proofs} \\[3pt]
$F(x,y)$ & $\displaystyle F(x,y)=\max_\Gamma\bigl(-\log|a_{x,y}\omega_\Gamma|\bigr)$,
where $\Gamma$ ranges over primitive sublattices of $\Lambda_A$. &
\S\ref{sec:near-power-proofs}, Lem.~\ref{lem:power-profile} \\[3pt]
$g_i(x)$ & $g_i(x)=\log\lambda_i(a_{x,\tau x}\Lambda_A)$ &
\eqref{eq:power-normalized-minima} \\[3pt]
$G_r(x)$ & $\displaystyle G_r(x)=\sum_{i=1}^r g_i(x)$ &
\eqref{eq:power-normalized-minima} \\[5pt]
$\mathbf f=(f_1,\ldots,f_d)$ & An $(m,n)$-template. &
\S\ref{subsec:templates}, Def.~\ref{def:template} \\[3pt]
\glossarygroup{Measures and transference}
$\Gamma_A$ and $H_A$ & $\Gamma_A=\langle A\Z^n\rangle$ and
$H_A=\overline{\langle A\Z^n\rangle}$ &
\S\ref{subsec:zero-full} \\[3pt]
$\mathcal L_A$ & Normalized Haar measure on $H_A$ &
\S\ref{subsec:zero-full} \\[3pt]
$\mathcal L^m$ & Lebesgue measure on $\T^m$ &
\S\ref{subsec:zero-full} \\[3pt]
$\mathcal H^f$ & $f$-Hausdorff measure &
\S\ref{subsec:zero-full} \\[3pt]
$\mathcal H^s$ & $s$-Hausdorff measure; corresponds to $f(r)=r^s$ &
\S\ref{subsec:zero-full} \\[3pt]
$\dim_H$ & Hausdorff dimension &
\S\ref{subsec:zero-full} \\[3pt]
Locally full measure & $\mathcal H^f(E\cap B)=\mathcal H^f(B)$ for every\newline
nonempty relatively open ball $B\subseteq H_A$ &
\S\ref{subsec:zero-full} \\[3pt]
$\kappa_d$ & $\kappa_d=2^{1-d}(d!)^2$ &
\S\ref{subsec:transference} \\[3pt]
\glossarygroup{One-dimensional notation}
\multicolumn{3}{@{}l}{\textit{Here $m=n=1$ and $A=\alpha\notin\Q$.}}\\*[5pt]
$p_k/q_k$ & Continued-fraction convergents &
\S\ref{sec:one-dimensional} \\[3pt]
$\omega(\alpha)$ & Ordinary Diophantine exponent &
\S\ref{sec:one-dimensional} \\[3pt]
$Y_k$ & $\displaystyle\frac{Y_k}{\psi(Y_k)}=\frac{q_k}{|q_k\alpha|_{\Z}}$ &
\S\ref{sec:one-dimensional} \\[5pt]
$h(x)$ & $h(x)=-\log\psi(\e^x)$ &
\S\ref{subsec:one-dimensional} \\[3pt]
$a_k$, $b_k$ & $a_k=\log q_k$, $b_k=\log|q_k\alpha|_{\Z}^{-1}$ &
\eqref{eq:a_k_b_k} \\[3pt]
$F_\alpha(x,y)$ & $\displaystyle\max\left\{
\begin{gathered}
0,\ x-y,\\[-2pt]
\displaystyle\sup_{k\ge0}\min\{x-a_k,b_k-y\}
\end{gathered}\right\}$ &
\S\ref{subsec:one-dimensional}, Lem.~\ref{lem:cf-count} \\[5pt]
$\sigma(x)$ & $\sigma(x)=x+h(x)$ & \eqref{eq:5.69} \\[3pt]
$\xi_k(w)$ & $\xi_k(w)=\sigma^{-1}(a_k+b_k+w)$ &
\eqref{eq:shifted-xi} \\[3pt]
$I_k(w)$ & $I_k(w)=\bigl(h^{-1}(a_k+w),\xi_k(w)\bigr)$ &
Lem.~\ref{lem:one-dimensional-branches} \\[3pt]
$J_k(w)$ & $J_k(w)=\bigl(\xi_k(w),h^{-1}(b_k+w)\bigr)$ &
Lem.~\ref{lem:one-dimensional-branches} \\[3pt]
\end{longtable}
\endgroup

\part{Proofs}

\section{The upper and lower estimates}
\label{sec:CD}

In this section we formulate some lower and upper estimates which will be used in the
proof of Theorem \ref{thm:main}. We also use the upper estimate directly in the proof of Lemma \ref{lem:power-upper-all}. 

Throughout this and the next section, let
$\psi:\R_+\to\R_+$ be strictly decreasing and tend to zero.
Let $\{T_j\}_{j\ge1}$ be a strictly increasing sequence of positive real
numbers such that $T_1\ge1$, $\psi(1) < \frac{1}{12}$, and
$T_j\to\infty$.

We use the quantity $\mathcal N_A$ defined in \eqref{eq:2.1}.

\begin{prop}[Upper {Hausdorff dimension} estimate]
\label{prop:upper}
For every sequence $\{T_j\}_{j\ge1}$ as above,
\begin{equation}\label{eq:2.4}
{\dim_H}\bigl(\Dhat_A(\psi)\bigr)
{\le}
\liminf_{\ell\to\infty}
{\frac{\displaystyle
(\ell-1)\log C_d+\sum_{j=1}^{\ell-1}\log
\frac{\mathcal N_A(\psi(T_j),T_{j+1})}
{\mathcal N_A(\psi(T_{j+1}),T_{j+1})}}
{-\log\psi(T_\ell)}},
\end{equation}
where $C_d\ge1$ depends only on $d$.
\end{prop}

\begin{prop}[Lower {Hausdorff dimension} estimate]
\label{prop:lower}
Suppose that 
\begin{equation}\label{eq:2.5}
 \psi(T_j)\le\rho\psi(T_{j+1}),\qquad j\ge1,
\end{equation}
for some $\rho>1$.  {Then}
\begin{equation}\label{eq:2.6}
\begin{split}
{\dim_H}\bigl(\Dhat_A(\rho\psi)\bigr)
{\ge}{}&
\liminf_{\ell\to\infty}{\frac{1}{-\log\psi(T_\ell)}}\\
&{\quad\times\sum_{j=1}^{\ell-1}\log}
\max\Biggl\{1,c_d
\frac{\mathcal N_A(\psi(T_j)-\psi(T_{j+1}),
 T_{j+1}-T_j)}
{\mathcal N_A(\psi(T_{j+1}),T_{j+1})}
\Biggr\},
\end{split}
\end{equation}
where $c_d>0$ is a constant depending only on $d$.

If, in addition, there exists $\eta>1$ such that
\begin{equation}\label{eq:2.7}
 T_{j+1}\ge\eta T_j\ \ \ \text{and} \ \ 
 \psi(T_j)\ge\eta\psi(T_{j+1})
 \ \ \text{for any} \ \ j\ge1,
\end{equation}
then
\begin{equation}\label{eq:2.8}
\begin{split}
{\dim_H}\bigl(\Dhat_A(\rho\psi)\bigr)
{\ge}{}&
\liminf_{\ell\to\infty}{\frac{1}{-\log\psi(T_\ell)}}\\
&{\quad\times\sum_{j=1}^{\ell-1}\log}
\max\Biggl\{1,c_d\left(\frac{\eta-1}{\eta}\right)^d
\frac{\mathcal N_A(\psi(T_j),T_{j+1})}
{\mathcal N_A(\psi(T_{j+1}),T_{j+1})}
\Biggr\}.
\end{split}
\end{equation}
\end{prop}

\begin{rem}\label{rem:lower-Q}
For a nonempty set $\mathcal Q\subseteq\Z^n$, let
$\Dhat_{A,\mathcal Q}(\psi)$ be defined by \eqref{eq:1.1}, with the
additional requirement $\q\in\mathcal Q$ in \eqref{eq:1.2}. Suppose
that $\mathcal Q$ satisfies
\begin{equation*}\tag{S}\label{eq:central-extension}
 \forall \ \text{$\q\in\mathcal Q$ and $\bx\in\Z^n$, at least one of }
 \q+\bx\text{ and }\q-\bx\text{ belongs to }\mathcal Q.
\end{equation*}
Then Proposition \ref{prop:lower} holds with
$\Dhat_A(\rho\psi)$ replaced by
$\Dhat_{A,\mathcal Q}(\rho\psi)$, after possibly replacing $c_d$ by
another positive constant depending only on $d$.

As a consequence, under the hypotheses of Theorem~\ref{thm:main}, put
\begin{equation*}\label{eq:restricted-classical-endpoints}
 \Sing_{A,\mathcal Q}(\psi):=
 \bigcap_{c>0}\Dhat_{A,\mathcal Q}(c\psi),
 \qquad
  U_{A,\mathcal Q}(\psi):=
 \bigcup_{c>0}\Dhat_{A,\mathcal Q}(c\psi).
\end{equation*}
Then Theorem~\ref{thm:main} remains valid with
$\Dhat_A(\psi)$, $\Sing_A(\psi)$ and $ U_A(\psi)$
replaced by their $\mathcal Q$-restricted counterparts.
\end{rem}

We prove Propositions \ref{prop:upper} and \ref{prop:lower} and the
first assertion of Remark \ref{rem:lower-Q} in Section
\ref{sec:coverings}. The last assertion of the remark is proved in
Subsection~\ref{subs:proof-of-main}.

\section{Lattice counting and the covering constructions}
\label{sec:coverings}

\subsection{The counting function}\label{ssec:counting}

For $Q,R>0$, let
\begin{equation*}\label{eq:3.1}
 N_A(R,Q):=
 \#\left\{(\p,\q)\in\Z^m\times\Z^n:
 |A\q-\p|\le R,\ |\q|\le Q\right\}.
\end{equation*}
Let $\cB:=[-1,1]^d$.
For a convex, centrally symmetric set $E \subseteq \R^d$ and lattice $\Lambda$ in $\R^d$, the notation $\lambda_i(E, \Lambda)$ stands for the $i$-th successive minimum of $E$ on $\Lambda$; that is, $\lambda_i(E, \Lambda)$ is the infimum of $\lambda > 0$ such that the set $\lambda E \cap \Lambda$ contains $i$ linearly independent vectors. In particular, previously defined minima $\lambda_i(\Lambda)$ satisfy $\lambda_i(\Lambda) = \lambda_i(\cB, \Lambda)$.

The following standard estimate is proved in
\cite[Lemma~5.1]{KN26}.

\begin{lem}
\label{lem:lattice-count}
Let $\Lambda$ be a full-rank lattice in $\R^d$.  Then, for every $r>0$,
\begin{equation*}\label{eq:3.2}
 \#\bigl(\Lambda\cap r\cB\bigr)
 \asymp_d
 \prod_{i=1}^d\max\left\{\frac r{\lambda_i(\Lambda)},1\right\}.
\end{equation*}
\end{lem}

\begin{lem}
\label{lem:profile-count}
One has
\begin{equation}\label{eq:3.3}
 N_A(R,Q)\asymp_d \mathcal N_A(R,Q)
\end{equation}
uniformly in $Q,R>0$.
\end{lem}

\begin{proof}
Note that by definition
$$
 N_A(R,Q)
 =
 \#\left(
 u_A\Z^d
 \cap [-R,R]^m\times[-Q,Q]^n
 \right).
$$
The diagonal map 
$$
a_{{t_{R,Q}}} = \begin{pmatrix}
\left( Q/R \right)^{\frac{n}{d}}I_m&0\\
 0& \left( Q/R \right)^{-\frac{m}{d}} I_n
 \end{pmatrix}
$$ takes the rectangle
\[
 [-R,R]^m\times[-Q,Q]^n
\]
to the cube of radius $Q^{n/d}R^{m/d}$.  Therefore
\[
 N_A(R,Q)
 =
 \#\left(
 a_{{t_{R,Q}}}u_A\Z^d
 \cap Q^{n/d}R^{m/d}\cB
 \right),
\]
and %
Lemma \ref{lem:lattice-count} yields that
\begin{equation}
    N_A(R,Q) \asymp_d
 \prod_{i=1}^d\max\left\{\frac {Q^{n/d}R^{m/d}}{\lambda_i({t_{R,Q}})},1\right\} = \mathcal N_A(R, Q),
\end{equation}
proving \eqref{eq:3.3}.
\end{proof}

{
\begin{lem}\label{lem:profile-properties}
The function $\mathcal N_A(R,Q)$ is nondecreasing in each argument. Moreover, for
every $R,Q,\beta>0$,
\begin{equation}\label{eq:profile-scaling}
 \min\{1,\beta^d\}\mathcal N_A(R,Q)
 \le \mathcal N_A(\beta R,\beta Q)
 \le\max\{1,\beta^d\}\mathcal N_A(R,Q).
\end{equation}
\end{lem}

\begin{proof}
It follows from \eqref{eq:2.1} that
\begin{equation}\label{eq:rectangular-profile}
 \mathcal N_A(R,Q)
 =\prod_{i=1}^d\max\left\{
 \lambda_i\left([-R,R]^m\times[-Q,Q]^n,u_A\Z^d\right)^{-1},1
 \right\}.
\end{equation}
Monotonicity follows since
enlarging a convex body decreases each of its successive minima.
Multiplying both arguments by $\beta$ dilates the rectangle by $\beta$
and divides its minima by $\beta$. Thus \eqref{eq:profile-scaling} follows
from the inequalities
\[
 \min\{1,\beta\}\max\{x,1\}
 \le\max\{\beta x,1\}
 \le\max\{1,\beta\}\max\{x,1\},\qquad x>0.
\]
\end{proof}

\begin{lem}\label{lem:translated-count}
Let $(\bx,\q_0)\in u_A\Z^d$, where $\bx\in\R^m, \ \q_0 \in \Z^n$, and let $Q,R>0$.
Then
\begin{equation}\label{eq:lattice-translation-count}
 \#\left(u_A\Z^d\cap
 \bigl((\bx+[-R,R]^m)\times(\q_0+[-Q,Q]^n)\bigr)\right)
 =N_A(R,Q).
\end{equation}
If $|\q_0|\le Q$, then
\begin{equation}\label{eq:translated-count-window}
\begin{split}
 N_A(R,Q)
 &\le\#\left(u_A\Z^d\cap
 \bigl((\bx+[-R,R]^m)\times[-2Q,2Q]^n\bigr)\right)\\
 &\le N_A(R,3Q).
\end{split}
\end{equation}
\end{lem}

\begin{proof}
Translation by $(\bx,\q_0)$ preserves $u_A\Z^d$ and gives
\eqref{eq:lattice-translation-count}. If $|\q_0|\le Q$, then
\[
 \q_0+[-Q,Q]^n\subseteq[-2Q,2Q]^n
 \subseteq\q_0+[-3Q,3Q]^n.
\]
Applying \eqref{eq:lattice-translation-count} to these inclusions proves
\eqref{eq:translated-count-window}.
\end{proof}
}

\subsection{The upper construction: proof of Proposition \ref{prop:upper}.}\label{sec:upperbd}

For $N\ge1$, let $\Dhat_{A,N}(\psi)$ be the set of $\bb\in\T^m$ for
which \eqref{eq:1.2} is solvable for every $T\ge N$.  Then
\begin{equation}\label{eq:3.7}
\Dhat_A(\psi)=\bigcup_{N\ge1}\Dhat_{A,N}(\psi).
\end{equation}
Fix the sequence $\{T_j\}$ from Section \ref{sec:CD} and an initial
index $k$.  Here and below, $B(\x,r)$ denotes the closed ball of radius
$r$ centered at $\x$; we will also define multiplication of balls by a scalar $\lambda > 0$ via $\lambda B(\x,r): = B(\x,\lambda r)$.  

We will inductively construct a sequence of collections $\{ \cG_{T_j}^{(k)} \}_{j \geq k}$ of balls which yield a cover for the set $\Dhat_{A, T_k}(\psi)$. First we define the collection
\begin{equation*}\label{eq:3.8}
 \cP_{T_k}^{(k)}
 :=
 \left\{
 B(\langle A\q\rangle,\psi(T_k)):
 \q\in\Z^n,\ |\q|\le T_k
 \right\},
\end{equation*} 
where $\langle\cdot\rangle$ denotes reduction modulo $\Z^m$.  Choose a
maximal pairwise disjoint subcollection (also called Vitali subcollection) $\cG_{T_k}^{(k)}\subset\cP_{T_k}^{(k)}$.  Every omitted ball meets a
selected ball of the same radius and is therefore contained in its triple.
Thus the triples of the selected balls cover the union of all balls in $\cP_{T_k}^{(k)}$.

Suppose that $j \ge k$ and $\cG_{T_j}^{(k)}$ has been defined, and let
$Q\in\cG_{T_j}^{(k)}$.  Let
\begin{equation*}\label{eq:3.9}
\begin{split}
 \cP_{T_j\to T_{j+1}}^{(k)}(Q)
 :=\{&
 B(\langle A\q\rangle,\psi(T_{j+1}))\subset6Q:\\
 &\q\in\Z^n,\ |\q|\le T_{j+1}\}.
\end{split}
\end{equation*}
Choose a Vitali subcollection
$\cG_{T_j\to T_{j+1}}^{(k)}(Q)$ of $\cP_{T_j\to T_{j+1}}^{(k)}(Q)$.  As above, triples of elements of $\cG_{T_j\to T_{j+1}}^{(k)}(Q)$ cover the union of all balls in $\cP_{T_j\to T_{j+1}}^{(k)}(Q)$. We will refer to elements of $\cG_{T_j\to T_{j+1}}^{(k)}(Q)$ as children of $Q$. Put
\begin{equation}\label{eq:3.10}
 \cG_{T_{j+1}}^{(k)}
 :=
 \bigcup_{Q\in\cG_{T_j}^{(k)}}
 \cG_{T_j\to T_{j+1}}^{(k)}(Q).
\end{equation}

\smallskip

If $\cG$ is a collection of balls and $\lambda > 0$, we will use the notation $\lambda \cG: = \{ \lambda B: B \in \cG \}$.

\begin{lem}
\label{lem:upper-cover}
For every $j\ge k$, the collection $4\cG_{T_j}^{(k)}$ covers
$\Dhat_{A,T_k}(\psi)$.
\end{lem}

\begin{proof}
Let $\bb\in\Dhat_{A,T_k}(\psi)$; for every $j\ge k$ there exists
$\q_j\in\Z^n$ such that
\begin{equation*}\label{eq:3.11}
 |\q_j|\le T_j,\qquad
 |A\q_j-\bb|_{\Z^m}\le\psi(T_j), 
\end{equation*}
equivalently, $\bb \in B(\langle A\q_j\rangle,\psi(T_j))$. We claim, inductively, that there is $B\in\cG_{T_j}^{(k)}$ for which
\begin{equation}\label{eq:3.12}
 B(\langle A\q_j\rangle,\psi(T_j))\subset4B.
\end{equation}
For $j=k$, this follows from the defining property of the initial Vitali
subcollection $\cG_{T_k}^{(k)}$.  Suppose \eqref{eq:3.12} holds at level $j$, with parent
$Q$ in place of $B$.  The two balls %
$B(\langle A\q_j\rangle,\psi(T_j))$ and $B(\langle A\q_{j+1}\rangle,\psi(T_{j+1}))$ both contain $\bb$, and $$\operatorname{diam} B(\langle A\q_{j+1}\rangle,\psi(T_{j+1})) = 2\psi(T_{j+1})\le 2\psi(T_j) = \operatorname{diam} Q.$$  Consequently,
\[
 B(\langle A\q_{j+1}\rangle,\psi(T_{j+1}))\subset6Q.
\]
The defining property of the next Vitali subcollection gives a ball $B$ at level $j+1$ for which \eqref{eq:3.12} holds.  In particular, $\bb$ belongs to a member of
$4\cG_{T_j}^{(k)}$ at every level.
\end{proof}

\begin{cor}
\label{cor:cover-dimension}
For every $k\ge1$,
\begin{equation}\label{eq:3.17}
 {\dim_H}
 \bigl(\Dhat_{A,T_k}(\psi)\bigr)
 \le
 {\liminf_{\ell\to\infty}
 \frac{\log\#\cG_{T_\ell}^{(k)}}{-\log\psi(T_\ell)}}.
\end{equation}
Consequently,
\begin{equation}\label{eq:3.18}
 {\dim_H}\bigl(\Dhat_A(\psi)\bigr)
 \le
 \sup_{k\ge1}\liminf_{\ell\to\infty}
 {\frac{\log\#\cG_{T_\ell}^{(k)}}{-\log\psi(T_\ell)}}.
\end{equation}
\end{cor}

\begin{proof}
By Lemma \ref{lem:upper-cover}, for every $\ell\ge k$ the set
$\Dhat_{A,T_k}(\psi)$ is covered by
$\#\cG_{T_\ell}^{(k)}$ balls of diameter at most $8\psi(T_\ell)$.
Therefore, 
\[
    \dim_H \Dhat_{A,T_k}(\psi) \leq \underline{\dim}_B \Dhat_{A,T_k}(\psi) 
    \leq\liminf_{\ell\to\infty}\frac{\log \#\mathcal{G}_{T_\ell}^{(k)}}{-\log \psi(T_{\ell})},
    \]
where $\underline{\dim}_B$ stands for the lower box dimension.
This proves \eqref{eq:3.17}. The sets $\Dhat_{A,N}(\psi)$ are monotonically increasing, thus by \eqref{eq:3.7} one has 
$$
\Dhat_A(\psi)=\bigcup_{k\ge1}\Dhat_{A,T_k}(\psi).
$$
Countable stability of Hausdorff dimension now gives
\eqref{eq:3.18}.
\end{proof}

\begin{lem}%
\label{lem:upper-transition}
For $j\ge k$ and $Q\in\cG_{T_j}^{(k)}$,
\begin{equation*}\label{eq:3.13}
 \#\cG_{T_j\to T_{j+1}}^{(k)}(Q)
 \le C_d
 \frac{\mathcal N_A(\psi(T_j),T_{j+1})}
 {\mathcal N_A(\psi(T_{j+1}),T_{j+1})}
\end{equation*}
for some constant $C_d$ depending only on dimension.
\end{lem}

\begin{proof}
Write
\[
 Q=B(\langle A\q_0\rangle,\psi(T_j)),
 \qquad |\q_0|\le T_j,
\]
and choose \(\p_0\in\Z^m\). If $\cG_{T_j\to T_{j+1}}^{(k)}(Q) = \varnothing$, the conclusion is immediate; from now on, assume this set to be nonempty.

Consider an arbitrary child
\[
 B=B(\langle A\q\rangle,\psi(T_{j+1}))
 \in\cG_{T_j\to T_{j+1}}^{(k)}(Q)
\]
of $Q$, where \(|\q|\le T_{j+1}\). Since $B \subseteq 6Q$ and $6 \psi(T_j) < \frac{1}{2}$, we can choose \(\p\in\Z^m\) so that
$$
B( A\q - \p,\psi(T_{j+1})) \subset B( A\q_0 - \p_0,6\psi(T_j))
$$
(here we are considering balls in $\R^m$ instead of $\T^m$). For convenience, from now on we use \(Q\) and \(B\) for the corresponding Euclidean balls $B( A\q_0 - \p_0,\psi(T_j))$ and $B( A\q - \p,\psi(T_{j+1}))$, respectively; we will refer to them as lifts of balls in $\T^m$.  Lifts of different balls in $\cG_{T_j\to T_{j+1}}^{(k)}(Q)$ are pairwise disjoint because their
projections to \(\T^m\) are pairwise disjoint.  Therefore
\begin{equation}\label{eq:3.14}
 \bigsqcup_{B\in\cG_{T_j\to T_{j+1}}^{(k)}(Q)}
 \left(
 u_A\Z^d\cap
 \bigl(B\times[-2T_{j+1},2T_{j+1}]^n\bigr)
 \right)
 \subseteq
 u_A\Z^d\cap
 \bigl(6Q\times[-2T_{j+1},2T_{j+1}]^n\bigr).
\end{equation}

We first estimate the number of lattice points on the right-hand side.
Since $|\q_0|\le T_j\le T_{j+1}$,
\eqref{eq:translated-count-window} bounds this number by
$N_A(6\psi(T_j),3T_{j+1})$.
Lemma \ref{lem:profile-count}, followed by monotonicity and
\eqref{eq:profile-scaling}, therefore gives
\begin{equation}\label{eq:3.15}
 \#\left(
 u_A\Z^d\cap
 \bigl(6Q\times[-2T_{j+1},2T_{j+1}]^n\bigr)
 \right)
 \ll_d
 \mathcal N_A(\psi(T_j),T_{j+1}).
\end{equation}

Now fix a child
\[
 B=B(A\q-\p,\psi(T_{j+1})).
\]
Since $|\q|\le T_{j+1}$, the lower bound in \eqref{eq:translated-count-window} and Lemma \ref{lem:profile-count} give
\begin{equation}\label{eq:3.16}
 \#\left(
 u_A\Z^d\cap
 \bigl(B\times[-2T_{j+1},2T_{j+1}]^n\bigr)
 \right)
 \gg_d
 \mathcal N_A(\psi(T_{j+1}),T_{j+1}).
\end{equation}

The sets on the left-hand side of \eqref{eq:3.14} are pairwise
disjoint.  Thus \eqref{eq:3.14}, \eqref{eq:3.15}, and
\eqref{eq:3.16} imply
$$
 \#\cG_{T_j\to T_{j+1}}^{(k)}(Q)\,
 \mathcal N_A(\psi(T_{j+1}),T_{j+1}) \ll_d
 \mathcal N_A(\psi(T_j),T_{j+1}),
$$
completing the proof.
\end{proof}

\begin{proof}[Proof of Proposition \ref{prop:upper}]
For $\ell\ge k$, Lemma \ref{lem:upper-transition} gives
\begin{equation}\label{eq:3.19}
 \#\cG_{T_\ell}^{(k)}
 \le
 \#\cG_{T_k}^{(k)}C_d^{\,\ell-k}
 \prod_{j=k}^{\ell-1}
 \frac{\mathcal N_A(\psi(T_j),T_{j+1})}
 {\mathcal N_A(\psi(T_{j+1}),T_{j+1})}.
\end{equation}
Note that $-\log\psi(T_\ell)\to\infty$,
while for each fixed $k$ the values $\#\cG_{T_k}^{(k)}$ and $ \prod_{j=1}^{k-1}
 \frac{\mathcal N_A(\psi(T_j),T_{j+1})}
 {\mathcal N_A(\psi(T_{j+1}),T_{j+1})}$ are fixed. Substituting \eqref{eq:3.19} into \eqref{eq:3.18} now gives
\eqref{eq:2.4}.
\end{proof}

\subsection{The lower construction: proof of Proposition \ref{prop:lower}.}\label{subs:proof-of-lower}

The collections $\cP$ and $\cG$ introduced below are different from those
in the preceding subsection, but play analogous roles; we therefore reuse
the same notation.

Fix
$k \geq 1$, pick $\q_k \in \Z^n$ with $|\q_k|\le T_k$ and define
\begin{equation*}\label{eq:3.20}
 \cG_{T_k}^{(k)}
 :=
 \{B(\langle A\q_k\rangle,\psi(T_k))\}.
\end{equation*}
Let $j \geq k$ and suppose that
$Q=B(\langle A\q_j\rangle,\psi(T_j))$ belongs to
$\cG_{T_j}^{(k)}$.  Let
\begin{equation}\label{eq:3.21}
\begin{split}
 \cP_{T_j\to T_{j+1}}^{(k)}(Q)
 :=\{&
 B(\langle A\q\rangle,\psi(T_{j+1}))\subset Q:\\
 &\q\in\Z^n,\ |\q|\le T_{j+1}\}.
\end{split}
\end{equation}
This collection is nonempty, since the ball with center
$\langle A\q_j\rangle$ and radius $\psi(T_{j+1})$ belongs to it.
Choose a maximal Vitali subcollection of \eqref{eq:3.21},
denote it by $\cG_{T_j\to T_{j+1}}^{(k)}(Q)$, and define the next
level as in \eqref{eq:3.10}.  Put
\begin{equation*}\label{eq:3.22}
 \mathbf G_j^{(k)}
 :=
\bigcup_{B\in\cG_{T_j}^{(k)}}B,
 \qquad
 \mathbf G_\infty^{(k)}
 :=
 \bigcap_{j\ge k}\mathbf G_j^{(k)}.
\end{equation*}

\noindent The sets $\mathbf G_j^{(k)}$ are compact, nonempty and nested, thus $\mathbf G_\infty^{(k)}$ is compact and nonempty.

\begin{lem}
\label{lem:limit-inclusion}
If \eqref{eq:2.5} holds, then
\begin{equation*}\label{eq:3.28}
 \mathbf G_\infty^{(k)}\subseteq\Dhat_A(\rho\psi).
\end{equation*}
\end{lem}

\begin{proof}
If $\bb\in\mathbf G_\infty^{(k)}$, there are $\q_j\in\Z^n$ such that
\[
 \bb\in B(\langle A\q_j\rangle,\psi(T_j)),
 \qquad |\q_j|\le T_j.
\]
For $T_j\le T<T_{j+1}$, condition \eqref{eq:2.5} gives
\[
 |A\q_j-\bb|_{\Z^m}
 \le\psi(T_j)
 \le\rho\psi(T_{j+1})
 \le\rho\psi(T),
 \qquad |\q_j|\le T.
\]
Thus $\bb \in \Dhat_A(\rho\psi)$.
\end{proof}

Due to Lemma \ref{lem:limit-inclusion}, it is enough to produce a lower bound for the {Hausdorff dimension} of the set $\mathbf G_\infty^{(k)}$ instead of the set $\Dhat_A(\rho \psi)$. We will find such a bound using the mass distribution principle. The lemma below is a {consequence} of \cite[p.~975, Mass Distribution Principle]{BV06}:

\begin{lem}
\label{lem:mass-distribution}
Let $(X,\dist)$ be a metric space, let $F\subseteq X$ be a Borel set, let $s>0$,
and let $\nu$ be a Borel probability measure such that
$\operatorname{supp}\nu\subseteq F$.  Suppose that there exist
$C>0$ and $r_0>0$ such that
\[
 \nu\bigl(B(x,r)\bigr)\le Cr^s
\]
for every $x\in X$ and every $0<r\le r_0$.  Then,
\[
 {\dim_H(F)\ge s}.
\]
\end{lem}

We formulate the next lemma with a slightly weaker assumption \eqref{eq:tree-log-ratio} on $\psi$ and sequence $\{ T_j \}$ than required for the purpose of this proof; it will be later used in the proof of Proposition \ref{prop:uniform-exponent-levels}. 
\begin{lem}
\label{lem:tree}
Suppose that
\begin{equation}\label{eq:tree-log-ratio}
 \frac{-\log\psi(T_{j+1})}{-\log\psi(T_j)}\longrightarrow1,
\end{equation}
and put
\[
 b_j^{(k)}:=
 \inf_{Q\in\cG_{T_j}^{(k)}}
 \#\cG_{T_j\to T_{j+1}}^{(k)}(Q).
\]
Then,
\begin{equation}\label{eq:3.23}
 {\dim_H}\bigl(\mathbf G_\infty^{(k)}\bigr)
 {\ge}
 \liminf_{\ell\to\infty}
 {\frac{\displaystyle\sum_{j=k}^{\ell-1}\log b_j^{(k)}}
 {-\log\psi(T_\ell)}}.
\end{equation}
\end{lem}

\begin{proof}
Let $\nu$ be the natural probability measure on
$\mathbf G_\infty^{(k)}$ obtained by distributing the mass of each
parent $Q \in \cG_{T_j}^{(k)}$ equally among its children $B \in \cG_{T_j\to T_{j+1}}^{(k)}(Q)$.  Every level-$\ell$ ball then has
measure at most
\[
 \left(\prod_{j=k}^{\ell-1}b_j^{(k)}\right)^{-1}.
\]
Suppose $\psi(T_\ell)\le r<\psi(T_{\ell-1})$.
Pairwise disjointness of level-$\ell-1$ balls shows that a ball of radius $r$ meets at most $O_m(1)$ balls from $\cG_{T_{\ell-1}}^{(k)}$.

By \eqref{eq:tree-log-ratio},
\[
\begin{aligned}
 \liminf_{\ell\to\infty}
 \frac{\sum_{j=k}^{\ell-2}\log b_j^{(k)}}{-\log\psi(T_\ell)}
 &=
 \liminf_{\ell\to\infty}
 \frac{\sum_{j=k}^{\ell-2}\log b_j^{(k)}}{-\log\psi(T_{\ell-1})}\\
 &=
 \liminf_{\ell\to\infty}
 \frac{\sum_{j=k}^{\ell-1}\log b_j^{(k)}}{-\log\psi(T_\ell)}.
\end{aligned}
\]
{If the liminf in \eqref{eq:3.23} is zero, the conclusion
is immediate. Otherwise, fix $s>0$ smaller than this liminf.}
For all sufficiently small $r$ (equivalently, sufficiently large $\ell$),
\[
 \nu(B(\x,r))
 \ll_{m}
 \left(\prod_{j=k}^{\ell-2}b_j^{(k)}\right)^{-1}
 \le {\psi(T_\ell)^s}
 \le {r^s}.
\]
Lemma \ref{lem:mass-distribution} gives
${\dim_H}(\mathbf G_\infty^{(k)})
 {\ge s}$.
Letting $s$ tend to the liminf completes the proof.
\end{proof}

The last ingredient in the proof of Proposition \ref{prop:lower} is a lower bound for $b_j^{(k)}$.

\begin{lem}\label{lem:lower-transition}
For $j\ge k$ and $Q\in\cG_{T_j}^{(k)}$,
\begin{equation}\label{eq:3.24}
 \#\cG_{T_j\to T_{j+1}}^{(k)}(Q)
 \ge c_d
 \frac{
 \mathcal N_A(\psi(T_j)-\psi(T_{j+1}),T_{j+1}-T_j)
 }{
 \mathcal N_A(\psi(T_{j+1}),T_{j+1})
 }
\end{equation}
for some constant $c_d$ depending only on dimension. If \eqref{eq:2.7} holds, then
\begin{equation}\label{eq:3.25}
 \#\cG_{T_j\to T_{j+1}}^{(k)}(Q)
 \ge
 c_d\left(\frac{\eta-1}{\eta}\right)^d
 \frac{\mathcal N_A(\psi(T_j),T_{j+1})}
 {\mathcal N_A(\psi(T_{j+1}),T_{j+1})}.
\end{equation}

Here we do not have to require \eqref{eq:2.5} to hold for any $\rho > 1$.
\end{lem}

\begin{proof}
Write
\[
 Q=B(\langle A\q_0\rangle,\psi(T_j)),
 \qquad |\q_0|\le T_j.
\]
By maximality of $\cG_{T_j\to T_{j+1}}^{(k)}(Q)$, every ball
\[
 B'\in\cP_{T_j\to T_{j+1}}^{(k)}(Q)
\]
meets some ball
\[
 B\in\cG_{T_j\to T_{j+1}}^{(k)}(Q).
\]
Since $B$ and $B'$ have the same radius, it follows that
\[
 B'\subseteq 3B.
\]

Choose $\p_0\in\Z^m$.  From now on we use the notation $Q$ for
the Euclidean lift
\[
 Q=B(A\q_0-\p_0,\psi(T_j)) \subseteq \R^m.
\]

Since $\psi(T_j) < \frac{1}{2}$, the projection map $\R^m \to \T^m$ is injective on $Q$. Therefore, every element of $\cP_{T_j\to T_{j+1}}^{(k)}(Q)$ has a unique lift to $\R^m$ which is contained in $Q$. Moreover, we can consider balls from $\cP_{T_j\to T_{j+1}}^{(k)}(Q)$ and $\cG_{T_j\to T_{j+1}}^{(k)}(Q)$ as subsets of $\R^m$; injectivity of the projection map on $Q$ yields that the triple-covering property lifts from $\T^m$ to $\R^m$.

Consider a lattice point
\begin{equation}\label{eq:lattice-pt-general-Q}
 \begin{pmatrix}
 A\q-\p\\
 \q
 \end{pmatrix}
 \in
 u_A\Z^d\cap
 \left(
 B(A\q_0-\p_0,\psi(T_j)-\psi(T_{j+1}))
 \times B(\q_0,T_{j+1}-T_j)
 \right).
\end{equation}
Then
\[
 |\q|
 \le |\q-\q_0|+|\q_0|
 \le T_{j+1},
\]
and
\[
 B' := B(A\q-\p,\psi(T_{j+1}))\subseteq Q,
\]
that is, $B' \in \cP_{T_j\to T_{j+1}}^{(k)}(Q)$. The triple-covering property now implies that there exists $B \in \cG_{T_j\to T_{j+1}}^{(k)}(Q)$ such that $B' \subseteq 3B$. Thus
\begin{equation}\label{eq:3.24.1}
\begin{split}
 &u_A\Z^d\cap
 \left(
 B(A\q_0-\p_0,\psi(T_j)-\psi(T_{j+1}))
 \times B(\q_0,T_{j+1}-T_j)
 \right)\\
 &\quad\subseteq
 \bigcup_{B\in\cG_{T_j\to T_{j+1}}^{(k)}(Q)}
 \left(
 u_A\Z^d\cap
 \bigl(3B\times[-T_{j+1},T_{j+1}]^n\bigr)
 \right).
\end{split}
\end{equation}

We first estimate the number of lattice points on the left-hand side.
Equation \eqref{eq:lattice-translation-count} and Lemma \ref{lem:profile-count} yield that
\begin{equation}\label{eq:3.26}
\begin{split}
 &\#\left(
 u_A\Z^d\cap
 \left(
 B(A\q_0-\p_0,\psi(T_j)-\psi(T_{j+1}))
 \times B(\q_0,T_{j+1}-T_j)
 \right)
 \right)\\
 &\quad=
 N_A\bigl(
 \psi(T_j)-\psi(T_{j+1}),T_{j+1}-T_j
 \bigr)\\
 &\quad\gg_d
 \mathcal N_A\bigl(
 \psi(T_j)-\psi(T_{j+1}),T_{j+1}-T_j
 \bigr).
\end{split}
\end{equation}

Now fix $B \in \cG_{T_j\to T_{j+1}}^{(k)}(Q)$: one has
\[
 B=B(A\q-\p,\psi(T_{j+1})) \quad \text{for some} \quad
 \p\in\Z^m,\quad \q\in\Z^n,\quad |\q|\le T_{j+1}.
\]
The upper bound in
\eqref{eq:translated-count-window} shows that
\[
\begin{aligned}
\#\left(u_A\mathbb Z^d\cap
 \bigl(3B\times[-T_{j+1},T_{j+1}]^n\bigr)\right)
&\le
\#\left(u_A\mathbb Z^d\cap
 \bigl(3B\times[-2T_{j+1},2T_{j+1}]^n\bigr)\right)\\
&\le
N_A(3\psi(T_{j+1}),3T_{j+1}).
\end{aligned}
\]
Lemma \ref{lem:profile-count} and \eqref{eq:profile-scaling} therefore give
\begin{equation}\label{eq:3.24.2}
 \#\left(
 u_A\Z^d\cap
 \bigl(3B\times[-T_{j+1},T_{j+1}]^n\bigr)
 \right)
 \ll_d
 \mathcal N_A(\psi(T_{j+1}),T_{j+1}).
\end{equation}

The union in \eqref{eq:3.24.1} need not be disjoint, but its cardinality is at most the sum of the cardinalities of its members.  Thus \eqref{eq:3.26}
and \eqref{eq:3.24.2} imply that
\[
\begin{split}
 &\mathcal N_A\bigl(
 \psi(T_j)-\psi(T_{j+1}),T_{j+1}-T_j
 \bigr)\\
 &\quad\ll_d
 \#\cG_{T_j\to T_{j+1}}^{(k)}(Q)\,
 \mathcal N_A(\psi(T_{j+1}),T_{j+1}),
\end{split}
\]
which proves \eqref{eq:3.24}.

Now suppose that \eqref{eq:2.7} holds and put
$\beta=(\eta-1)/\eta$.  Then
\begin{equation*}\label{eq:3.27}
 T_{j+1}-T_j\ge\beta T_{j+1},
 \qquad
 \psi(T_j)-\psi(T_{j+1})\ge\beta\psi(T_j).
\end{equation*}
By Lemma \ref{lem:profile-properties},
\[
\begin{split}
  \mathcal N_A(&\psi(T_j)-\psi(T_{j+1}),T_{j+1}-T_j)\\
  &\ge
  \mathcal N_A(\beta\psi(T_j),\beta T_{j+1})\\
  &\ge\beta^d \mathcal N_A(\psi(T_j),T_{j+1}).
\end{split}
\]
Combining this with \eqref{eq:3.24} proves
\eqref{eq:3.25}.
\end{proof}

\begin{proof}[Proof of Proposition \ref{prop:lower}]
{Fix $k\ge1$.}
For any $j \geq k$ and any $Q \in \cG_{T_j}^{(k)}$, the set $\cG_{T_j\to T_{j+1}}^{(k)}(Q)$ is nonempty.  Thus Lemma
\ref{lem:lower-transition} implies that
\[
 b_j^{(k)}\ge
 \max\left\{1,
 c_d\frac{
 \mathcal N_A(\psi(T_j)-\psi(T_{j+1}),T_{j+1}-T_j)}
 {\mathcal N_A(\psi(T_{j+1}),T_{j+1})}
 \right\}.
\]
Lemma \ref{lem:tree} gives the corresponding lower bound for
${\dim_H}(\mathbf G_\infty^{(k)})$.  Lemma
\ref{lem:limit-inclusion} and monotonicity of Hausdorff {dimension} give
\[
 {\dim_H}\bigl(\Dhat_A(\rho\psi)\bigr)
 \ge {\dim_H}\bigl(\mathbf G_\infty^{(k)}\bigr),
\]
{Since $-\log\psi(T_\ell)\to\infty$, adding the first
$k-1$ logarithmic summands does not change the liminf. Thus every fixed
$k$ gives the same lower bound, proving} \eqref{eq:2.6}.  Under \eqref{eq:2.7}, use
\eqref{eq:3.25} instead to obtain \eqref{eq:2.8}.
\end{proof}

\subsection{Proof of the first assertion of Remark \ref{rem:lower-Q}}

The current proof of Proposition \ref{prop:lower} works, with only minor local adjustments needed. Namely:

    \begin{itemize}
        \item Replace $N_A$ by $N_{A, \mathcal Q}^{\p_0, \q_0}$ defined via
        \begin{equation*}\label{eq:3.1-new}
 \begin{aligned}N_{A, \mathcal{Q}}^{\p_0, \q_0}(R,Q):= &
 \#\left\{(\p,\q)\in\Z^m\times \mathcal Q: \right. \\
 & \left. |A(\q-\q_0)-(\p - \p_0)|\le R,\ |\q - \q_0|\le Q\right\}.
 \end{aligned}
\end{equation*}

If $\q_0 \in \mathcal Q$, then, due to \eqref{eq:central-extension}, 
$$\frac{1}{2} N_{A}(R,Q) \leq N_{A, \mathcal{Q}}^{\p_0, \q_0}(R,Q) \leq N_{A}(R,Q).
$$
In particular, Lemma \ref{lem:profile-count} holds with $N_A$ replaced by $N_{A, \mathcal Q}^{\p_0, \q_0}$ (with a different implied constant). 

\item In Subsection \ref{subs:proof-of-lower}, define $\cG_{T_k}^{(k)}$ using $\q_k \in \mathcal Q$. To be able to choose such $\q_k$, we need to choose $k$ large enough so that there exist elements of $\mathcal Q$ with norm not exceeding $T_k$. {Omitting the first $k-1$ logarithmic summands does not change the dimension bounds.} Analogously, define $ \cP_{T_j\to T_{j+1}}^{(k)}(Q)$ with restriction $\q \in \mathcal Q$. 

\item Lemma \ref{lem:limit-inclusion} holds, line to line together with its proof, with $\Dhat_A(\rho\psi)$ replaced by $\Dhat_{A,\mathcal Q}(\rho\psi)$.

\item Lemma \ref{lem:tree} and its proof hold unchanged.

\item Lemma \ref{lem:lower-transition} holds, and its proof is subject to the following local changes:
\begin{itemize}
    \item In \eqref{eq:lattice-pt-general-Q}, \eqref{eq:3.24.1} and \eqref{eq:3.26}, the lattice $u_A \Z^d$ should be replaced by the set
    $$
    \left\{ u_A \begin{pmatrix}
 \p\\
 \q
 \end{pmatrix}: \p \in \Z^m, \q \in \mathcal Q \right\}.
    $$
    \item To obtain the modified version of \eqref{eq:3.26}, one should directly use the modified version of Lemma \ref{lem:profile-count} for $N_{A, \mathcal Q}^{\p_0, \q_0}$ described above.
\end{itemize}

\item With the modifications described above, the proof of Proposition \ref{prop:lower} works as currently written.
    \end{itemize}

\section{Proof of the main theorem}
\label{sec:pressure}

Throughout the section, $\{ T_j \}$ is the equal-$\psi$ mesh defined in \eqref{eq:1.11}.

\subsection{Uniform estimates for Hausdorff dimension}

In this section we show how lower quasimultiplicativity allows us to express the {dimension} bounds obtained in Section \ref{sec:CD} through $\mathcal P_{A,\psi}$.

\begin{thm}
\label{thm:uniform-pressure}
Suppose that the function $\psi$ is lower quasimultiplicative.  There exist $L_0>0$ and
numbers $\varepsilon_L\ge0$,
\begin{equation}\label{eq:4.5}
 \varepsilon_L=O(L^{-1})\qquad \text{as} \ L\to\infty,
\end{equation}
such that for every $L\ge L_0$ and $c>0$ one has
\begin{equation}\label{eq:4.6}
 \mathcal P_{A,\psi}(L,c\e^{-L})-\varepsilon_L
 \le
 \dim_H\Dhat_A(c\psi)
 \le
 \mathcal P_{A,\psi}(L,c)+\varepsilon_L.
\end{equation}
The error $\varepsilon_L$ is uniform in $c$ and $A$.
\end{thm}

We will need the following technical lemma.

\begin{lem}
\label{lem:vertical}
There is a constant $K_d\ge1$, depending only on $d$, such that for all $Q,R_1,R_2>0$ with
$R_1\ge R_2$ one has
\begin{equation}\label{eq:4.2}
 1\le\frac{\mathcal N_A(R_1,Q)}{\mathcal N_A(R_2,Q)}
 \le K_d\left(\frac{R_1}{R_2}\right)^m.
\end{equation}
Consequently, for every $L>0$, $c>0$, and $j\ge0$,
\begin{equation}\label{eq:4.3}
 0\le
 \log\frac{\mathcal N_A(c\psi(T_j),T_{j+1})}
           {\mathcal N_A(c\psi(T_{j+1}),T_{j+1})}
 \le mL+\log K_d.
\end{equation}
\end{lem}

\begin{proof}

The first inequality in \eqref{eq:4.2} follows from the monotonicity in Lemma \ref{lem:profile-properties}.

For the second inequality, cover $[-R_1,R_1]^m$ by $\ll\left(\frac{R_1}{R_2}\right)^m$ axis-parallel cubes of side length $R_2$.  Let $\mathcal K$ be one of these cubes, and suppose $\mathcal K \times [-Q, Q]^n$ contains a point $ \bv_0$ of $u_A \Z^d$. 
Since
\[
 \mathcal K\times[-Q,Q]^n
 \subseteq \bv_0+\bigl([-R_2,R_2]^m\times[-2Q,2Q]^n\bigr),
\]
Lemma \ref{lem:translated-count} gives
\[
 \#\left(u_A\Z^d\cap\bigl(\mathcal K\times[-Q,Q]^n\bigr)\right)
 \le N_A(R_2,2Q).
\]
Summing over these cubes and applying Lemma \ref{lem:profile-count},
monotonicity, and \eqref{eq:profile-scaling}, we obtain
\[
 \mathcal N_A(R_1,Q)
 \ll_d\left(\frac{R_1}{R_2}\right)^m N_A(R_2,2Q)
 \ll_d\left(\frac{R_1}{R_2}\right)^m\mathcal N_A(R_2,Q).
\]
This proves \eqref{eq:4.2}.  Finally, \eqref{eq:4.3} follows by taking $Q=T_{j+1}$, $R_1=c\psi(T_j)$, and $R_2=c\psi(T_{j+1})$, since \eqref{eq:1.11} gives $R_1/R_2=\e^L$.
\end{proof}

\smallskip

For later use, observe that for every $L>0$, $c>0$, and every fixed integer $k\ge0$,
\begin{equation}\label{eq:1.13.1}
 \mathcal P_{A,\psi}(L,c)
 =
 \liminf_{\ell\to\infty}
 \frac1{\ell L-\log c}
 \log\prod_{j=k}^{\ell-1}
 \frac{\mathcal N_A(c\psi(T_j),T_{j+1})}
      {\mathcal N_A(c\psi(T_{j+1}),T_{j+1})}.
\end{equation}
Indeed, deleting finitely many terms does not change the liminf in \eqref{eq:4.19}, and
$\lim\limits_{\ell\to\infty}
\frac{\ell L-\log c}{\ell L}=1$.

\smallskip

Now we are ready to prove Theorem \ref{thm:uniform-pressure}.

\begin{proof}[Proof of Theorem \ref{thm:uniform-pressure}.]
For $c > 0$, put
\begin{equation}\label{eq:R_j(c)}
 \mathcal R_j(c):=
 \frac{\mathcal N_A(c\psi(T_j),T_{j+1})}
      {\mathcal N_A(c\psi(T_{j+1}),T_{j+1})}.
\end{equation}

Fix $\tau$ and $R_0$ as in Definition \ref{def:regular}.
Choose $L_0>\max\{\tau\log R_0,\log R_0\}$.
Fix $L\ge L_0$ and $c>0$. We apply both Propositions \ref{prop:upper} and \ref{prop:lower} to
$c\psi$. By \eqref{eq:1.3} and \eqref{eq:1.11}, for all sufficiently
large $j$ one has
\[
 \psi(\e^{L/\tau}T_j)\ge\e^{-L}\psi(T_j)=\psi(T_{j+1}).
\]
Strict monotonicity of $\psi$ therefore gives
\begin{equation}\label{eq:4.13}
 T_{j+1}\ge \e^{L/\tau}T_j\ge R_0T_j.
\end{equation}
Moreover, \eqref{eq:1.11} gives
\begin{equation}\label{eq:4.11}
 \frac{c\psi(T_j)}{{c\psi(T_{j+1})}}
 =\e^L\ge R_0.
\end{equation}
Thus \eqref{eq:2.5} holds with $\rho=\e^L$, and
\eqref{eq:2.7} holds with the fixed constant
\begin{equation}\label{eq:4.14}
 \eta:=R_0.
\end{equation}
We omit finitely many initial mesh points and modify
$c\psi$ on the resulting bounded initial interval to satisfy the
standing smallness assumption in Section \ref{sec:CD}. This changes
neither $\Dhat_A(c\psi)$ and $\Dhat_A(\e^Lc\psi)$ nor the tail limits below.

First, we obtain the upper estimate. Using \eqref{eq:R_j(c)} and \eqref{eq:1.11}, Proposition \ref{prop:upper} {and \eqref{eq:1.13.1}} give
\begin{equation}\label{eq:4.9}
\begin{split}
 \dim_H\Dhat_A({c\psi})
 &\le
 {\liminf_{\ell\to\infty}
 \frac{(\ell-1)\log C_d+
       \sum_{j=1}^{\ell-1}\log\mathcal R_j(c)}
      {\ell L-\log c}}\\
 &{=}
 \mathcal P_{A,\psi}(L,{c})+\frac{\log C_d}{L}.
\end{split}
\end{equation}

For the lower estimate, apply estimate \eqref{eq:2.8} of Proposition
\ref{prop:lower} to $c\psi$. Put
\[
 c_*:=c_d\left(\frac{\eta-1}{\eta}\right)^d.
\]
Proposition \ref{prop:lower} gives
\[
 \dim_H\Dhat_A(\e^Lc\psi)
 \ge
 \liminf_{\ell\to\infty}
 \frac{\sum_{j=1}^{\ell-1}\log\max\{1,c_*\mathcal R_j(c)\}}
      {\ell L-\log c}
 \ge \mathcal P_{A,\psi}(L,c)+\frac{\log c_*}{L}.
\]
Here the last inequality follows from
$\log\max\{1,c_*\mathcal R_j(c)\}\ge\log c_*+\log\mathcal R_j(c)$
and \eqref{eq:1.13.1}. Replacing $c$ by $c\e^{-L}$ in this lower
bound and combining it with \eqref{eq:4.9} proves \eqref{eq:4.6} with
\[
 \varepsilon_L=
 \frac1L\max\left\{\log C_d,
 {0,-\log c_*}\right\}.
\]
\end{proof}

For {$0<c_1<c_2$}, one has
\[
 \Dhat_A({c_1}\psi)\subseteq\Dhat_A({c_2}\psi).
\]
Consequently, {$c\mapsto\dim_H\Dhat_A(c\psi)$} is nondecreasing, and
\begin{equation}\label{eq:4.17}
\begin{split}
 {\lim_{c\to0}\dim_H\Dhat_A(c\psi)}
 &={\inf_{c>0}\dim_H\Dhat_A(c\psi)}
 =\inf_{c>0}\dim_H\Dhat_A(\widetilde\psi_c),\\
 {\lim_{c\to\infty}\dim_H\Dhat_A(c\psi)}
 &={\sup_{c>0}\dim_H\Dhat_A(c\psi)}
 =\sup_{c>0}\dim_H\Dhat_A(\widetilde\psi_c).
\end{split}
\end{equation}
The last equalities follow from Lemma \ref{lem:rescaling-comparison} under the assumption of lower quasimultiplicativity of $\psi$.
Indeed, \eqref{eq:rescaling-comparison} places the dimension of every
$\Dhat_A(\widetilde\psi_c)$ between
{$\inf_{c>0}\dim_H\Dhat_A(c\psi)$} and
{$\sup_{c>0}\dim_H\Dhat_A(c\psi)$}, while
\eqref{eq:elementary-rescaling-comparison}, first for $0<c\le1$ and
then for $c\ge1$, gives the reverse inequalities after taking the
infimum and the supremum, respectively.

Taking the infimum and the supremum over {$c>0$} in \eqref{eq:4.6}, and
observing that {$c\e^{-L}$ ranges over $\R_+$ together with $c$}, gives the
following comparison.  Its final assertion follows from \eqref{eq:4.5}
and \eqref{eq:4.17}.

\begin{cor}\label{cor:endpoints}
For every $L\ge L_0$, where $L_0$ is as in Theorem \ref{thm:uniform-pressure},
\begin{equation*}\label{eq:4.21}
\begin{split}
 \left|
 {\inf_{c>0}\mathcal P_{A,\psi}(L,c)}
 -{\inf_{c>0}\dim_H\Dhat_A(c\psi)}
 \right|&\le\varepsilon_L,\\
 \left|
 {\sup_{c>0}\mathcal P_{A,\psi}(L,c)}
 -{\sup_{c>0}\dim_H\Dhat_A(c\psi)}
 \right|&\le\varepsilon_L.
\end{split}
\end{equation*}
Consequently, the limits $P_-(A,\psi)$ and $P_+(A,\psi)$ introduced
in Theorem \ref{thm:main} exist, and
\begin{equation}\label{eq:4.22}
\begin{split}
 P_-(A,\psi)
 &={\inf_{c>0}\dim_H\Dhat_A(c\psi)}
 =\inf_{c>0}\dim_H\Dhat_A(\widetilde\psi_c),\\
 P_+(A,\psi)
 &={\sup_{c>0}\dim_H\Dhat_A(c\psi)}
 =\sup_{c>0}\dim_H\Dhat_A(\widetilde\psi_c).
\end{split}
\end{equation}
\end{cor}

\subsection{No dimension drop}\label{subsec:no-drop}

\begin{prop}
\label{prop:no-drop}
Under the assumptions of Theorem \ref{thm:main},
\begin{equation}\label{eq:4.28}
 \dim_H\Sing_A(\psi)
 =\inf_{c>0}\dim_H\Dhat_A(\widetilde\psi_c)
 =\inf_{c>0}\dim_H\Dhat_A(c\psi)
 =P_-(A,\psi).
\end{equation}
\end{prop}

Recall that $P_-(A,\psi)$ si well-defined due to Corollary \ref{cor:endpoints}. To prove Proposition \ref{prop:no-drop}, we will need the following diagonalization lemma.

\begin{lem}
\label{lem:diagonal}
Let $d_0 \in \R$, and let
\[
 a_j^{(p)}\in[0,M], \qquad j\in\Z_{\ge0},\quad p\in\N,
\]
satisfy
\begin{equation}\label{eq:4.23}
 \liminf_{N\to\infty}\frac1N
 \sum_{j=0}^{N-1}a_j^{(p)}
 \ge d_0\qquad \text{for any} \  p\in\N.
\end{equation}
Then there is a sequence $\{p_j\}_{j\ge0}$ of positive integers such
that
\begin{equation}\label{eq:4.24}
 p_{j+1}-p_j\in\{0,1\},\qquad
 p_j\to\infty,\qquad
 \frac{p_j}{j}\to0,
\end{equation}
and
\begin{equation}\label{eq:4.25}
 \liminf_{N\to\infty}\frac1N
 \sum_{j=0}^{N-1}a_j^{(p_j)}
 \ge d_0.
\end{equation}
The same assertion holds if both of the inequalities in
\eqref{eq:4.23} and \eqref{eq:4.25} reversed.

\end{lem}

\begin{proof}
Choose a decreasing sequence
 $\{\delta_p\}_{p\in \N}$ of positive real numbers such that $\delta_p\to0$ and put
\[
 B_p(N):=
 \sum_{j=0}^{N-1}
 \bigl(a_j^{(p)}-(d_0-\delta_p)\bigr),
 \qquad N\in\Z_{\ge0},\quad p\in\N.
\]
For every fixed $p\in\N$, equation \eqref{eq:4.23} implies that $\liminf\limits_{N \to \infty} \frac{B_p(N)}{N} \geq \delta_p$ and thus
\[
 B_p(N)\ge\frac{\delta_p}{2}N
\]
for all sufficiently large $N$. In particular, $B_p(N)\to+\infty$.

We recursively choose integers
\[
 1\le K_1<K_2<\cdots
\]
as follows.  Put $K_0=0$.  Having chosen $K_{p-1}$, choose an integer
\[
 R_p>\max\{K_{p-1},p^2\},
\]
and let $K_p\ge R_p$ be a point at which $B_p$ attains its minimum on
the integers $N\ge R_p$.  Such a point exists because
$B_p(N)\to+\infty$.  For every $N\ge K_p$, we then have
$B_p(N)-B_p(K_p)\ge0$,
and consequently
\begin{equation}\label{eq:4.26}
 \sum_{j=K_p}^{N-1}a_j^{(p)}
 \ge(d_0-\delta_p)(N-K_p).
\end{equation}

Put $p_j=1$ for $0\le j<K_1$, and define $p_j=p$ when
\[
 K_p\le j<K_{p+1}, \qquad p\in\N.
\]
The first two assertions in \eqref{eq:4.24} are immediate. Moreover, if
$K_p\le j<K_{p+1}$, then
\[
 0\le\frac{p_j}{j}
 =\frac pj
 \le\frac p{K_p}
 <\frac1p.
\]
Since $p\to\infty$ as $j\to\infty$, this proves the last assertion in
\eqref{eq:4.24}.

For $N\ge K_1$, applying \eqref{eq:4.26} to each completed block and
the last partial block gives
\begin{equation}\label{eq:4.27}
 \sum_{j=0}^{N-1}a_j^{(p_j)}
\ge d_0N-\sum_{j=0}^{N-1}\delta_{p_j}+O(1).
\end{equation}
Here the $O(1)$ term covers the fixed initial block $0\le j<K_1$.  Since $p_j\to\infty$ and $\delta_p\to0$,
\[
 \frac1N\sum_{j=0}^{N-1}\delta_{p_j}\longrightarrow0
\]
as $N \to \infty$. Dividing \eqref{eq:4.27} by $N$ and taking the lower limit proves
\eqref{eq:4.25}.

For the reversed inequalities, choose  $K_1=0$ and
recursively choose integers $K_{p+1}>\max\{K_p,(p+1)^2\}$ such that
\begin{equation}\label{eq:reverse_ineq_c}
 \frac{MK_p+\displaystyle\sum_{j=0}^{K_{p+1}-1}a_j^{(p)}}
      {K_{p+1}}
 \le d_0+\frac1p.
\end{equation}
This is possible by the reversed hypothesis \eqref{eq:4.23}.
Define $p_j$ by the same block rule as above. Then
\eqref{eq:4.24} holds as before, while boundedness and nonnegativity of $a_j^{(p)}$ toogether with \eqref{eq:reverse_ineq_c} give
\[
 \sum_{j=0}^{K_{p+1}-1}a_j^{(p_j)}
 \le MK_p+\sum_{j=0}^{K_{p+1}-1}a_j^{(p)}
 \le\left(d_0+\frac1p\right)K_{p+1}.
\]
Dividing by $K_{p+1}$ and letting $p\to\infty$ proves the reversed
\eqref{eq:4.25}.
\end{proof}

The technical lemma below will also be used in proofs of Propositions \ref{prop:endpoint-representations} and
\ref{prop:dimension-continuity}.

\begin{lem}\label{lem:interpolation}
Let $T_j\to\infty$ be strictly increasing and let $c_j\to\infty$ be
nondecreasing and positive. Suppose that, for some $0<\delta<1$,
\begin{equation}\label{eq:interpolation-condition}
 1\le\frac{c_{j+1}}{c_j}
 \le\left(\frac{\psi(T_j)}{\psi(T_{j+1})}\right)^\delta
 \qquad\text{for every }j.
\end{equation}
Then there exist continuous strictly decreasing functions $\varphi$
and $\vartheta$, tending to zero, such that
\begin{equation*}\label{eq:interpolation-values}
 \varphi(T_j)=c_j\psi(T_j),
 \qquad
 \vartheta(T_j)=c_j^{-1}\psi(T_j),
\end{equation*}
and $\psi=o(\varphi)$, $\vartheta=o(\psi)$. The ratios
$\varphi/\psi$ and $\psi/\vartheta$ are nondecreasing.
If $\psi$ is lower quasimultiplicative, then so are $\varphi$
and $\vartheta$.
\end{lem}

\begin{proof}
For $T_j\le T\le T_{j+1}$, put
\[
 \beta_j:=
 \frac{\log(c_{j+1}/c_j)}
      {\log(\psi(T_j)/\psi(T_{j+1}))}\in[0,\delta],
 \qquad
 c(T):=c_j\left(\frac{\psi(T_j)}{\psi(T)}\right)^{\beta_j},
\]
and extend $c$ constantly to the omitted initial interval.
Then $c$ is continuous and nondecreasing, with $c(T_j)=c_j$.
Define
\[
 \varphi:=c\psi,\qquad \vartheta:=c^{-1}\psi.
\]
On each mesh interval, these functions are positive constant multiples of
$\psi^{1-\beta_j}$ and $\psi^{1+\beta_j}$, respectively.

For $T_j \le T \le S \le T_{j+1}$, the definition shows that $\frac{c(S)}{c(T)}
=\left(\frac{\psi(T)}{\psi(S)}\right)^{\beta_j}$. Therefore, for any $S \ge T$ one has
$$
1 \le \frac{c(S)}{c(T)} \le \left(\frac{\psi(T)}{\psi(S)}\right)^{\delta}
$$
and so
\begin{equation}\label{eq:interpolation-comparison}
\begin{split}
 \frac{\psi(S)}{\psi(T)}
 &\le\frac{\varphi(S)}{\varphi(T)}
 \le\left(\frac{\psi(S)}{\psi(T)}\right)^{1-\delta},\\
 \left(\frac{\psi(S)}{\psi(T)}\right)^{1+\delta}
 &\le\frac{\vartheta(S)}{\vartheta(T)}
 \le\frac{\psi(S)}{\psi(T)}.
\end{split}
\end{equation}
Thus both functions are strictly decreasing to zero. Since
$c(T)\nearrow \infty$, we get the relative asymptotics of the new functions compared to $\psi$ and monotonicity of ratios. Finally, \eqref{eq:interpolation-comparison} and \eqref{eq:1.3} give lower
quasimultiplicativity with exponents $\tau$ and $(1+\delta)\tau$,
respectively.
\end{proof}

\begin{proof}[Proof of Proposition \ref{prop:no-drop}]
The inclusion
{$\Sing_A(\psi)\subseteq\Dhat_A(c\psi)$ for every $c>0$}, together
with \eqref{eq:4.22}, gives
\begin{equation}\label{eq:4.29}
 \dim_H\Sing_A(\psi)\le P_-(A,\psi).
\end{equation}
We prove the reverse inequality.  The case $P_-(A,\psi)=0$ is
immediate, so assume $P_-(A,\psi)>0$.  Fix $s$ with $0<s<P_-(A,\psi)$.

Fix $\tau$ and $R_0$ as in Definition \ref{def:regular}. As in \eqref{eq:4.14}, let $\eta=R_0$, and put
\[
 c_*:=c_d\left(\frac{\eta-1}{\eta}\right)^d,
 \qquad
 \delta_L:=\frac{\max\{0,-\log c_*\}}{L}.
\]
Due to the definition given in \eqref{eq:pressure-minus}, we can choose $L$ so large that
\begin{equation}\label{eq:choice-of-L}
 s<{\inf_{c>0}\mathcal P_{A,\psi}(L,c)}-\delta_L.
\end{equation}
Fix such an $L>1$.  For $p\ge1$ and $j\ge0$, let
\begin{equation*}\label{eq:4.30}
 a_j^{(p)}:=\frac1L
 \log
 \frac{\mathcal N_A(\e^{-p}\psi(T_j),T_{j+1})}
      {\mathcal N_A(\e^{-p}\psi(T_{j+1}),T_{j+1})}.
\end{equation*}
By Lemma \ref{lem:vertical},
$$
 a_j^{(p)}\in\left[0,m+\frac{\log K_d}{L}\right].
$$
By \eqref{eq:4.19},
\[
 \liminf_{N\to\infty}\frac1N
 \sum_{j=0}^{N-1}a_j^{(p)}
 =
 \mathcal P_{A,\psi}(L,{\e^{-p}})
 \ge{\inf_{c>0}\mathcal P_{A,\psi}(L,c)}.
\]
Apply Lemma \ref{lem:diagonal} with {$d_0=\inf\limits_{c>0}\mathcal P_{A,\psi}(L,c)$}, and let $\{p_j\}$ be
the resulting sequence.  Then
{
\begin{equation}\label{eq:4.31}
 \liminf_{N\to\infty}
\frac1N\sum_{j=0}^{N-1}a_j^{(p_j)}
 \ge\inf_{c>0}\mathcal P_{A,\psi}(L,c).
\end{equation}
}

By \eqref{eq:1.11} and \eqref{eq:4.24}, the sequences $\{T_j\}$ and
$c_j:=\e^{p_j}$ satisfy \eqref{eq:interpolation-condition} with
$\delta=L^{-1}$. Let $\vartheta$ be the continuous strictly decreasing,
lower quasimultiplicative function given by Lemma
\ref{lem:interpolation}; one has
\begin{equation}\label{eq:4.33}
 \frac{\vartheta(T)}{\psi(T)}
 \longrightarrow0.
\end{equation}
At the mesh points,
{
\begin{equation}\label{eq:4.34}
 \vartheta(T_j)=\e^{-p_j}\psi(T_j)=\e^{-(jL+p_j)}
\end{equation}
}
and
\begin{equation}\label{eq:4.35}
 \frac{\vartheta(T_j)}{\vartheta(T_{j+1})}
 =\e^{L+p_{j+1}-p_j}\in[\e^L,\e^{L+1}].
\end{equation}
We can additionally assume that $L \geq L_0$, where $L_0$ is chosen as in the proof of Theorem \ref{thm:uniform-pressure}. 
After omitting finitely many initial indices, \eqref{eq:4.13} gives
\[
 T_{j+1}\ge R_0T_j.
\]
Moreover, since $L$ is sufficiently large that $\e^L\ge R_0$, the
lower bound in \eqref{eq:4.35} gives
\[
 \vartheta(T_j)\ge R_0\vartheta(T_{j+1}).
\]
Thus \eqref{eq:2.7} holds for $\vartheta$ with
$\eta=R_0$.  The upper bound in \eqref{eq:4.35} shows that
\eqref{eq:2.5} holds with
\[
 \rho_L:=\e^{L+1}.
\]
As before, $\vartheta$ may be modified on the omitted bounded interval
to satisfy the standing smallness assumption in Section \ref{sec:CD}. Now we can apply Proposition \ref{prop:lower} to the function $\vartheta$ and sequence $\{ T_j \}$.

By \eqref{eq:4.34} and since $p_{j+1}\ge p_j$, one has
$\vartheta(T_{j+1})\le\e^{-p_j}\psi(T_{j+1})$. Thus monotonicity
in \eqref{eq:4.2} gives
\[
 \log\max\left\{1,c_*
 \frac{\mathcal N_A(\vartheta(T_j),T_{j+1})}
      {\mathcal N_A(\vartheta(T_{j+1}),T_{j+1})}\right\}
 \ge La_j^{(p_j)}+\log c_*.
\]
Proposition \ref{prop:lower} and \eqref{eq:4.34} therefore give
\begin{equation}\label{eq:4.40}
\begin{gathered}
 \dim_H\Dhat_A(\rho_L\vartheta)
 \ge\liminf_{N\to\infty}
 \frac{L\sum_{j=1}^{N-1}a_j^{(p_j)}+(N-1)\log c_*}{NL+p_N}\\
 =\liminf_{N\to\infty}\frac1N\sum_{j=0}^{N-1}a_j^{(p_j)}
   +\frac{\log c_*}{L}
 \ge\inf_{c>0}\mathcal P_{A,\psi}(L,c)-\delta_L>s.
\end{gathered}
\end{equation}
The equality uses boundedness of $a_j^{(p_j)}$ and $p_N/N\to0$
from \eqref{eq:4.24}; the last line follows from \eqref{eq:4.31},
the definition of $\delta_L$, and \eqref{eq:choice-of-L}.

By \eqref{eq:4.33}, for every $c>0$ one has
$\rho_L\vartheta(T)\le c\psi(T)$ for all sufficiently large $T$.
Therefore
\begin{equation*}\label{eq:4.41}
 \Dhat_A(\rho_L\vartheta)
 \subseteq
 \bigcap_{c>0}\Dhat_A(c\psi)
 =\Sing_A(\psi),
\end{equation*}
and thus
$$
 \dim_H\Sing_A(\psi)\ge s.
$$
Letting $s\nearrow P_-(A,\psi)$ proves the reverse inequality in
\eqref{eq:4.29}.  The remaining equalities in \eqref{eq:4.28} follow
from \eqref{eq:4.22}.
\end{proof}

\subsection{Proof of Theorem \ref{thm:main}}\label{subs:proof-of-main}

Corollary \ref{cor:endpoints} shows that the limits defining
$P_-(A,\psi)$ and $P_+(A,\psi)$ exist.  Proposition
\ref{prop:no-drop} gives \eqref{eq:1.15}.
For $0<a<b$, one has $a\psi(T)\le b\psi(T)$.
Consequently,
\begin{equation}\label{eq:4.42}
  U_A(\psi)
 =
 \bigcup_{N=1}^\infty\Dhat_A(N\psi).
\end{equation}
By countable stability of Hausdorff dimension,
\[
 \dim_H U_A(\psi)
 =\sup_{N\ge1}\dim_H\Dhat_A(N\psi)
 =\sup_{c>0}\dim_H\Dhat_A(c\psi)
 =P_+(A,\psi).
\]
Here the last
equality is \eqref{eq:4.22}.  This proves \eqref{eq:1.16}.

By \eqref{eq:4.22}, $P_-(A,\psi)=P_+(A,\psi)$ if and only if
$\dim_H\Dhat_A(c\psi)$ is independent of $c>0$. This proves the
equivalence in \eqref{eq:1.17}. Under this condition, denote the
common value by $D$. Equations \eqref{eq:1.15} and \eqref{eq:1.16} give
\[
 \dim_H\Sing_A(\psi)=D=\dim_H U_A(\psi).
\]
Finally, $\dim_H\Dhat_A(\psi)=D$, and
\[
 {\inf_{c>0}\mathcal P_{A,\psi}(L,c)}
 \le\mathcal P_{A,\psi}(L,{1})
 \le{\sup_{c>0}\mathcal P_{A,\psi}(L,c)}.
\]
Equation \eqref{eq:1.18} now follows from \eqref{eq:1.17}.

\smallskip

The last assertion of Remark \ref{rem:lower-Q} follows in the same
way. Indeed, the first assertion of that remark allows us to replace
$\Dhat_A(\rho\psi)$ by $\Dhat_{A,\mathcal Q}(\rho\psi)$ in Proposition
\ref{prop:lower}, after omitting finitely many initial mesh points if
necessary. Proposition \ref{prop:upper} remains valid with
$\Dhat_A(\psi)$ replaced by $\Dhat_{A,\mathcal Q}(\psi)$, by the
trivial inclusion
\[
 \Dhat_{A,\mathcal Q}(\psi)\subseteq\Dhat_A(\psi).
\]
The right-hand sides of the estimates in Propositions
\ref{prop:upper} and \ref{prop:lower} are the only input from the
covering constructions used in this section. Hence all the arguments
above remain unchanged.

\section{Functions close to powers and uniform exponents}\label{sec:near-power-proofs}

\subsection{Some lemmata from lattice geometry}

Write $\Lambda_A=u_A\Z^d$ and put
\[
 a_{x,y}:=\operatorname{diag}(\e^yI_m,\e^{-x}I_n),
 \qquad
 F(x,y):=\max_{\Gamma}\bigl(-\log|a_{x,y}\omega_\Gamma|\bigr),
\]
where $\Gamma$ runs over primitive sublattices of $\Lambda_A$, including
the rank-zero lattice, and $\omega_\Gamma$ is the exterior product of
a basis of $\Gamma$; the empty exterior product is $1$.
Here primitive means
$\Gamma=\Lambda_A\cap\spanR\Gamma$. The norm is the supremum norm in the standard
exterior basis; considering the rank-zero $\Gamma$ yields that $F(x,y) \ge 0$ for any $x, y$. 

\begin{lem}\label{lem:power-profile}
The function $F$ is continuous and locally piecewise affine, and
\begin{equation}\label{eq:power-profile-comparison}
 F(x,y)=\log\mathcal N_A(\e^{-y},\e^x)+O_d(1),
 \qquad 0\le-\frac{\partial F}{\partial y}\le m
\end{equation}
wherever the derivative exists.
\end{lem}

\begin{proof}
Minkowski's second theorem, applied in the span of a rank-$r$
sublattice $a_{x,y}\Gamma$, gives
\begin{equation}\label{eq:Mink_second_first}
 |a_{x,y}\omega_\Gamma|
 \gg_d\prod_{i=1}^r\lambda_i(a_{x,y}\Lambda_A).
\end{equation}
Taking $\Gamma$ to be the primitive sublattice of $\Lambda_A$ in the span of independent vectors whose images under $a_{x,y}$ realize the first $r$ minima yields the reverse bound.
Now let 
\begin{equation}\label{eq:power-small-lattice}
 r':=\#\{i:\lambda_i(a_{x,y}\Lambda_A)<1\}, \quad
 \Gamma_{r'}:=\Lambda_A\cap
 \spanR\{\bv\in\Lambda_A:|a_{x,y}\bv|<1\}.
\end{equation}
The discussion above shows that
\begin{equation}\label{eq:Mink_second_equiv}
 \min_{\Gamma}|a_{x,y}\omega_\Gamma| \asymp_d |a_{x,y}\omega_{\Gamma_{r'}}|
 \asymp_d\prod_{i=1}^{r'}\lambda_i(a_{x,y}\Lambda_A).
\end{equation}
It remains to use \eqref{eq:rectangular-profile} to deduce
the first assertion in \eqref{eq:power-profile-comparison}.

For a fixed $\Gamma$ of rank $r$, write $$\omega_\Gamma=\sum_{|I|=r}\omega_I e_I,
\qquad
e_I=e_{i_1}\wedge\cdots\wedge e_{i_r}.$$
If \(I\) contains \(a\) indices from the first \(m\) coordinates and \(b\) from the last \(n\), then \(a+b=r\), and 
\[
a_{x,y}e_I=\e^{ay-bx}e_I.
\]
Therefore, 
$$
|a_{x,y}\omega_\Gamma| =\max_{\substack{a+b=r\\0\le a\le m,\;0\le b\le n}} C_a\,\e^{ay-bx}
$$
for some $C_a \ge 0$, and thus
\begin{equation}\label{eq:power-exterior-profile}
 \log|a_{x,y}\omega_\Gamma|
 =\max_{a+b=r}\{c_a+ay-bx\},
 \qquad 0\le a\le m,\quad0\le b\le n.
\end{equation}
Here $c_a = \log C_a \in\R\cup\{-\infty\}$ depends on $\Gamma$.
We have established earlier that $F(x,y) \ge 0$, and on any compact set of $(x,y)$ only finitely many primitive exterior vectors can satisfy
$|a_{x,y}\omega_\Gamma|\le1$ at some point, so the maximum defining
$F$ is locally finite. Formula~\eqref{eq:power-exterior-profile}
therefore proves continuity, local piecewise affinity, and the
slope bound in \eqref{eq:power-profile-comparison}.
\end{proof}

Fix $\tau>0$ and write
\begin{equation}\label{eq:power-normalized-minima}
 g_i(x):=\log\lambda_i\left(\frac{mn(1+\tau)}d x\right)
             +\frac{m\tau-n}{d}x,
 \qquad G_r(x):=\sum_{i=1}^r g_i(x).
\end{equation}

Note that $\e^{g_i(x)}$ are the successive minima of
$a_{x,\tau x}\Lambda_A$.

\begin{lem}\label{lem:power-gaps}
Let $1\le r<d$. On each interval where $g_r<g_{r+1}$,
the vectors in $\Lambda_A$ whose images realize the first $r$
minima span a fixed subspace, and
\begin{equation}\label{eq:power-gap-profile}
 G_r(x)=\max_{a+b=r}\{c_a+(a\tau-b)x\}+O_d(1),
 \qquad 0\le a\le m,\quad0\le b\le n,
\end{equation}
with constants $c_a\in\R\cup\{-\infty\}$ depending on this span.
For $r=0$ and $r=d$, the spans are always $\{0\}$ and $\R^d$,
respectively, and \eqref{eq:power-gap-profile} holds on every
interval, with no inequality between successive minima required.
\end{lem}

\begin{proof}
On an interval where $g_r<g_{r+1}$, the span is locally constant and hence constant. If $\Gamma$ is the primitive lattice in this
span, the minima of $a_{x,\tau x}\Gamma$ are precisely
$\e^{g_1},\ldots,\e^{g_r}$. Minkowski's second theorem and
\eqref{eq:power-exterior-profile} prove
\eqref{eq:power-gap-profile}.
For $r=0$, the empty sum gives $G_0=0$. For $r=d$,
unimodularity of $\Lambda_A$ and Minkowski's second theorem give
$G_d(x)=(m\tau-n)x+O_d(1)$ for every $x$.
\end{proof}

\begin{lem}\label{lem:power-maximizer}
Let $r'$ and $\Gamma_{r'}$ be as in \eqref{eq:power-small-lattice}. There is $H_d>0$ such that, if every logarithm of successive minimum of $a_{x,y}\Lambda_A$ lies outside $[-H_d,H_d]$, then the maximum
defining $F(x,y)$ is attained uniquely by $\Gamma_{r'}$.
\end{lem}

\begin{proof}
Recall that, by \eqref{eq:Mink_second_first} and \eqref{eq:Mink_second_equiv}, 
\[
\begin{aligned}
 |a_{x,y}\omega_{\Gamma_{r'}}|
 \asymp_d\prod_{i=1}^{r'}\lambda_i(a_{x,y}\Lambda_A), \quad
 |a_{x,y}\omega_\Gamma|
 \gg_d\prod_{i=1}^r\lambda_i(a_{x,y}\Lambda_A)
 \quad\text{if }\rank\Gamma=r.
\end{aligned}
\]
For $r\ne r'$, the latter product exceeds the former by a factor at least $\e^{H_d}$. If $0<r'<d$ and $\Gamma\ne\Gamma_{r'}$ has rank
$r'$, the last minimum of $a_{x,y}\Gamma$ is at least
$\lambda_{r'+1}(a_{x,y}\Lambda_A)$: otherwise its span
would equal that of $a_{x,y}\Gamma_{r'}$, and primitivity would give
$\Gamma=\Gamma_{r'}$. Minkowski's second theorem therefore gives
\[
\begin{aligned}
 |a_{x,y}\omega_\Gamma|
 \gg_d\lambda_{r+1}(a_{x,y}\Lambda_A)
            \prod_{i=1}^{r-1}\lambda_i(a_{x,y}\Lambda_A)
 >\e^{2H_d}\prod_{i=1}^r\lambda_i(a_{x,y}\Lambda_A).
\end{aligned}
\]
Taking $H_d$ large enough to suppress absolute constants arising from applications of Minkowski's second theorem proves uniqueness. For $r=0,d$, there is no other primitive sublattice of the same rank, so only the different-rank comparison is needed.
\end{proof}

\begin{lem}\label{lem:power-threshold}
If $\tau\notin\mathcal E_{m,n}$, then, for every $1\le i\le d$,
\begin{equation}\label{eq:power-single-threshold-density}
 \lim_{\eta\searrow0}\limsup_{x\to\infty}\frac1x
 \left|\left\{u\in[0,x]:|g_i(u)|\le\eta x\right\}\right|=0.
\end{equation}
Consequently,
\begin{equation}\label{eq:power-threshold-density}
 \lim_{\eta\searrow0}\limsup_{x\to\infty}\frac1x
 \left|\left\{u\in[0,x]:\min_{1\le i\le d}|g_i(u)|\le\eta x
 \right\}\right|=0.
\end{equation}
In particular, replacing $\eta x$ by any fixed constant gives
a set of length $o(x)$.
\end{lem}

\begin{proof}
Fix $1\le i\le d$. Each $g_k$ is Lipschitz with constant
$K:=\max\{1,\tau\}$: changing $u$ by $v\ge0$ multiplies the norm
of any vector by a factor between $\e^{-v}$ and $\e^{\tau v}$.

\smallskip
\noindent\emph{Step 1. Bound the length of the set where certain sums of minima are close.}
Let $0\le r<s\le d$, and suppose that \begin{equation}\label{eq:k_non-bdry_comparison}
    g_k(u)<g_{k+1}(u)
\ \text{for any} \ u \in I \ \text{and every} \
k\in\{r,s\}\cap\{1,\ldots,d-1\}.
\end{equation}
Recall that $G_0=0$. %
By Lemma~\ref{lem:power-gaps}, the function $G_s-G_r$ differs
by $O_d(1)$ from a continuous piecewise-affine function with at
most $2d+1$ pieces. Its slopes have the form
\[
 a\tau-b,\qquad a+b=s-r>0,\qquad
 a,b\in\Z,\quad |a|\le m,\quad |b|\le n.
\]
A zero slope would imply $a,b>0$ and
$\tau=b/a\in\mathcal E_{m,n}$. Thus the finitely many possible
slopes are bounded away from zero in terms of $m,n,\tau$.
On each affine piece, the inverse image of an interval of
length $O_d(H+1)$ therefore has length $O_{m,n,\tau}(H+1)$.
It follows that, for every $H\ge0$,
\begin{equation}\label{eq:power-separated-block}
 \left|\left\{u\in I:|G_s(u)-G_r(u)|\le H\right\}\right|
 \ll_{m,n,\tau}H+1.
\end{equation}

\smallskip
\noindent\emph{Step 2. Decomposing the desired set into a finite union.}
Fix $x\ge1$ and $0<\eta<4^{-d}$, and put
\[
 R_j:=\eta^{(d-j)/d}x,
 \qquad 0\le j\le d.
\]
At any point $u$ with $|g_i(u)|\le\eta x=R_0$, the $d+1$
nested sets
\begin{equation}\label{eq: set-of-indices-Rj}
 \{k:|g_k(u)|\le R_j\},\qquad 0\le j\le d,
\end{equation}
all contain $i$. Two consecutive sets must coincide, since their cardinalities cannot increase strictly $d$ times. Since $g_1(u)\le\cdots\le g_d(u)$, every index between two indices satisfying $-R_j\le g_k(u)\le R_j$ also satisfies these inequalities. Denote the first and last indices in \eqref{eq: set-of-indices-Rj} by $r+1$ and $s$, respectively; the set \eqref{eq: set-of-indices-Rj} is therefore equal to $\{r+1,\ldots,s\}$, and $0\le r<i\le s\le d$ because it contains $i$.
For $0\le j<d$ and $0\le r<i\le s\le d$, let $E_{j,r,s}$
be the set of points $u\in[0,x]$ satisfying
\begin{equation}\label{eq:power-block-separation}
\begin{gathered}
 |g_k(u)|\le R_j \ \text{if}\ r<k\le s,\\
 g_r(u)<-R_{j+1} \ \text{if } \ r\ge1, \qquad 
 g_{s+1}(u)>R_{j+1} \ \text{if } \ s<d.
\end{gathered}
\end{equation}
The preceding argument proves that
\[
 \{u\in[0,x]:|g_i(u)|\le\eta x\}
 \subseteq\bigcup_{\substack{0\le j<d\\0\le r<i\le s\le d}}E_{j,r,s}.
\]

\smallskip
\noindent\emph{Step 3. Estimate for lengths of the sets $E_{j,r,s}$.}
Fix $j,r,s$ and partition $[0,x]$ into intervals of length $R_{j+1}/(8K)$ (the last one may be shorter). If an interval $I$ meets $E_{j,r,s}$ at $u_0$, then $$\begin{aligned}
g_{r+1}(u_0)-g_r(u_0)&>R_{j+1}-R_j
&&\text{if }r\ge1,\\
g_{s+1}(u_0)-g_s(u_0)&>R_{j+1}-R_j
&&\text{if }s<d.
\end{aligned}$$
Lipschitz continuity then guarantees that 
$$\begin{gathered}
g_{r+1}(u)-g_r(u)>R_{j+1}-R_j-\frac{R_{j+1}}4 > 0
\ \text{if }r\ge1,\\
g_{s+1}(u)-g_s(u)>R_{j+1}-R_j-\frac{R_{j+1}}4 > 0
\ \text{if }s<d
\end{gathered}$$
for any $u\in I$; thus \eqref{eq:k_non-bdry_comparison} holds.
Moreover, \eqref{eq:power-block-separation} gives $|G_s-G_r|\le dR_j$ on $E_{j,r,s}$. Applying
\eqref{eq:power-separated-block}, we conclude that
\[
 |E_{j,r,s}\cap I|\ll_{m,n,\tau}R_j+1.
\]
There are $O_\tau(x/R_{j+1})$ intervals in the partition , since $R_{j+1}\le x$.
Using $R_j=\eta^{1/d}R_{j+1}$ and
$R_{j+1}\ge\eta^{(d-1)/d}x$, we obtain
\[
 |E_{j,r,s}|
 \ll_{m,n,\tau}\frac{x}{R_{j+1}}(R_j+1)
 \le \eta^{1/d}x+\eta^{-(d-1)/d}.
\]
Summing over the finitely many choices in Step~2 gives
\[
 \left|\left\{u\in[0,x]:|g_i(u)|\le\eta x\right\}\right|
 \ll_{m,n,\tau}\eta^{1/d}x+\eta^{-(d-1)/d}.
\]
Dividing by $x$, then letting $x\to\infty$ and
$\eta\searrow0$, proves \eqref{eq:power-single-threshold-density}.
Taking the union over $i=1,\ldots,d$ proves
\eqref{eq:power-threshold-density}. The last assertion holds trivially.
\end{proof}

\subsection{Proof of Theorem~\ref{thm:near-powers}}

Put
\[
 B(x):=-\tau\int_0^x\frac{\partial F}{\partial y}(u,\tau u)\,du,
\]
setting the derivative to zero where it is undefined.

\begin{lem}\label{lem:power-pressure}
Suppose that $\tau\notin\mathcal E_{m,n}$. Then
\begin{equation}\label{eq:power-pressure-integral}
 \mathcal P_{A,\psi_\tau}(L,c)
 =\liminf_{x\to\infty}\frac{B(x)}{\tau x}+O_d(L^{-1})
\end{equation}
for every $L,c>0$, with an implied
constant independent of $c$.
\end{lem}

\begin{proof}
Fix $L,c>0$, put $w:=\log c$, and choose
$R>|w|$. For $\psi=\psi_\tau$, equation~\eqref{eq:1.11} gives
$T_j=\e^{jL/\tau}$. Put $x_j=jL/\tau$ and consider the rectangles
\begin{equation}\label{eq:profile-rectangles}
 [x_j,x_{j+1}]\times[jL-R,(j+1)L+R],\qquad j\ge0.
\end{equation}
Call a mesh interval $[x_j,x_{j+1}]$ bad if some $|g_i|$ is at most
$H_d+2(L+R+1)\max\{1,\tau,\tau^{-1}\}$ at some point of it.
Lipschitz continuity and Lemma~\ref{lem:power-threshold} show that
there are $o(N)$ bad intervals among the first $N$: the union of
such intervals is contained in a fixed, slightly wider threshold
strip. Call each remaining mesh interval good. A \emph{good block}
is a maximal union of consecutive good mesh intervals. Among the
first $N$ mesh intervals, the good intervals form at most one
more block than there are bad intervals, hence $o(N)+1$ good blocks.

Fix a good block and let $r'$ and $\Gamma_{r'}$ be defined by
\eqref{eq:power-small-lattice} at $(x,\tau x)$; namely, $r'=\#\{i:g_i(x)<0\}$. The rank $r'$ is constant throughout the block, since all the $|g_i|$ are separated from zero. Thus $\Gamma_{r'}$ is fixed by Lemma~\ref{lem:power-gaps}.

If $(x,y)$ belongs to a rectangle  \eqref{eq:profile-rectangles}, then $|y-\tau x|\le L+R$.
Choose a rectangle corresponding to an interval from our fixed good block, and consider the lattice $a_{x,y} \Lambda_A$ instead of $a_{x, \tau x} \Lambda_A$. Application of $a_{x,y}a_{x, \tau x}^{-1}$ multiplies each vector norm by a factor between $\e^{-(L+R)}$ and $\e^{L+R}$. Since the interval to which $x$ belongs is good, the first $r'$ minimizing vectors remain shorter than one and the number of minima less than one remains $r'$.
Thus the rank and lattice in \eqref{eq:power-small-lattice}
remain $r'$ and $\Gamma_{r'}$, and all logarithmic minima remain
outside $[-H_d,H_d]$. Lemma~\ref{lem:power-maximizer} and
\eqref{eq:power-exterior-profile} therefore give
\[
 F(x,y)= -\log|a_{x,y}\omega_{\Gamma_{r'}}| = - \max_{a+b=r'}\{c_a+ay-bx\} =r'x-\max_a\{c_a+a(x+y)\}.
\]
The affine branch of the maximum changes at at most $d$ lines of form $x + y = const$.
On the $j$th rectangle, $x+y$ ranges from
$(1+\tau^{-1})jL-R$ to $(1+\tau^{-1})(j+1)L+R$.
Each branch change line $x+y = const$ therefore meets only $O_{L,R,\tau}(1)$ rectangles.
Apart from these $O_{L,R,\tau,d}(1)$ exceptions per block (summing to $o(N)$ exceptional rectangles overall), $F$ has
constant vertical slope on each rectangle. Including all rectangles
over bad intervals, the total number of exceptional rectangles is $o(N)$.

On a nonexceptional rectangle with vertical slope $-a_j$, 
\begin{equation}\label{eq:power-pressure-transition}
\begin{gathered}
 \log\frac{\mathcal N_A(c\psi_\tau(T_j),T_{j+1})}
              {\mathcal N_A(c\psi_\tau(T_{j+1}),T_{j+1})}
 ={}F(x_{j+1},jL-w)\\
 -F(x_{j+1},(j+1)L-w)+O_d(1)
 ={}B(x_{j+1})-B(x_j)+O_d(1).
\end{gathered}
\end{equation}
Indeed: both arguments of $F$ belong to the rectangle, so the difference of these terms equals $a_jL$, which is$B(x_{j+1})-B(x_j)$.
The first equality uses
\eqref{eq:power-profile-comparison}.

\smallskip

On an exceptional rectangle, the vertical increment of $F$ and
the increment of $B$ both lie in $[0,mL]$, by the slope bound in
\eqref{eq:power-profile-comparison}. The $o(N)$ exceptions therefore
contribute a total error $o(NL)$. Summing
\eqref{eq:power-pressure-transition}, dividing by $NL$, and using
\eqref{eq:4.19}, we obtain
\[
 \mathcal P_{A,\psi_\tau}(L,c)
 =\liminf_{N\to\infty}\frac{B(x_N)}{NL}+O_d(L^{-1}).
\]
The implied constant is independent of $c$, since it comes from
the uniform error in \eqref{eq:power-profile-comparison}. 

For $x_N\le x<x_{N+1}$, one has
\[
 0\le B(x)-B(x_N)\le mL,\qquad 0\le\tau x-NL<L.
\]
Therefore, $\left| \frac{B(x)}{\tau x} - \frac{B(x_N)}{NL} \right| = O(N^{-1})$ and replacing $x_N$ by $x$ does not change the liminf. This proves \eqref{eq:power-pressure-integral}.
\end{proof}

\begin{lem}\label{lem:power-rank-identity}
Suppose that $\tau\notin\mathcal E_{m,n}$. Then
\begin{equation}\label{eq:power-rank-identity}
 \frac{B(x)}{\tau}
 =\frac{\displaystyle\int_0^x r_{A,\tau}(\e^u)\,du
             -\log\mathcal N_A(\e^{-\tau x},\e^x)}{1+\tau}
   +o(x).
\end{equation}
\end{lem}

\begin{proof}
Outside of the exceptional set $E = \{u:\min_i|g_i(u)|\le H_d\}$,
take $r'$ and $\Gamma_{r'}$ from \eqref{eq:power-small-lattice}
at $(u,\tau u)$.
By definition of $r_{A,\tau}$ and \eqref{eq:power-normalized-minima},
\begin{equation}\label{eq:r'_through_r}
r'(u)=\#\{i:g_i(u)<0\}=r_{A,\tau}(\e^u),
\end{equation}
and
Lemma~\ref{lem:power-maximizer} gives
\[
F(u,\tau u)=-\log|a_{u,\tau u}\omega_{\Gamma_{r'}}|.
\]
The lattice $\Gamma_{r'}$ is locally fixed by
Lemma~\ref{lem:power-gaps}.
Its affine branch changes in \eqref{eq:power-exterior-profile} within a connected component of $[0, x] \setminus E$ occur at finite set of fixed values of $u+y$, so the line $y=\tau u$ meets them only at isolated points.
Away from these points, \eqref{eq:power-exterior-profile} gives
\begin{equation}\label{eq:equal-deriv}
\frac d{du}F(u,\tau u)
=r'(u)+(1+\tau)\frac{\partial F}{\partial y}(u,\tau u).
\end{equation}
By Lemma~\ref{lem:power-threshold}, $|E\cap[0,x]|=o(x)$.

By \eqref{eq:r'_through_r} and
\eqref{eq:power-profile-comparison},
\[
0\le r_{A,\tau}(\e^u)\le d,
\qquad
0\le-\frac{\partial F}{\partial y}(u,\tau u)\le m,
\]
with the assumed zero extension of the partial derivative.
Moreover, Lemma~\ref{lem:power-profile} and
\eqref{eq:power-exterior-profile} show that $u\mapsto F(u,\tau u)$
is locally absolutely continuous, with
\[
\left|\frac d{du}F(u,\tau u)\right|\le n+m\tau
\quad\text{almost everywhere}.
\]
Thus the error produced by changing the integration domain from $[0, x]$ to $[0, x] \setminus E$ is $O_{m,n,\tau}(|E\cap[0,x]|) = o(x)$.
Integrating \eqref{eq:equal-deriv} and using \eqref{eq:r'_through_r} gives
\[
\begin{aligned}
F(x,\tau x)-F(0,0)
&=\int_0^x\frac d{du}F(u,\tau u)\,du\\
&=\int_0^x r_{A,\tau}(\e^u)\,du
 -\frac{1+\tau}{\tau}B(x)+o(x),
\end{aligned}
\]
where the last equality uses the definition of $B$ and
Lemma~\ref{lem:power-threshold}.
Using \eqref{eq:power-profile-comparison} and absorbing the constant
$F(0,0)$ into $o(x)$, we obtain \eqref{eq:power-rank-identity}.
\end{proof}

\begin{proof}[Proof of Theorem~\ref{thm:near-powers}]
First let $\psi=\psi_\tau$. The error in
\eqref{eq:power-pressure-integral} is uniform in $c$, so
Theorem~\ref{thm:main} proves $c$-invariance and identifies the
common dimension with the liminf there. Formula
\eqref{eq:power-rank-identity}, with $T=\e^x$, gives \eqref{eq:power-rank-dimension}.

We next prove that this dimension is continuous in $\tau$ away
from $\mathcal E_{m,n}$. Give the normalized minima an additional
superscript indicating the exponent. Since \(a_{x,\sigma x} =\operatorname{diag}\bigl(\mathrm e^{(\sigma-\tau)x}I_m,I_n\bigr) a_{x,\tau x}\),
\begin{equation}\label{eq:minima-Lips}
 |g_i^\sigma(x)-g_i^\tau(x)|\le|\sigma-\tau|x.
\end{equation}
Moreover, \eqref{eq:2.1} and
\eqref{eq:power-normalized-minima} give
\[
\log\mathcal N_A(\e^{-\tau x},\e^x)
=\sum_{i=1}^d\max\{-g_i^\tau(x),0\},
\]
thus by \eqref{eq:minima-Lips}
 $$
 |\log\mathcal N_A(\e^{-\sigma x},\e^x)
       -\log\mathcal N_A(\e^{-\tau x},\e^x)|
 \le d|\sigma-\tau|x.
$$

The quantities $r'(u)$ for $\sigma$ and $\tau$, and therefore (see \eqref{eq:r'_through_r}) the quantities $r_{A,\sigma}(\e^u)$ and $r_{A,\tau}(\e^u)$, can differ only where
$\min_i|g_i^\tau(u)|\le|\sigma-\tau|x$, for $0\le u\le x$.
Since both lie in $[0,d]$, Lemma~\ref{lem:power-threshold}
 gives
\[
\lim_{\sigma\to\tau}\limsup_{x\to\infty}
\frac1x\left|
\int_0^x
\bigl(r_{A,\sigma}(\e^u)-r_{A,\tau}(\e^u)\bigr)\,du
\right|=0.
\]
Therefore, continuity follows directly from \eqref{eq:power-rank-dimension}.

Finally, suppose that $-\log\psi(T)/\log T\to\tau$. For every
$c>0$ and arbitrarily small $\varepsilon>0$, eventually
\[
 T^{-(\tau+\varepsilon)}\le c\psi(T)
       \le T^{-(\tau-\varepsilon)}.
\]
The corresponding inclusions of approximation sets, followed by
the continuity just proved as $\varepsilon\searrow0$, prove the "Moreover" assertion.
\end{proof}

\subsection{Uniform Diophantine exponents}

Proposition~\ref{prop:no-drop} shows that intersecting constant
rescalings of an approximating function causes no dimension drop.
The next lemma gives the analogous conclusion when power exponents increase to a fixed value; the construction follows the idea in Section \ref{subsec:no-drop}.

\begin{lem}\label{lem:no-drop-powers}
For every $A\in\Mat_{m,n}$ and every $\tau>0$,
\begin{equation}\label{eq:no-drop-powers}
 \dim_H\left(\bigcap_{0<\sigma<\tau}\Dhat_A(\psi_\sigma)\right)
 =\inf_{0<\sigma<\tau}\dim_H\Dhat_A(\psi_\sigma).
\end{equation}
\end{lem}

\begin{proof}
Put $D:=\inf_{0<\sigma<\tau}\dim_H\Dhat_A(\psi_\sigma)$.
The left-hand side does not exceed the right-hand side by inclusion, so we only need to prove the reverse bound when $D>0$.

\smallskip
\noindent\emph{Step 1: choose transition averages with a common lower bound.}
Fix $L>0$ and $0<\sigma_1<\tau$, choose a sequence $\sigma_p\nearrow\tau$,
and put $T_j=\e^{jL}$ and
\[
 a_j^{(p)}:=\frac1L\log
 \frac{\mathcal N_A(T_j^{-\sigma_p},T_{j+1})}
      {\mathcal N_A(T_{j+1}^{-\sigma_p},T_{j+1})},
 \qquad j\ge0.
\]
By Lemma~\ref{lem:vertical},
$0\le a_j^{(p)}\le m\tau+L^{-1}\log K_d$.
Proposition~\ref{prop:upper} gives \eqref{eq:4.23} with
$d_0=\sigma_1D-L^{-1}\log C_d$, after changing each power on a
bounded initial interval to meet the smallness assumption in
Section~\ref{sec:CD}. Lemma~\ref{lem:diagonal} therefore gives
a nondecreasing sequence $p_j\to\infty$ such that
\begin{equation}\label{eq:exponent-diagonal-average}
 \liminf_{N\to\infty}\frac1N
       \sum_{j=0}^{N-1}a_j^{(p_j)}
 \ge \sigma_1D-\frac{\log C_d}{L}.
\end{equation}

\smallskip
\noindent\emph{Step 2: construct a Cantor set with the required dimension bound.}
Put $R_j:=T_j^{-\sigma_{p_j}}$ and choose a continuous strictly
decreasing $\varphi$ with $\varphi(T_j)=R_j$ for all sufficiently
large $j$ and $\varphi(1)<1/12$.
We use Lemmata~\ref{lem:lower-transition} and~\ref{lem:tree}
with $\varphi$ in place of $\psi$ and $k$ sufficiently large,
retaining their notation $\mathbf G_\infty^{(k)}$ and $b_j^{(k)}$. Since $p_j$ is nondecreasing,
$R_j/R_{j+1}\ge\e^{\sigma_1L}$.
Put $\eta_L:=\min\{\e^L,\e^{\sigma_1L}\}>1$; then \eqref{eq:2.7} holds for large enough $j$. Put also $c_*(L):=c_d\left(\frac{\eta_L-1}{\eta_L}\right)^d$. Note that Lemma \ref{lem:lower-transition} does not require \eqref{eq:2.5} to hold, so we can use it in the estimate below. One  has
$$
\begin{gathered}
 b_j^{(k)}
 \underset{\eqref{eq:3.25}}{\ge}
 c_*(L)
 \frac{\mathcal N_A(R_j,T_{j+1})}
      {\mathcal N_A(R_{j+1},T_{j+1})}
\underset{\substack{
      \eqref{eq:4.2}, \\
      R_{j+1}\le T_{j+1}^{-\sigma_{p_j}}
   }}{\ge}
 c_*(L)
 \frac{\mathcal N_A(T_{j}^{-\sigma_{p_j}},T_{j+1})}
      {\mathcal N_A(T_{j+1}^{-\sigma_{p_j}},T_{j+1})}.
\end{gathered}
$$
Taking logarithms and retaining the definition of $a_j^{(p)}$ gives
\[
 \log b_j^{(k)}
 \ge La_j^{(p_j)}+\log c_*(L).
\]

Moreover,
\[ \begin{gathered}
\frac{- \log \varphi(T_{j+1})}{- \log \varphi(T_{j})} = \frac{-\log R_{j+1}}{-\log R_j}
 =\frac{(j+1)\sigma_{p_{j+1}}}{j\sigma_{p_j}}
 \longrightarrow1, 
\end{gathered} \]
so \eqref{eq:tree-log-ratio} holds. Note also that $$-\log \varphi(T_N) =  -\log R_N=NL\sigma_{p_N} \le NL \tau.$$ Lemma~\ref{lem:tree} and \eqref{eq:exponent-diagonal-average}
therefore give
\[
 \dim_H\mathbf G_\infty^{(k)}
 \ge \frac{\sigma_1}{\tau}D
 -\frac{\log C_d-\log c_*(L)}{\tau L}.
\]

\smallskip
\noindent\emph{Step 3: verify inclusion in every set in the intersection.}
Fix $0<\sigma<\tau$. Since $j\sigma_{p_j}/(j+1)\to\tau$,
one has $R_j\le T_{j+1}^{-\sigma}$ for all sufficiently large $j$.
For $\bb\in\mathbf G_\infty^{(k)}$ and $T_j\le T<T_{j+1}$,
its level-$j$ ball provides $\q_j\in\Z^n$ satisfying
\[
 |\q_j|\le T_j\le T,
 \qquad
 |A\q_j-\bb|_{\Z^m}\le R_j\le T_{j+1}^{-\sigma}\le T^{-\sigma}.
\]
Thus $\mathbf G_\infty^{(k)}\subseteq
\bigcap_{0<\sigma<\tau}\Dhat_A(\psi_\sigma)$.
Letting $L\to\infty$, so that $c_*(L)\to c_d>0$, and then
$\sigma_1\nearrow\tau$ proves \eqref{eq:no-drop-powers}.
\end{proof}

\begin{proof}[Proof of Proposition~\ref{prop:uniform-exponent-levels}]
By \eqref{eq:power-pressure-ordering}, applied with
$\psi(T)=T^{-1}$, the left limits of $P_-$ and $P_+$ at every $\tau>0$
coincide, as do their right limits. The same limits are obtained
from $\dim_H\Dhat_A(\psi_\sigma)$, which lies between them.
Put
\[
 G_\tau:=\{\bb:\widehat\omega(A,\bb)\ge\tau\},
 \qquad
 H_\tau:=\{\bb:\widehat\omega(A,\bb)>\tau\}.
\]
The identities
\[
 G_\tau=\bigcap_{0<\sigma<\tau}\Dhat_A(\psi_\sigma),
 \qquad
 H_\tau=\bigcup_{\sigma>\tau}\Dhat_A(\psi_\sigma)
\]
hold, and both intersections and unions may be restricted to
rational $\sigma$. Lemma~\ref{lem:no-drop-powers} and countable stability, respectively,
give
\[
 \dim_H G_\tau
 =\inf_{0<\sigma<\tau}\dim_H\Dhat_A(\psi_\sigma),
 \qquad
 \dim_H H_\tau
 =\sup_{\sigma>\tau}\dim_H\Dhat_A(\psi_\sigma).
\]
This proves \eqref{eq:exponent-superlevel-dimension} and \eqref{eq:exponent-strict-superlevel-dimension}.

Suppose $E_\tau:=G_\tau\setminus H_\tau$ is nonempty and choose
$\bb_0\in E_\tau$. One can directly check that $\bb_0+H_\tau\subseteq E_\tau$, thus
$\dim_H H_\tau\le\dim_H E_\tau$. Since
$G_\tau=E_\tau\cup H_\tau$, countable stability proves
$\dim_H E_\tau=\dim_H G_\tau$, as claimed.
\end{proof}

\begin{proof}[Proof of Corollary~\ref{cor:noncritical-exponent-levels}]
By Theorem~\ref{thm:near-powers}, the common dimension for powers
is continuous at every noncritical exponent.
Thus both one-sided limits in
Proposition~\ref{prop:uniform-exponent-levels} equal
$\dim_H\Dhat_A(\psi_\tau)$, completing the proof.
\end{proof}

\section{Proofs of remaining theorems.}\label{sec:remaining-theorems}

\subsection{Proof of Theorem \ref{thm:zero-full}}
\begin{lem}\label{lem:1.9}
Let $\mathcal E\subseteq H_A$ be a $\Gamma_A$-invariant Borel set, and let $f$ be a dimension function. Then $\mathcal E$ has either zero or locally full $\mathcal H^f$-measure in $H_A$.
\end{lem}

\begin{proof}
Put $h:=\dim H_A$. If $h=0$, then $H_A$ is finite (a zero-dimensional manifold is discrete and a compact discrete set is finite). Thus $\mathcal H^f(H_A)=0$. We may therefore assume that $h>0$ and $\mathcal H^f(\mathcal E)>0$.

If $\mathcal H^f(\mathcal E)=\infty$, the translates by $\Gamma_A$ of any nonempty relatively open ball in
$H_A$ cover $H_A$, and compactness gives a finite subcover. Since
$\mathcal E$ is $\Gamma_A$-invariant, it follows that
\[
\mathcal H^f(\mathcal E\cap B)=\infty = \mathcal H^f(B)
\]
for every such ball $B$.

Suppose now that $0<\mathcal H^f(\mathcal E)<\infty$; it remains to prove that $\mathcal E$ has locally full $\mathcal H^f$-measure.
Consider the measure
\begin{equation}\label{eq:measure_induced_E}
 F\longmapsto \mathcal H^f(F\cap\mathcal E);
\end{equation}
easy to see that it is well-defined finite Borel measure on $H_A$. This measure is invariant under the dense subgroup $\Gamma_A$ of $H_A$ and thus (see \cite[Lemma 1.1.3]{KSS02}) under $H_A$ itself. Therefore, it is a constant multiple of Haar measure on $H_A$. Since it is concentrated on $\mathcal E$, $\mathcal E$ is Haar-conull.

It is known that compact $h$-dimensional Riemannian manifolds are $h$-Ahlfors regular. The left-invariant Riemannian metric on $H_A$ induced from $\T^m$ is locally bi-Lipschitz equivalent to supremum metric and the corresponding Riemannian volume is left-invariant, thus up to a constant coincides with the Haar measure. Therefore $(H_A, \dist_{\infty}, \mathcal L_A)$ is $h$-Ahlfors regular.
By $h$-Ahlfors regularity,
$0<\mathcal H^h(H_A)<\infty$.
Since translations are isometries, the restriction of
$\mathcal H^h$ to $H_A$ is translation-invariant, and hence, by uniqueness of Haar measure, is a positive constant multiple of $\mathcal L_A$, and therefore of the measure \eqref{eq:measure_induced_E}. We conclude that $\mathcal H^h(\mathcal E)>0$. 

If $\liminf_{r\to0}f(r)/r^h=\infty$, then, for every $C>0$, the inequality $f(r)\ge Cr^h$ for all sufficiently small $r$ gives
\[
 \mathcal H^f(\mathcal E)\ge C\mathcal H^h(\mathcal E),
\]
contradicting finiteness. Hence there are $C>0$ and $r_k\searrow0$ such that $f(r_k)\le Cr_k^h$.

For each $r_k$, choose a maximal $r_k/2$-separated set in $H_A$. The balls of radius $r_k/2$ cover $H_A$, and the open balls of radius $r_k/4$ are pairwise disjoint. Lower Ahlfors regularity bounds their number by $O(r_k^{-h})$, and therefore
\[
 \mathcal H^f_{r_k}(H_A)\ll r_k^{-h}f(r_k)\le C.
\]
It follows that $0<\mathcal H^f(H_A)<\infty$.
The restriction of $\mathcal H^f$ to $H_A$ is therefore a positive finite
translation-invariant measure, and hence a constant multiple of
$\mathcal L_A$. Since $\mathcal E$ is Haar-conull, it is also
$\mathcal H^f$-conull. This proves the lemma.
\end{proof}

All the sets under consideration are Borel. Indeed, 
\[
 \Dhat_A(\eta)=\bigcup_{N\in\N}\bigcap_{T\ge N}\bigcup_{\substack{\q\in\Z^n\\|\q|\le T}}
 \{\bb\in\T^m:|A\q-\bb|_{\Z^m}\le\eta(T)\}.
\]
Since $\Dhat_A(\widetilde\psi_{c_1}) \subseteq
\Dhat_A(\widetilde\psi_{c_2})$ when $c_1<c_2$, we can replace the
intersection and union over all $c>0$ in \eqref{eq:1.14} by countable
intersection and union, proving that $\Sing_A(\psi)$ and
$ U_A(\psi)$ are also Borel.

If $\dim H_A=0$, then $H_A$ is finite and the density of $\Gamma_A$ gives
$\Gamma_A=H_A$.  In this case $\Dhat_A(\widetilde\psi_c)=H_A$ for every $c>0$, and all the assertions
are immediate.  We may therefore assume that $\dim H_A>0$.

Let us first prove part \ref{enu:zero-full-amplification}. Fix a
dimension function $f$ and put
\[
 h_0:=\inf\left\{c>0:
 \mathcal H^f\bigl(\Dhat_A(\widetilde\psi_c)\bigr)>0\right\},
\]
where the infimum of the empty set is $+\infty$. By monotonicity,
\[
 \mathcal H^f\bigl(\Dhat_A(\widetilde\psi_b)\bigr)=0
 \qquad(0<b<h_0).
\]
Now let $b>h_0$. Choose $a$ such that $0<a<b$ and
\[
 \mathcal H^f\bigl(\Dhat_A(\widetilde\psi_a)\bigr)>0.
\]
Pick $\boldsymbol\gamma = A {\bf r} \in\Gamma_A$. Suppose $\bb \in \Dhat_A(\widetilde\psi_a) + \boldsymbol\gamma$; then $\bb - A {\bf r} \in \Dhat_A(\widetilde\psi_a)$, and by definition (it requires solvability of \eqref{eq:1.2} with $\psi$ replaced by $\widetilde\psi_a$, $\bb$ replaced by $\bb - A {\bf r}$ and $T$ replaced by $\frac{a}{b}T$) the system 
\begin{equation}\label{eq:1.2_prime}
 |A(\q + {\bf r})-\p-\bb|\le a\psi\left(\frac{T}{b}\right),\qquad |\q|\le \frac{a}{b}T
\end{equation}
has a solution $(\p,\q)\in\Z^m\times\Z^n$ for any $T$ large enough. Since $a < b$, increasing $T$ if necessary, we see that \eqref{eq:1.2_prime} yields solvability of
\begin{equation*}
 |A(\q + {\bf r})-\p-\bb|\le b\psi\left(\frac{T}{b}\right),\qquad |\q + {\bf r}|\le T
\end{equation*}
in $(\p,\q)\in\Z^m\times\Z^n$ for any $T$ large enough. By definition it means that $\bb \in \Dhat_A(\widetilde\psi_b)$, hence

\begin{equation*}\label{eq:orbit-amplification}
 \Dhat_A(\widetilde\psi_a)+\boldsymbol\gamma\subseteq \Dhat_A(\widetilde\psi_b).
\end{equation*}

Put
$$
 \mathcal E:=\bigcup_{\boldsymbol\gamma\in\Gamma_A}
 \bigl(\Dhat_A(\widetilde\psi_a)+\boldsymbol\gamma\bigr).
$$
This is a $\Gamma_A$-invariant Borel set, and
$\mathcal E\subseteq \Dhat_A(\widetilde\psi_b)$. 

Since $\mathbf0\in\Gamma_A$, the set $\mathcal E$ contains
$\Dhat_A(\widetilde\psi_a)$, and hence
$\mathcal H^f(\mathcal E)>0$. Lemma \ref{lem:1.9} shows that
$\mathcal E$, and hence $\Dhat_A(\widetilde\psi_b)$, has locally full
$\mathcal H^f$-measure in $H_A$. A finite cover of $H_A$ by relative
balls then gives
\[
 \mathcal H^f\bigl(\Dhat_A(\widetilde\psi_b)\bigr)
 =\mathcal H^f(H_A).
\]
Applying the same argument directly to arbitrary $0<a<b$, with
$f(r)=r^{\dim H_A}$, and using the proportionality of
$\mathcal H^{\dim H_A}$ and $\mathcal L_A$ established in the proof of
Lemma \ref{lem:1.9} gives \eqref{eq:Haar-amplification}, completing part
\ref{enu:zero-full-amplification}.

Now we prove part \ref{enu:zero-full-subgroups}. %
Let $a,b>0$, and suppose $\bb_1 \in \Dhat_A(\widetilde\psi_a), \ \bb_2 \in \Dhat_A(\widetilde\psi_b)$. The systems of inequalities 
\begin{equation*}
 |A\q_1-\p_1-\bb_1|\le a\psi\left(\frac{T}{a+b}\right),\qquad |\q_1|\le \frac{a}{a+b}T
\end{equation*}
and 
\begin{equation*}
 |A\q_2-\p_2-\bb_2|\le b\psi\left(\frac{T}{a+b}\right),\qquad |\q_2|\le \frac{b}{a+b}T
\end{equation*}
have solutions $(\p_1,\q_1), (\p_2,\q_2) \in\Z^m\times\Z^n$ for any $T$ large enough, thus 
\begin{equation*}
 |A(\q_1+ \q_2)-(\p_1+\p_2)-(\bb_1+ \bb_2)|\le (a+b)\psi\left(\frac{T}{a+b}\right),\qquad |\q_1 + \q_2|\le T
\end{equation*}
and 
\begin{equation*}
 |A(\q_1- \q_2)-(\p_1-\p_2)-(\bb_1- \bb_2)|\le (a+b)\psi\left(\frac{T}{a+b}\right),\qquad |\q_1 - \q_2|\le T
\end{equation*}
which implies that $\bb_1 \pm \bb_2 \in \Dhat_A(\widetilde\psi_{a+b})$ and
\begin{equation}\label{eq:1.19}
 \Dhat_A(\widetilde\psi_a)\pm
 \Dhat_A(\widetilde\psi_b)
 \subseteq\Dhat_A(\widetilde\psi_{a+b}).
\end{equation}
It follows from \eqref{eq:1.14} and \eqref{eq:1.19} that $\Sing_A(\psi)$ and
$ U_A(\psi)$ are Borel subgroups of $H_A$ containing $\Gamma_A$.

Both sets are therefore $\Gamma_A$-invariant, and Lemma \ref{lem:1.9} gives the required zero-full alternative.

It remains to prove part \ref{enu:zero-full-exact}. The set
$\Dhat_A(\widetilde\psi_a)$ is Haar-conull, so
$\Dhat_A(\widetilde\psi_a)-\Dhat_A(\widetilde\psi_a)=H_A$. Indeed: let
$\boldsymbol\gamma \in H_A$. The sets
$\Dhat_A(\widetilde\psi_a)$ and
$\Dhat_A(\widetilde\psi_a)+\boldsymbol\gamma$ are both conull, hence
their intersection is nonempty and contains a vector $\bb_1$. Since
$\bb_1 \in \Dhat_A(\widetilde\psi_a)+\boldsymbol\gamma$, one has
$\boldsymbol\gamma = \bb_1 -\bb_2$ for some
$\bb_2 \in \Dhat_A(\widetilde\psi_a)$. Equation \eqref{eq:1.19} gives
\[
 H_A=\Dhat_A(\widetilde\psi_a)-\Dhat_A(\widetilde\psi_a)
 \subseteq \Dhat_A(\widetilde\psi_{2a}),
\]
which proves \eqref{eq:exact-amplification}.

\subsection{Proofs of Propositions \ref{prop:endpoint-representations} and
\ref{prop:dimension-continuity}}

\begin{proof}[Proof of Proposition \ref{prop:endpoint-representations}]
We first prove part \ref{enu:upper-intersection}.
If $\bb\in U_A(\psi)$, then
$\bb\in\Dhat_A(c\psi)$ for some $c>0$ by
\eqref{eq:classical-endpoints}. For every $\varphi$ satisfying $\psi=o(\varphi)$ one has $c\psi(T)\le\varphi(T)$ for all
sufficiently large $T$, and hence $\bb\in\Dhat_A(\varphi)$.

Conversely, put
\[
 \rho_{\bb}(T):=
 \min_{|\q|\le T}|A\q-\bb|_{\Z^m}.
\]
If $\bb\notin U_A(\psi)$, then, again by
\eqref{eq:classical-endpoints}, $\rho_{\bb}(T)/\psi(T)$ is not eventually
bounded. Choosing a sufficiently sparse sequence $\{ T_k \}$, we can guarantee that
\[
 \rho_{\bb}(T_k)>k\psi(T_k),
 \qquad
 \psi(T_{k+1})<{\tfrac14}\psi(T_k).
\]
Since $(k+1)/k\le2$, Lemma \ref{lem:interpolation} applies with
$c_k=k$ and $\delta=1/2$. It gives a lower quasimultiplicative
function $\varphi$ such that $\psi=o(\varphi)$ and
\[
 \varphi(T_k)=k\psi(T_k)<\rho_{\bb}(T_k).
\]
Hence $\bb\notin\Dhat_A(\varphi)$, proving \eqref{eq:1.28}, including
the assertion about the restricted intersection.

We now prove part \ref{enu:lower-union}. If $\vartheta=o(\psi)$, then
$\vartheta(T)\le c\psi(T)$ for every $c>0$ and all sufficiently large
$T$. Hence \eqref{eq:classical-endpoints} gives
\[
 \Dhat_A(\vartheta)
 \subseteq\bigcap_{c>0}\Dhat_A(c\psi)=\Sing_A(\psi).
\]

Conversely, let $\bb\in\Sing_A(\psi)$; then $\frac{\rho_{\bb}(T)}{\psi(T)}\to 0$.
Choose $T_k\to\infty$ so that
\[
 \rho_{\bb}(T)\le\frac{\psi(T)}{k+1}
 \quad\text{for every }T\ge T_k,
 \qquad
 \psi(T_{k+1})<\tfrac14\psi(T_k).
\]
Lemma \ref{lem:interpolation}, with $c_k=k$ and $\delta=1/2$, gives
a lower quasimultiplicative function $\vartheta$ such that
$\vartheta(T_k)=\psi(T_k)/k$ and $\vartheta=o(\psi)$.
Since $\psi/\vartheta$ is nondecreasing, for $T\in[T_k,T_{k+1}]$ one has
\[
 \rho_{\bb}(T)
 \le\frac{\psi(T)}{k+1}
 \le\vartheta(T).
\]
Thus $\bb\in\Dhat_A(\vartheta)$, while
$\vartheta=o(\psi)$, proving \eqref{eq:sing-union}.
Since $\vartheta$ is lower quasimultiplicative, the restricted union
is unchanged as well.

\end{proof}

\begin{proof}[Proof of Proposition \ref{prop:dimension-continuity}]
Proposition \ref{prop:endpoint-representations} gives
\[
\begin{split}
 \sup_{\vartheta=o(\psi)}\dim_H\Dhat_A(\vartheta)
 \le\dim_H\Sing_A(\psi), \quad 
 \dim_H U_A(\psi)
\le\inf_{\psi=o(\varphi)}\dim_H\Dhat_A(\varphi).
\end{split}
\]
We prove the reverse inequalities using lower quasimultiplicative
functions, which will also prove the last assertion. Once the
endpoint equalities are established, the two intermediate
inequalities in \eqref{eq:dimension-continuity} follow from
\[
 \Sing_A(\psi)\subseteq\Dhat_A(\psi)\subseteq U_A(\psi).
\]

For the first equality, the case $\dim_H\Sing_A(\psi)=0$ is trivial.
Otherwise, fix $0<s<\dim_H\Sing_A(\psi)$.
Let $\vartheta$ be the lower quasimultiplicative function constructed in the proof of Proposition \ref{prop:no-drop}; for convenience, we multiply it by the constant $\rho_L$. Due to \eqref{eq:4.33} and \eqref{eq:4.40}, $\vartheta=o(\psi)$ and
\[
 \dim_H\Dhat_A(\vartheta)>s.
\]
Letting $s\nearrow\dim_H\Sing_A(\psi)$ proves the reverse inequality and finished the proof of the first equality, including the version for lower quasimultiplicative functions.

The proof of the last equality is similar to the proof of Proposition \ref{prop:no-drop}. Fix $\varepsilon>0$. By Theorem \ref{thm:main}, $$\dim_H U_A(\psi) = \sup_{c>0}\dim_H\Dhat_A(c\psi).$$ By Corollary \ref{cor:endpoints} then, ${\sup_{c>0}\mathcal P_{A,\psi}(L,c)}
 -\dim_H U_A(\psi) \to 0$ as $L \to \infty$.  Therefore, we can choose $L$ so large that
\begin{equation}\label{eq:upper-continuity-choice}
 \sup_{c>0}\mathcal P_{A,\psi}(L,c)+\frac{\log C_d}{L}<\dim_H U_A(\psi)+\varepsilon, 
\end{equation}
where $C_d$ was introduced in Proposition \ref{prop:upper}.
Choose $T_j$ as in \eqref{eq:1.11} and put
\[
 a_j^{(p)}:=
 \frac1L\log
 \frac{\mathcal N_A(\e^p\psi(T_j),T_{j+1})}
      {\mathcal N_A(\e^p\psi(T_{j+1}),T_{j+1})},
 \qquad p\in\N,\quad j\ge0.
\]
By Lemma \ref{lem:vertical}, these numbers belong
to $[0,M]$ for a constant $M = M(L)$ independent of $j$ and $p$. By \eqref{eq:4.19},
\[
 \liminf_{N\to\infty}\frac1N
 \sum_{j=0}^{N-1}a_j^{(p)}
 =\mathcal P_{A,\psi}(L,\e^p)\le \sup_{c>0}\mathcal P_{A,\psi}(L,c).
\]

By Lemma \ref{lem:diagonal},
with $d_0=\sup_{c>0}\mathcal P_{A,\psi}(L,c)$ and the inequalities
in \eqref{eq:4.23} and \eqref{eq:4.25} reversed, there exists a sequence $\{p_j\}$ satisfying
\begin{equation}\label{eq:upper-diagonal-sequence}
 p_{j+1}-p_j\in\{0,1\},\qquad
 p_j\to\infty,\qquad \frac{p_j}{j}\to0,
\end{equation}
and
\begin{equation}\label{eq:upper-diagonal-average}
 \liminf_{N\to\infty}\frac1N
 \sum_{j=0}^{N-1}a_j^{(p_j)}\le \sup_{c>0}\mathcal P_{A,\psi}(L,c).
\end{equation}

By \eqref{eq:1.11} and \eqref{eq:upper-diagonal-sequence},
Lemma \ref{lem:interpolation} applies with $c_j=\e^{p_j}$ and
$\delta=L^{-1}$. It gives a lower quasimultiplicative function
$\varphi$ such that $\psi=o(\varphi)$ and
\[
 \varphi(T_j)=\e^{p_j}\psi(T_j).
\]
The standing smallness assumption in Section \ref{sec:CD} may be
arranged by modifying $\varphi$ on a bounded initial interval;
only values at finitely many $T_j$ are affected.

For all sufficiently large $j$, monotonicity of
$R\mapsto\mathcal N_A(R,Q)$ and $p_{j+1}\ge p_j$ yield that
\begin{equation}\label{eq:upper-variable-pressure}
 \log
 \frac{\mathcal N_A(\varphi(T_j),T_{j+1})}
      {\mathcal N_A(\varphi(T_{j+1}),T_{j+1})}
 \le L a_j^{(p_j)}.
\end{equation}
Since $-\log\varphi(T_N)=NL-p_N$ for all sufficiently large $N$, by Proposition \ref{prop:upper} we have
\[
\begin{aligned}
 \dim_H\Dhat_A(\varphi)
 &\underset{\mathclap{\substack{
      \eqref{eq:upper-variable-pressure}}}}{\le} \
 \liminf_{N\to\infty}
 \frac{(N-1)\log C_d+
       L\displaystyle\sum_{j=1}^{N-1}a_j^{(p_j)}}
      {NL-p_N}\\[1ex]
 \underset{\mathclap{\eqref{eq:upper-diagonal-sequence}}}{=}
 \frac{\log C_d}{L}
 +\liminf_{N\to\infty}\frac1N
  \sum_{j=0}^{N-1}a_j^{(p_j)}
 &\underset{\mathclap{\eqref{eq:upper-diagonal-average}}}{\le}
\ \sup_{c>0}\mathcal P_{A,\psi}(L,c)
 +\frac{\log C_d}{L}\\[1ex]
 &\underset{\mathclap{\eqref{eq:upper-continuity-choice}}}{<} \
 \dim_H U_A(\psi)+\varepsilon.
\end{aligned}
\]
Letting $\varepsilon\to0$ proves the reverse inequality for the
infimum, including the version for lower quasimultiplicative functions, and completes the proof.
\end{proof}

\section{Proofs of one-dimensional results}
\label{subsec:one-dimensional}

Let $\psi$ be continuous and strictly decreasing to zero, and put
$h(x)=-\log\psi(\e^x)$.  {In the slow regime
\eqref{eq:one-dimensional-slow-hypothesis}, fix $0<\theta<1$ such that
\begin{equation}\label{eq:one-dimensional-slow-bound}
 h(x)\le\theta x
 \qquad\text{for all sufficiently large }x.
\end{equation}
In the fast regime \eqref{eq:one-dimensional-fast-hypothesis}, fix
$\theta>1$ such that
\begin{equation}\label{eq:one-dimensional-fast-bound}
 h(x)\ge\theta x
 \qquad\text{for all sufficiently large }x.
\end{equation}
}

Throughout the proof, we will use the well-known relation
\begin{equation}\label{eq:5.56}
 \frac1{q_{k+1}+q_k}
 <|q_k\alpha|_{\Z}
 <\frac1{q_{k+1}},
 \qquad
 \log|q_k\alpha|_{\Z}^{-1}=\log q_{k+1}+O(1).
\end{equation}

for $k \ge 1$. Set
\[
 F_\alpha(x,y):=\max\left\{
 0,x-y,\sup_{k\ge0}\min\bigl(x-\log q_k,
 \log|q_k\alpha|_{\Z}^{-1}-y\bigr)
 \right\}.
\]

\begin{lem}\label{lem:cf-count}
Uniformly for $x\ge0$ and $y\in\R$, one has
\begin{equation}\label{eq:5.67}
\log \mathcal N_\alpha(\e^{-y},\e^x)=F_\alpha(x,y)+O(1).
\end{equation}
\end{lem}

\begin{proof}
Put $t:=\frac{x+y}{2}$.
By the definition of $\mathcal N_\alpha$ in \eqref{eq:2.1} and the inequality
$\lambda_1(t)\le\lambda_2(t)$,
$$
 \log \mathcal N_\alpha(\e^{-y},\e^x)
 =
 \max\left\{
 0,
 \frac{x-y}{2}-\log\lambda_1(t),
 x-y-\log\bigl(\lambda_1(t)\lambda_2(t)\bigr)
 \right\}.
$$

Recall that $a_t$ was defined in \eqref{eq:1.6} via 
$$
a_t = \begin{pmatrix}
    e^t & \\ & e^{-t}
\end{pmatrix}.
$$
The lattice $a_tu_\alpha\Z^2$ is unimodular, so Minkowski's second
theorem gives $\log\bigl(\lambda_1(t)\lambda_2(t)\bigr)=O(1)$. Consequently,
\begin{equation}\label{eq:cf-profile-first-minimum}
\log \mathcal N_\alpha(\e^{-y},\e^x)
 =
 \max\left\{0,\frac{x-y}{2}-\log\lambda_1(t),x-y\right\}+O(1).
\end{equation}

The definition of the first minimum gives
\[
 \lambda_1(t)
 =\min\left\{
 \e^t,
 \min_{q\ne0}
 \max\left\{
 \e^t|q\alpha|_{\Z},\e^{-t}|q|
 \right\}
 \right\}.
\]
Here the first term comes from $q=0$ and $|p|=1$.
Therefore
\begin{equation}\label{eq:cf-profile-minimum}
\frac{x-y}{2}-\log\lambda_1(t)
 =
 \max\left\{
 -y,
 \sup_{q\ne0}
 \min\bigl(x-\log|q|,
 \log|q\alpha|_{\Z}^{-1}-y\bigr)
 \right\}.
\end{equation}
If $q_k\le|q|<q_{k+1}$, the best-approximation property gives
$|q\alpha|_{\Z}\ge|q_k\alpha|_{\Z}$, and hence
\[
 \min\bigl(x-\log|q|,
 \log|q\alpha|_{\Z}^{-1}-y\bigr)
 \le
 \min\bigl(x-\log q_k,
 \log|q_k\alpha|_{\Z}^{-1}-y\bigr).
\]
Thus the supremum
in \eqref{eq:cf-profile-minimum} can be restricted to the convergent denominators $q_k$.  Finally, $x\ge0$ implies
$-y\le x-y$.  Substituting \eqref{eq:cf-profile-minimum} into
\eqref{eq:cf-profile-first-minimum} proves \eqref{eq:5.67}.
\end{proof}

Put
\begin{equation}\label{eq:a_k_b_k}
 a_k:=\log q_k,
 \qquad
 b_k:=\log|q_k\alpha|_{\Z}^{-1},
\end{equation}
and define
\begin{equation}\label{eq:5.69}
 \sigma(x):=x+h(x),
 \qquad
 \xi_k:=\sigma^{-1}(a_k+b_k).
\end{equation}

Note that
\begin{equation*}\label{eq:5.70}
 \log Y_k=\xi_k.
\end{equation*}

For fixed
$x$, the function $y \mapsto F_{\alpha}(x,y)$ is continuous and locally piecewise affine, with
slopes belonging to $\{0,-1\}$.
In every integral and pointwise sum below, we assign this partial derivative the value zero wherever it is undefined; this convention does not change any integral. Hence, if $y_1<y_2$, then Lemma \ref{lem:cf-count} gives
\begin{equation}\label{eq:5.72}
\begin{split}
 \log
 \frac{\mathcal N_\alpha(\e^{-y_1},\e^x)}
      {\mathcal N_\alpha(\e^{-y_2},\e^x)}
 &=F_\alpha(x,y_1)-F_\alpha(x,y_2)+O(1)\\
 &=-\int_{y_1}^{y_2}
 \frac{\partial F_\alpha}{\partial y}(x,y)\,dy+O(1).
\end{split}
\end{equation}

\smallskip

Let $a_k$ and $b_k$ be as in \eqref{eq:a_k_b_k}.
For $w\in\R$, put
\begin{equation}\label{eq:shifted-xi}
 \xi_k(w):=\sigma^{-1}(a_k+b_k+w).
\end{equation}
When $w=0$, we omit the
argument and write $\xi_k=\xi_k(0)$, matching the definition given in \eqref{eq:5.69}.

\begin{lem}
\label{lem:one-dimensional-branches}
For every $w\in\R$, there exist $k_0=k_0(w)\in\N$ and $x_w>0$
such that the following statements hold.
\begin{enumerate}[label=\textup{(\roman*)}]
\item Suppose that $\theta<1$.  Set
\begin{equation}\label{eq:5.73}
 I_k(w):=\bigl(h^{-1}(a_k+w),\xi_k(w)\bigr).
\end{equation}
For every $k\ge k_0$, this interval is nonempty precisely when
\[
\e^w\psi\bigl(|q_k\alpha|_{\Z}^{-1}\bigr)<q_k^{-1}.
\]
The nonempty intervals $I_k(w)$ with $k\ge k_0$ are pairwise
disjoint, and, for every $x\ge x_w$, one has
\begin{equation}\label{eq:displ_1}
 \frac{\partial F_{\alpha}}{\partial y}(x,h(x)-w)
 =
 \begin{cases}
  0,
  &x\in\displaystyle\bigcup_{k\ge k_0} I_k(w),\\[2mm]
  -1,
  &x\notin\displaystyle\bigcup_{k\ge k_0}\overline{I_k(w)},\\
  \text{undefined},
  &\text{otherwise}.
 \end{cases}
\end{equation}

\item Suppose that $\theta>1$.  Set
\begin{equation}\label{eq:5.83}
 J_k(w):=\bigl(\xi_k(w),h^{-1}(b_k+w)\bigr).
\end{equation}
For every $k\ge k_0$, this interval is nonempty precisely when
\[
 \e^w\psi(q_k)>|q_k\alpha|_{\Z}.
\]
The nonempty intervals $J_k(w)$ with $k\ge k_0$ are pairwise
disjoint, and, for every $x\ge x_w$, one has
\begin{equation}\label{eq:displ_2}
 \frac{\partial F_{\alpha}}{\partial y}(x,h(x)-w)
 =
 \begin{cases}
  -1,
  &x\in\displaystyle\bigcup_{k\ge k_0} J_k(w),\\[2mm]
  0,
  &x\notin\displaystyle\bigcup_{k\ge k_0}\overline{J_k(w)},\\
  \text{undefined},
  &\text{otherwise}.
 \end{cases}
\end{equation}
\end{enumerate}
In either case, the set of $x\ge x_w$ where the indicated partial
derivative is undefined is locally finite.
\end{lem}

\begin{proof}
We first prove both statements for $w=0$, in which case we write
$I_k:=I_k(0)$ and $J_k:=J_k(0)$.

For $z\in\R$, write $z_+:=\max\{z,0\}$. The definition of
$F_\alpha$ gives
\begin{equation}\label{eq:cf-positive-part}
 F_\alpha(x,y)=
 \begin{cases}
 x-y+\displaystyle\sup_{k\ge0}
       \bigl(\min\{y-a_k,b_k-x\}\bigr)_+,&y<x,\\[2mm]
 \displaystyle\sup_{k\ge0}
       \bigl(\min\{x-a_k,b_k-y\}\bigr)_+,&y>x.
 \end{cases}
\end{equation}

We will refer to the terms $\bigl(\min\{y-a_j,b_j-x\}\bigr)_+$ and $\bigl(\min\{x-a_j,b_j-y\}\bigr)_+$ as \emph{correction terms}.

Suppose that $\theta<1$.
Equation~\eqref{eq:one-dimensional-slow-bound} gives $h(x)<x$
for all sufficiently large $x$. In the first line of
\eqref{eq:cf-positive-part}, the term corresponding to $k$ is positive at $y=h(x)$ precisely on intervals
\begin{equation}\label{eq:interval_corr_Ik}
\left(h^{-1}(a_k),b_k\right).
\end{equation}
Since $\sigma(h^{-1}(a_k))=h^{-1}(a_k)+a_k$, monotonicity of
$\sigma$ and \eqref{eq:5.69} imply
\[
 h^{-1}(a_k)<\xi_k
 \ \Longleftrightarrow\ h^{-1}(a_k)<b_k
 \ \Longleftrightarrow\
 \psi\bigl(|q_k\alpha|_{\Z}^{-1}\bigr)<q_k^{-1}.
\]
Thus the interval \eqref{eq:interval_corr_Ik} and $I_k$ are nonempty precisely for the
indices satisfying \eqref{eq:5.59}. Equations~\eqref{eq:5.56} and
\eqref{eq:one-dimensional-slow-bound} give
\[
 b_k=a_{k+1}+O(1)
 <\theta^{-1}a_{k+1}\le h^{-1}(a_{k+1})
\]
for all sufficiently large $k$, so intervals \eqref{eq:interval_corr_Ik} are pairwise disjoint. Let us assume that furthermore we work with such sufficiently large indices $k$.

For a nonempty interval \eqref{eq:interval_corr_Ik}, the two entries of the minimum in
\eqref{eq:cf-positive-part} agree at $y = h(x)$ precisely when $x=\xi_k$ due to \eqref{eq:5.69}, and
$h^{-1}(a_k)<\xi_k<b_k$. Hence, for every $h^{-1}(a_k)<x<b_k$, if $x \ne \xi_k$, there exists a neighborhood of $h(x)$, depending on $x$, such that for $y$ from this neighborhood one has
\begin{equation}\label{eq:slow-F-two-variables}
 F_\alpha(x,y)=
 \begin{cases}
 x-a_k,&h^{-1}(a_k)<x<\xi_k,\\
 b_k-y,&\xi_k<x<b_k.
 \end{cases}
\end{equation}
Note that only finitely many of the correction terms can be nonzero on this range of $x$ and for $y$ sufficiently close to $h(x)$; by continuity, all of these terms for $j \neq k$ are strictly smaller than the term for $k$. Thus, we do not have to consider indices other than $k$ in \eqref{eq:slow-F-two-variables}.

Outside the closures of the nonempty intervals
\eqref{eq:interval_corr_Ik}, all the correction terms vanish for $y$ sufficiently close to $h(x)$. 
At each endpoint of a nonempty $I_k$, the local formula has distinct one-sided $y$-slopes $-1$ and $0$, so the partial derivative is undefined. At $x=b_k$, the $k$th correction term
is identically zero as a function of $y$; this also handles the degenerate case $h^{-1}(a_k)=b_k$. All remaining boundary points
therefore have vertical derivative $-1$. This proves
\eqref{eq:displ_1}, as well as the disjointness of the nonempty $I_k$.

\medskip

Now suppose that $\theta>1$.
Equation~\eqref{eq:one-dimensional-fast-bound} gives $h(x)>x$
for all sufficiently large $x$. In the second line of
\eqref{eq:cf-positive-part}, the term corresponding to $k$ is positive at $y=h(x)$ precisely on intervals
\begin{equation}\label{eq:interval_corr_Jk}
\left(a_k,h^{-1}(b_k)\right).
\end{equation}
Since $\sigma(h^{-1}(b_k))=h^{-1}(b_k)+b_k$, monotonicity of
$\sigma$ and \eqref{eq:5.69} imply
\begin{equation}\label{eq:equiv-fast-regime}
\xi_k<h^{-1}(b_k)
\ \Longleftrightarrow\ a_k<h^{-1}(b_k)
\ \Longleftrightarrow\
\psi(q_k)>|q_k\alpha|_{\Z}.
\end{equation}
Thus the interval \eqref{eq:interval_corr_Jk} and $J_k$ are nonempty precisely for the
indices satisfying \eqref{eq:5.62}. Equations~\eqref{eq:5.56} and
\eqref{eq:one-dimensional-fast-bound} give
\[
h^{-1}(b_k)\le\theta^{-1}b_k
=\theta^{-1}a_{k+1}+O(1)<a_{k+1}
\]
for all sufficiently large $k$, so intervals \eqref{eq:interval_corr_Jk} are pairwise disjoint. From now on, work with such sufficiently large indices $k$. For a nonempty interval \eqref{eq:interval_corr_Jk}, the two entries of the minimum in \eqref{eq:cf-positive-part} agree at $y=h(x)$ precisely when $x=\xi_k$ due to \eqref{eq:5.69}, and
$a_k<\xi_k<h^{-1}(b_k)$. Hence, for every $a_k<x<h^{-1}(b_k)$, if $x\ne\xi_k$, there exists a neighborhood of $h(x)$, depending on $x$, such that for $y$ from this neighborhood one has
\begin{equation*}\label{eq:fast-F-two-variables}
F_\alpha(x,y)=
\begin{cases}
x-a_k,&a_k<x<\xi_k,\\
b_k-y,&\xi_k<x<h^{-1}(b_k).
\end{cases}
\end{equation*}
The justification of the above expression and the analysis of the remaining cases exactly follows such for the case $\theta < 1$.

\smallskip
Local finiteness of the sets of endpoints of the nonempty intervals $I_k$ and $J_k$ follows from the observation that
\[
 h^{-1}(a_k)\to\infty,
 \qquad
 h^{-1}(b_k)\to\infty,
 \qquad
 \xi_k=\sigma^{-1}(a_k+b_k)\to\infty.
\]

To deduce the theorem for arbitrary $w$,  one needs to apply the case $w=0$ which we just proved to the function $\e^w\psi$.  Its logarithmic transform is
$h(x)-w$, and $\e^w\psi$ satisfies the same hypothesis \eqref{eq:one-dimensional-slow-hypothesis} or \eqref{eq:one-dimensional-fast-hypothesis}.

Taking the thresholds supplied by this application as
 $k_0(w)$ and $x_w$ completes the proof.
\end{proof}

Let $|\cdot|_h$ denote the Lebesgue--Stieltjes measure induced by $h$.
Thus, for $r\le s$,
\[
 |[r,s]|_h:=h(s)-h(r),
 \qquad
 |\varnothing|_h:=0.
\]
Since $h$ is continuous, the inclusion or exclusion of the endpoints
does not affect this quantity.  Choose $x_\ast$ beyond the $w=0$ threshold in
Lemma \ref{lem:one-dimensional-branches} and all endpoints arising
from the finitely many omitted indices.  For $w\in\R$ and
$x\ge x_\ast$, define $B_{-,w}$ in the slow regime and $B_{+,w}$ in
the fast regime by
\begin{equation}\label{eq:branch-density-integral}
 B_{\pm,w}(x):=
 -\int_{x_\ast}^{x}
 \frac{\partial F_\alpha}{\partial y}(t,h(t)-w)\,dh(t),
\end{equation}
where the sign is chosen according to the regime. 

Lemma \ref{lem:one-dimensional-branches} gives, for every fixed
$w\in\R$ and all sufficiently large $x$,
\begin{align}
 B_{-,w}(x)
 &=h(x)-h(x_\ast)
   -\sum_{\substack{k\ge1\\I_k(w)\ne\varnothing}}
    |I_k(w)\cap[x_\ast,x]|_h+O_w(1),
 \label{eq:slow-branch-density}\\
 B_{+,w}(x)
 &=\sum_{\substack{k\ge1\\J_k(w)\ne\varnothing}}
    |J_k(w)\cap[x_\ast,x]|_h+O_w(1).
 \label{eq:fast-branch-density}
\end{align}
The error terms account only for the bounded initial range before the
threshold $x_w$.  For $w=0$, they may be omitted by the choice
of $x_\ast$.
We abbreviate
\[
 B_-(x):=B_{-,0}(x),
 \qquad
 B_+(x):=B_{+,0}(x).
\]

\smallskip

\begin{lem}
\label{lem:one-dimensional-log-sparsity}
For every $w\in\R$, the number of
nonempty intervals $I_k(w)$ in the slow regime, respectively $J_k(w)$
in the fast regime, which meet $[x_\ast,x]$ is {$O_w(\log(h(x)))$}.
\end{lem}

\begin{proof}
Fix $w\in\R$. Recall the nonemptiness criteria in
Lemma~\ref{lem:one-dimensional-branches} and apply the map $\sigma$ to the endpoints of intervals $I_k$ or $J_k$; together with
\eqref{eq:one-dimensional-slow-bound} and \eqref{eq:one-dimensional-fast-bound} it gives, for all sufficiently large $k$,
\[
\begin{aligned}
 a_k+w&<h(b_k)\le\theta b_k
 &&\text{if }\theta<1,\ I_k(w)\ne\varnothing,\\
 \theta a_k&\le h(a_k)<b_k+w
 &&\text{if }\theta>1,\ J_k(w)\ne\varnothing.
\end{aligned}
\]
Using $b_k=a_{k+1}+O(1)$ from \eqref{eq:5.56}, we obtain
\begin{equation}\label{eq:one-dimensional-selected-growth}
 a_{k+1}\ge
 \begin{cases}
 \theta^{-1}a_k-O_w(1),&\theta<1,\ I_k(w)\ne\varnothing,\\
 \theta a_k-O_w(1),&\theta>1,\ J_k(w)\ne\varnothing.
 \end{cases}
\end{equation}

{\it Slow regime.} If a nonempty $I_k(w)$ meets $[x_\ast,x]$, then
$$h^{-1}(a_k+w)<x \ \Longleftrightarrow \ a_k+w<h(x).$$
{Hence $a_k<h(x)+|w|$.} Therefore,
\begin{equation}\label{eq:upper-count-1}
    \begin{gathered}\# \{ k: I_k(w) \cap [x_\ast,x]\ne\varnothing \}\le \# \{ k\ge 1: I_k(w)\ne\varnothing,
    \ {a_k\le h(x)+|w|} \}
    \\
    \le O_w(1)+\# \{ k\ge k_0(w): I_k(w)\ne\varnothing,
    \ {a_k\le h(x)+|w|} \}.
    \end{gathered}
\end{equation}
Here the last inequality holds since the indices
below the threshold $k_0(w)$ from Lemma
\ref{lem:one-dimensional-branches} contribute only $O_w(1)$.

Since {$\theta < 1$}, one has
$\lambda = \frac{{\theta^{-1}} + 1}{2} > 1$. Let
$k'=k'(w,{\theta})\ge\max\{k_0(w),1\}$ be an integer with the
following property: for any $k\ge k'$ such that
$I_k(w)\ne\varnothing$, one has $a_{k+1}\ge\lambda a_k$. Such a
$k'$ exists due to \eqref{eq:one-dimensional-selected-growth}.
Therefore, if $k<\ell$ are consecutive indices greater than $k'$ for which the corresponding intervals are nonempty, then $a_\ell\ge a_{k+1}\ge\lambda a_k$. It remains to notice that, since $a_k=\log q_k$, one has $a_{k'+1}\ge a_2\ge\log 2$. Thus, for all
sufficiently large $x$, the indices $k\le k'$ contribute only $O_w(1)$,
whereas the geometric estimate gives
\[
\begin{aligned}
 &\#\{k>k':I_k(w)\ne\varnothing,\ {a_k\le h(x)+|w|}\}\\
 &\qquad\le
 1+\frac{\log({(h(x)+|w|)}/\log 2)}{\log\lambda}
 ={O_w(\log(h(x)))}.
\end{aligned}
\]
Due to \eqref{eq:upper-count-1}, this completes the proof of the
slow-regime assertion.

{\it Fast regime.} If a nonempty $J_k(w)$ meets $[x_\ast,x]$, then
$$a_k<\xi_k(w)<x.$$
The second inequality here is direct; for the first, note that nonemptiness is equivalent to \eqref{eq:equiv-fast-regime}, thus $h(a_k)<b_k+w$, and
\[
 \sigma(a_k)=a_k+h(a_k)<a_k+b_k+w
 =\sigma(\xi_k(w))
\]
by \eqref{eq:shifted-xi}.
Hence {$a_k\le x\le h(x)/\theta$ for all sufficiently large $x$, by
\eqref{eq:one-dimensional-fast-bound}}. The remainder of the proof is analogous; instead of
$\lambda$ above, we use
$\lambda:=\frac{{\theta}+1}{2}>1$.

\end{proof}

\begin{lem}
\label{lem:one-dimensional-shift-stability}
For every $w\in\R$,
\begin{equation}\label{eq:shifted-branch-density}
 B_{\pm,w}(x)=B_\pm(x)+{O_w(\log(h(x)))}
 \qquad\text{as }x\to\infty.
\end{equation}
\end{lem}

\begin{proof}
For $r,s\in\R$, put
\[
 \mathcal T_x(r,s)
 :=\bigl(\min\{s,h(x)\}-\max\{r,h(x_\ast)\}\bigr)_+,
 \qquad\text{where}\qquad z_+:=\max\{z,0\}.
\]
The operations minimum, maximum, and positive part are $1$-Lipschitz,
and hence
\begin{equation}\label{eq:truncation-lipschitz}
 \bigl|\mathcal T_x(r,s)-\mathcal T_x(r',s')\bigr|
 \le |r-r'|+|s-s'|.
\end{equation}

Equations \eqref{eq:5.73} and \eqref{eq:5.83}, together with the
strict monotonicity of $h$, give
\begin{align}
 |I_k(w)\cap[x_\ast,x]|_h
 &=\mathcal T_x\bigl(a_k+w,
   h(\xi_k(w))\bigr),
 \label{eq:truncated-I-h-length}\\
 |J_k(w)\cap[x_\ast,x]|_h
 &=\mathcal T_x\bigl(h(\xi_k(w)),
   b_k+w\bigr).
 \label{eq:truncated-J-h-length}
\end{align}
In particular, these identities cover intervals whose endpoints
occur in the reverse order.

We next show that, for every $k$,
\begin{equation}\label{eq:h_k_O}
 {|h(\xi_k(w))-h(\xi_k)|\le |w|.}
\end{equation}
{Indeed, \eqref{eq:5.69} and monotonicity of $h$ give, for $u\le v$,
\[
 0\le h(\sigma^{-1}(v))-h(\sigma^{-1}(u))
 =(v-u)-(\sigma^{-1}(v)-\sigma^{-1}(u))\le v-u.
\]
Thus $h\circ\sigma^{-1}$ is $1$-Lipschitz, and
\eqref{eq:shifted-xi} gives \eqref{eq:h_k_O}.}

Combining \eqref{eq:truncated-I-h-length} and
\eqref{eq:truncated-J-h-length} with \eqref{eq:h_k_O} and applying \eqref{eq:truncation-lipschitz},
we obtain
\begin{equation}\label{eq:truncated-branch-length-stability}
\begin{aligned}
 \left|
 |I_k(w)\cap[x_\ast,x]|_h-|I_k\cap[x_\ast,x]|_h
 \right|&{\le 2|w|},\\
 \left|
 |J_k(w)\cap[x_\ast,x]|_h-|J_k\cap[x_\ast,x]|_h
 \right|&{\le 2|w|}.
\end{aligned}
\end{equation}

Only indices for which the interval corresponding to either $w$ or
$0$ meets $[x_\ast,x]$ contribute to the differences in
\eqref{eq:truncated-branch-length-stability}. Lemma
\ref{lem:one-dimensional-log-sparsity}, applied separately to $w$ and
to $0$, shows that there are only {$O_w(\log(h(x)))$} such indices. Extending the sums in
\eqref{eq:slow-branch-density}--\eqref{eq:fast-branch-density} over all
$k$ by assigning length zero to empty intervals and using \eqref{eq:truncated-branch-length-stability}, we conclude that
\[
 |B_{\pm,w}(x)-B_\pm(x)|
 \le 2|w|\,O_w(\log(h(x)))+O_w(1)
 =O_w(\log(h(x))).
\]
This proves \eqref{eq:shifted-branch-density}.
\end{proof}

\begin{lem}
\label{lem:one-dimensional-pressure}
For every sufficiently large $L$ and every $c>0$, one has
\begin{equation}\label{eq:one-dimensional-pressure}
 \mathcal P_{\alpha,\psi}(L,c)
 =
 \begin{cases}
  \displaystyle
  \liminf_{x\to\infty}\frac{B_-(x)}{h(x)}+O(L^{-1}),
  &{\theta<1},\\[4mm]
  \displaystyle
  \liminf_{x\to\infty}\frac{B_+(x)}{h(x)}+O(L^{-1}),
  &{\theta>1}.
 \end{cases}
\end{equation}
The constants described by $O(L^{-1})$ are independent of {$c$}.
\end{lem}

\begin{proof}
Fix $L$ and {$c>0$}, put {$w:=\log c$}, and
{write $x_j:=\log T_j$, where $T_j$ is defined in \eqref{eq:1.11}. } Then,
\[
 {h(x_j)=jL}.
\]
The finitely many initial indices for which $x_{j+1}<0$ contribute
{$O_{L,c}(1)$} and hence do not affect $\mathcal P_{\alpha,\psi}(L,c)$; we suppress them below.
For $v\in[w,w+L]$, put
\begin{equation}\label{eq:G_N(v)}
 G_N(v):=-\frac1N\sum_{j=0}^{N-1}
 \frac{\partial F_\alpha}{\partial y}
 \bigl(x_{j+1},h(x_{j+1})-v\bigr).
\end{equation}
Thus $0\le G_N(v)\le1$.

Fix $v$.  Apart from finitely many initial mesh intervals, Lemma
\ref{lem:one-dimensional-branches} shows that the integrand in
\eqref{eq:branch-density-integral} is constant on every mesh interval
which contains no endpoint of a nonempty $I_k(v)$ or $J_k(v)$.  On
such an interval the corresponding Riemann sum \eqref{eq:G_N(v)} is exact for the integral \eqref{eq:branch-density-integral}, while on any other interval its error is at most $L$.  Lemma
\ref{lem:one-dimensional-log-sparsity}, applied at $x_{N+1}$, therefore
gives
\begin{equation*}
 G_N(v)
 =\frac{B_{\pm,v}(x_N)}{h(x_N)}
 +O_{L,v}\left(\frac{\log(h(x_{N+1}))}{N}\right).
\end{equation*}
Here and below the sign is chosen according to the regime.  Since
{$h(x_{N+1})=(N+1)L$ by \eqref{eq:1.11}}, the last error tends to zero.  Lemma
\ref{lem:one-dimensional-shift-stability} also gives, for this fixed
$v$,
\[
 \frac{B_{\pm,v}(x_N)-B_\pm(x_N)}{h(x_N)} = \frac{{O_v(\log(NL))}}{NL} \longrightarrow0.
\]
Consequently,
\begin{equation}\label{eq:pressure-pointwise-convergence}
 \lim\limits_{N \to \infty} \left( G_N(v)-\frac{B_\pm(x_N)}{h(x_N)} \right) =0
 \qquad\text{for every }v\in[w,w+L].
\end{equation}

The function $y \mapsto F_{\alpha}(x,y)$ is continuous and locally piecewise affine, thus $G_N$ is measurable.  Moreover,
\eqref{eq:branch-density-integral} gives
$0\le B_\pm(x_N)\le h(x_N)-h(x_\ast)$, so the absolute value of the
left-hand side of \eqref{eq:pressure-pointwise-convergence} is bounded
independently of $v$ and all sufficiently large $N$.  Dominated
convergence theorem now yields
\begin{equation}\label{eq:pressure-averaged-convergence}
 \frac1L\int_w^{w+L}G_N(v)\,dv
 =\frac{B_\pm(x_N)}{h(x_N)}+o_{L,c}(1)
\end{equation}
as $N \to \infty$. By {\eqref{eq:1.11}} and \eqref{eq:5.72},
\begin{equation}\label{eq:pressure-slope-integral}
 {\log
 \frac{\mathcal N_\alpha(c\psi(T_j),T_{j+1})}
      {\mathcal N_\alpha(c\psi(T_{j+1}),T_{j+1})}}
 =-\int_w^{w+L}
 \frac{\partial F_\alpha}{\partial y}
 \bigl(x_{j+1},h(x_{j+1})-v\bigr)\,dv+O(1),
\end{equation}
where the last error is uniform in $j$, $L$, and {$c$}.  Summing \eqref{eq:pressure-slope-integral} and
using \eqref{eq:pressure-averaged-convergence}, we obtain
\[
 \begin{aligned}\frac1{NL} & \sum_{j=0}^{N-1}
 \log
 \frac{\mathcal N_\alpha(c\psi(T_j),T_{j+1})}
      {\mathcal N_\alpha(c\psi(T_{j+1}),T_{j+1})} \\
 &=\frac1L\int_w^{w+L}G_N(v)\,dv+O(L^{-1})
 =\frac{B_\pm(x_N)}{h(x_N)}+{o_{L,c}(1)}+O(L^{-1})
 \end{aligned}
\]
as $N \to \infty$. The constant implicit in $O(L^{-1})$ is independent of {$c$}, since it
comes from the uniform $O(1)$ in Lemma \ref{lem:cf-count}.  Now \eqref{eq:4.19} gives
\[
 \mathcal P_{\alpha,\psi}(L,{c})
 =\liminf_{N\to\infty}\frac{B_\pm(x_N)}{h(x_N)}+O(L^{-1}).
\]

Finally, \eqref{eq:branch-density-integral} gives, for $x\le y$,
\[
 0\le B_\pm(y)-B_\pm(x)\le h(y)-h(x).
\]
Thus, if $x_N\le x<x_{N+1}$, then
\[
 |B_\pm(x)-B_\pm(x_N)|\le L,
 \qquad
 |h(x)-h(x_N)|\le L,
\]
while $B_\pm(x_N)=O(h(x_N))$.  Therefore
\[
 \left|
 \frac{B_\pm(x)}{h(x)}
 -
 \frac{B_\pm(x_N)}{h(x_N)}
 \right|
 =O\left(\frac{L}{h(x_N)}\right)
 =o(1) \qquad \text{as} \ N \to \infty,
\]
and \eqref{eq:one-dimensional-pressure} follows.
\end{proof}

\begin{proof}[Proof of Theorem \ref{thm:one-dimensional}]
{Since $\psi$ is lower quasimultiplicative,} Theorem \ref{thm:main} applies. By Lemma \ref{lem:one-dimensional-pressure} and Theorem \ref{thm:main}\textup{(ii)}, 
\begin{equation}\label{eq:5.80}
\begin{gathered}
 \dim_H\Dhat_\alpha(\psi)
 ={\lim_{L\to\infty}\inf_{c>0}
 \mathcal P_{\alpha,\psi}(L,c)} \\
 =
 {\lim_{L\to\infty}\sup_{c>0}
 \mathcal P_{\alpha,\psi}(L,c)}
 =\liminf_{x\to\infty}\frac{B_\pm(x)}{h(x)},
 \end{gathered}
\end{equation}
where the sign in the last expression is chosen according to regime.

\smallskip
\noindent\emph{The slow regime.} Suppose {$\theta<1$}.
If only finitely many $I_k$ are nonempty, due to Lemma \ref{lem:one-dimensional-branches} and \eqref{eq:branch-density-integral}, the last expression equals
one, proving \eqref{eq:5.60}.  Otherwise, let $\{k_i\}$ be the
increasing sequence from part \textup{(i)}. The quotient $\frac{B_-(x)}{h(x)}$ increases off the intervals $I_{k_i}$ and decreases on
them.  Its liminf is therefore reached at $x=\xi_{k_i}$.  At such a point,
\begin{equation}\label{eq:5.81}
 B_-(\xi_{k_i})
 =a_{k_i}
 +\sum_{j=1}^{i-1}\bigl(a_{k_j}-h(\xi_{k_j})\bigr)
 +O(1).
\end{equation}
Recall that $a_k=\log q_k$, $\xi_k=\log Y_k$, and
$h(\xi_k)=\log\psi(Y_k)^{-1}$. Substituting \eqref{eq:5.81} into \eqref{eq:5.80} gives exactly
the quotient in \eqref{eq:5.61}.  This proves part \textup{(i)}.

\smallskip
\noindent\emph{The fast regime.}
Suppose {$\theta>1$}.
If only finitely many $J_k$ are nonempty, the last expression equals
zero, proving \eqref{eq:5.63}.  Otherwise, let $\{k_i\}$ be the
increasing sequence from part \textup{(ii)}.  The quotient $\frac{B_+(x)}{h(x)}$ increases on the intervals $J_{k_i}$ and decreases
between them, so its liminf is reached at $x=\xi_{k_i}$.  One has
\begin{equation}\label{eq:5.88}
 B_+(\xi_{k_i})
 =\sum_{j=1}^{i-1}\bigl(b_{k_j}-h(\xi_{k_j})\bigr)+O(1).
\end{equation}
Since $b_k=\log|q_k\alpha|_{\Z}^{-1}$, substituting \eqref{eq:5.88} into \eqref{eq:5.80} gives
exactly the quotient in \eqref{eq:5.64}.  This proves part
\textup{(ii)}.

\smallskip

In either regime, \eqref{eq:5.65} follows from the equality {$\lim_{L\to\infty}\inf_{c>0}
 \mathcal P_{\alpha,\psi}(L,c) =
 \lim_{L\to\infty}\sup_{c>0}
 \mathcal P_{\alpha,\psi}(L,c)$} established in \eqref{eq:5.80} via Theorem
\ref{thm:main}\textup{(ii)} {and \eqref{eq:4.22}}.
\end{proof}

\begin{proof}[Proof of Theorem \ref{cor:kim-liao}]
Parts \textup{(iii)} and \textup{(iv)} follow directly from Theorem
\ref{thm:one-dimensional} and Remark \ref{rem:for-KL}. Namely, in case \textup{(iii)} we apply Theorem \ref{thm:one-dimensional}, \textup{(i)} to
$\psi=a^{-1}\psi_\tau$ to get the result for $c = a^{-1}$. In case \textup{(iv)} we apply Theorem \ref{thm:one-dimensional}, \textup{(ii)} to
$\psi=b\psi_\tau$ to get the result for $c = b$. In both cases, \eqref{eq:5.65} shows that the result holds for arbitrary choice of $c > 0$.

Part \textup{(i)} follows directly from the additional assertion in
Proposition \ref{prop:transference-full}. Part \textup{(ii)} follows
directly from the additional assertion in Proposition
\ref{prop:trivial-shifts} applied with $\mathcal Q=\N$.

\end{proof}

\section{Proofs of Propositions \ref{prop:aggarwal-general}, \ref{prop:aggarwal-special} and \ref{prop:aggarwal-sharpness}}\label{sec:Aggarwal_result}

\subsection{Upper dimension estimates}

We start this section with a technical lemma, which is just a translation between two notations. For $\tau,\rho>0$, let $g_i$ be defined by
\eqref{eq:power-normalized-minima} and put
\[
 H_\rho(t):=\log\mathcal N_A(\rho\e^{-\tau t},\rho\e^t).
\]

\begin{lem}\label{lem:power-profile-products}
For $x,h\ge0$, put $t=x+h/(1+\tau)$. Then
\begin{equation}\label{eq:power-shifted-profile-product}
 \mathcal N_A(\rho\e^{-\tau x},\rho\e^{x+h})
 =\prod_{i=1}^d
 \max\left\{\rho\e^{\tau h/(1+\tau)-g_i(t)},1\right\},
\end{equation}
and
\begin{equation}\label{eq:power-profile-product}
 \e^{H_\rho(t)}
 =\prod_{i=1}^d\max\left\{\rho\e^{-g_i(t)},1\right\}.
\end{equation}
\end{lem}

\begin{proof}
By \eqref{eq:power-normalized-minima}, the successive minima of
\begin{equation}\label{eq:rect-to-dilate}
 [-\e^{-\tau t},\e^{-\tau t}]^m
 \times[-\e^t,\e^t]^n
\end{equation}
with respect to $\Lambda_A$ are $\e^{g_i(t)}$. Since
\[
 (\rho\e^{-\tau x},\rho\e^{x+h})
 =\rho\e^{\tau h/(1+\tau)}(\e^{-\tau t},\e^t),
\]
the rectangles defining the lattice counts in
\eqref{eq:power-shifted-profile-product} and
\eqref{eq:power-profile-product}  via \eqref{eq:rectangular-profile} are dilates of \eqref{eq:rect-to-dilate} by $\rho\e^{\tau h/(1+\tau)}$ and $\rho$, respectively. Dilation divides each successive minimum by the same factor,
so their minima are, respectively,
\[
 \frac{\e^{g_i(t)}}{\rho\e^{\tau h/(1+\tau)}}
 \qquad\text{and}\qquad
 \frac{\e^{g_i(t)}}{\rho}.
\]
Both identities now follow from \eqref{eq:rectangular-profile}.
\end{proof}

The upper dimension estimate below is a key step to prove Proposition~\ref{prop:aggarwal-general}.

\begin{lem}\label{lem:power-upper-all}
Let $A\in\Mat_{m,n}$. For every $\tau,c,L>0$, one has
\begin{equation}\label{eq:critical-power-upper}
\begin{split}
 \dim_H\Dhat_A(c\psi_\tau)
 \le{}&\liminf_{T\to\infty}
 \frac{\displaystyle\int_1^T
 r_{A,\tau}\bigl(S;(c\e^L)^{1/(1+\tau)}\bigr)\,\frac{dS}{S}
 -\log\mathcal N_A(T^{-\tau},T)}{(1+\tau)\log T}\\
 &+\frac{\log C_d}{L},
\end{split}
\end{equation}
where $C_d\ge1$ is the constant in Proposition~\ref{prop:upper}.
\end{lem}

\begin{proof}
Put $h=L/\tau$, $\rho=c^{1/(1+\tau)}$, and
\[
 V_\rho(x):=\log
 \frac{\mathcal N_A(\rho\e^{-\tau x},\rho\e^{x+h})}
      {\mathcal N_A(\rho\e^{-\tau(x+h)},\rho\e^{x+h})}.
\]
Lemma~\ref{lem:vertical} gives
$0\le V_\rho\le mL+\log K_d$.

\smallskip
\noindent\emph{Bound one transition factor.} One has

\[
\begin{aligned}
&\log\mathcal N_A(\rho\e^{-\tau x},\rho\e^{x+h})
 -H_\rho\!\left(x+\frac{h}{1+\tau}\right)\\
&\underset{\substack{\eqref{eq:power-shifted-profile-product}\\
                    \eqref{eq:power-profile-product}}}{=}
 \sum_{i=1}^d\Biggl[
 \max\left\{\log\rho+\frac{\tau h}{1+\tau}
             -g_i\!\left(x+\frac{h}{1+\tau}\right),0\right\}\\
&\hspace{35mm}
 -\max\left\{\log\rho
             -g_i\!\left(x+\frac{h}{1+\tau}\right),0\right\}
 \Biggr]\\
&\le \frac{\tau h}{1+\tau}
 \#\left\{i:
 g_i\!\left(x+\frac{h}{1+\tau}\right)
 <\log\rho+\frac{\tau h}{1+\tau}\right\}\\
&\underset{\eqref{eq:power-rank-cutoff}}{=}
 \frac{\tau h}{1+\tau}
 r_{A,\tau}\!\left(
 \e^{x+h/(1+\tau)};(c\e^L)^{1/(1+\tau)}
 \right).
\end{aligned}
\]
The last equality becomes easy to see if one uses the variable $t=x+h/(1+\tau)$. By the definition of $V_\rho$, this gives

\[
\begin{aligned}
 V_\rho(x)\le{}&
 \frac{\tau h}{1+\tau}
 r_{A,\tau}\!\left(
 \e^{x+h/(1+\tau)};(c\e^L)^{1/(1+\tau)}
 \right)\\
 &+H_\rho\!\left(x+\frac{h}{1+\tau}\right)-H_\rho(x+h).
\end{aligned}
\]

\smallskip
\noindent\emph{Integrate this bound from $0$ to $x$.}
By monotonicity and \eqref{eq:profile-scaling},
$H_\rho$ is Lipschitz with constant $d\max\{1,\tau\}$ and
$|H_\rho-H_1|\le d|\log\rho|$. The difference of the two $H_\rho$ integrals reduces to
endpoint intervals:
\[
\begin{gathered}
 \int_0^x
 \left[
 H_\rho\!\left(u+\frac{h}{1+\tau}\right)-H_\rho(u+h)
 \right]\,du=
 \int_{h/(1+\tau)}^h H_\rho(u)\,du
\\ -\int_{x+h/(1+\tau)}^{x+h}H_\rho(u)\,du
=-\frac{\tau h}{1+\tau}H_\rho(x)+O_{c, \tau, h}(1),
\end{gathered}
\]
where the last equality holds by Lipschitz continuity, and $O_{c, \tau, h}(1) = O_{\rho, \tau, h}(1)$.
Since $0\le r_{A,\tau}\le d$, shifting its integral by
$h/(1+\tau)$ also changes it by $O_{c, h, \tau}(1)$. Consequently,
\begin{equation}\label{eq:all-power-average-bound}
 \frac1h\int_0^x V_\rho(u)\,du
 \le\frac{\tau}{1+\tau}
 \left[
 \int_0^x
 r_{A,\tau}\!\left(\e^u;(c\e^L)^{1/(1+\tau)}\right)\,du
 -H_1(x)
 \right]+O_{c, h, \tau}(1).
\end{equation}

\smallskip
\noindent\emph{Apply the upper dimension estimate.}
For each $v\in[0,h]$, use the mesh
$T_j=\rho\e^{v+jh}$. Since
$c\psi_\tau(T_j)=\rho\e^{-\tau(v+jh)}$, the ratio in
\eqref{eq:2.4} satisfies
\[
 \log
 \frac{\mathcal N_A(c\psi_\tau(T_j),T_{j+1})}
      {\mathcal N_A(c\psi_\tau(T_{j+1}),T_{j+1})}
 =V_\rho(v+jh).
\]
Discarding a common fixed initial segment and modifying the
approximation function on a bounded interval, we can satisfy the initial assumptions of Proposition~\ref{prop:upper} without affecting
the limits below. Applying that proposition, averaging over $v$, and using Fatou's lemma gives
\[
\begin{gathered}
 \dim_H\Dhat_A(c\psi_\tau)-\frac{\log C_d}{L}
 \le
 \frac1h\int_0^h
 \liminf_{N\to\infty}
 \frac{\sum_{j=0}^{N-1}V_\rho(v+jh)}{\tau Nh}\,dv\\
 \quad \le
 \liminf_{N\to\infty}\frac1h\int_0^h
 \frac{\sum_{j=0}^{N-1}V_\rho(v+jh)}{\tau Nh}\,dv
 =
 \liminf_{N\to\infty}
 \frac{\int_0^{Nh}V_\rho(u)\,du}{\tau Nh^2}\\
 \quad\underset{\eqref{eq:all-power-average-bound}}{\le}
 \liminf_{N\to\infty}
 \frac{\displaystyle
 \int_0^{Nh}
 r_{A,\tau}\!\left(\e^u;(c\e^L)^{1/(1+\tau)}\right)\,du
 -H_1(Nh)}
 {(1+\tau)Nh}.
\end{gathered}
\]
The numerator in the last line is Lipschitz, so replacing
$Nh$ by arbitrary $x\to\infty$ leaves the liminf unchanged.
Substituting $T=\e^x$ proves \eqref{eq:critical-power-upper}.
\end{proof}

\begin{proof}[Proof of Proposition~\ref{prop:aggarwal-general}]
Fix $\delta,L>0$ and apply \eqref{eq:critical-power-upper} with
$c=\delta^{1+\tau}\e^{-L}$. Since
$\Sing_A(\psi_\tau)\subseteq\Dhat_A(c\psi_\tau)$, letting
$L\to\infty$ and then $\delta\searrow0$ gives \eqref{eq:5.31}.
The outer limit exists by monotonicity in $\delta$.
\end{proof}

\begin{proof}[Proof of Proposition~\ref{prop:aggarwal-special}]
Take $\tau=n/m$.
Minkowski's second theorem gives
$\prod_{i=1}^d\lambda_i(t)\ge1/d!$, and so $\lambda_d(t) \ge \left(d! \right)^{-\frac 1d}$. Thus, for
$0<\delta<(d!)^{-1/d}$,
\[
 r_{A,n/m}(T;\delta)
 \le(d-1)\one_{\{\lambda_1(\frac{mn}{d}(1+n/m)\log T)<\delta\}}.
\]
Since $\log\mathcal N_A(T^{-n/m},T)\ge0$, using \eqref{eq:5.31} and making the change of
variables $t=\frac{mn}{d}(1+n/m)\log T$ gives
\[
 \dim_H\Sing_A(\psi_{n/m})
 \le\frac{d-1}{1+n/m}\,\underline{\Emass}(A)
 =\frac{m(d-1)}d\,\underline{\Emass}(A).
\]
Here we used $<$ instead of $\le$ in the definition of escape of mass, which does not change its value since $\delta\searrow0$.
\end{proof}

\subsection{Templates and lattice trajectories}\label{subsec:templates}

We recall the definitions and correspondence theorem from
\cite{DFSU24} needed in Subsection~\ref{subsec:sharpness}.
For $1\le r\le d$, put
\begin{equation}\label{eq:template-slopes}
 \mathcal Z_r:=
 \left\{\frac am-\frac bn:
 a,b\in\Z_{\ge0},\ a\le m,\ b\le n,\ a+b=r\right\}.
\end{equation}

\begin{defn}[{\cite[Definition~4.1]{DFSU24}}]\label{def:template}
An \emph{$(m,n)$-template} is a continuous map
$\mathbf f=(f_1,\ldots,f_d):[0,\infty)\to\R^d$,
affine on each piece of a finite partition of every bounded interval,
with the following properties:
\begin{enumerate}[label=\textup{(\roman*)}]
\item $f_1\le\cdots\le f_d$.
\item $-1/n\le f_i'\le1/m$ wherever the derivative exists.
\item For each $1\le r\le d$, on any interval where $f_r<f_{r+1}$,
the partial sum $\sum_{i=1}^r f_i$ is convex and all its slopes
belong to $\mathcal Z_r$. Here $f_{d+1}:=+\infty$.
\end{enumerate}
\end{defn}

Since $\mathcal Z_d=\{0\}$, the total sum is constant.

In the normalization \eqref{eq:1.6}, we say that a matrix $A$
\emph{follows} $\mathbf f$ if, for some constant $0\le C<\infty$,
\begin{equation}\label{eq:template-correspondence}
 |\log\lambda_i(t)-f_i(t)|\le C
 \qquad \text{for} \ t\ge0,\ 1\le i\le d.
\end{equation}

\begin{thm}[{\cite[Theorem~4.2]{DFSU24}}]\label{thm:template-correspondence}
\begin{enumerate}[label=\textup{(\roman*)}]
\item Every $A\in\Mat_{m,n}$ follows some $(m,n)$-template.
\item Every $(m,n)$-template is followed by some $A\in\Mat_{m,n}$.
\end{enumerate}
\end{thm}

\subsection{Sharpness of the escape-of-mass estimate}\label{subsec:sharpness}

\begin{proof}[Proof of Proposition~\ref{prop:aggarwal-sharpness}]
For $\eta = 0$, the desired statement only requires us to prove the existence of a matrix $A$ for which $\underline{\Emass}(A) = 0$. Such matrices exist: for instance, it holds for badly approximable matrices. From now on, assume $0<\eta\le1$.

\smallskip
\noindent\emph{Step 1: describe the templates.}
Choose positive numbers $h_k$ such that
\begin{equation}\label{eq:sharp-height-conditions}
 h_k\longrightarrow\infty,
 \qquad
 h_k=o\left(\sum_{j=1}^{k-1}h_j\right).
\end{equation}
Take $s_1\ge0$ and put $s_{k+1}=s_k+dh_k/\eta$.
Define $b=0$ on $[0,s_1]$ and, for $s_k\le t\le s_{k+1}$, put
\begin{equation}\label{eq:sharp-tents}
 b(t)=
 \begin{cases}
 \displaystyle\min\left\{\frac{t-s_k}{m},
                  \frac{s_k+dh_k-t}{n}\right\},
       &s_k\le t\le s_k+dh_k,\\[4pt]
 0,&s_k+dh_k\le t\le s_{k+1}.
 \end{cases}
\end{equation}
Set
\begin{equation}\label{eq:sharp-template}
 f_1(t)=\cdots=f_{d-1}(t)=-\frac{b(t)}{d-1},
 \qquad f_d(t)=b(t).
\end{equation}
Figure~\ref{fig:sharp-template} shows one cycle.
These functions satisfy Definition~\ref{def:template}. Indeed, they are ordered, sum to
zero, and their slopes lie in $[-1/n,1/m]$.
The only strict gap between consecutive coordinates is $f_{d-1}<f_d$ on $(s_k,s_k+dh_k)$. On these segments $\sum_{i=1}^{d-1}f_i=-b$ is convex, with successive slopes $-\frac1m=\frac{m-1}{m}-\frac nn$ and $\frac1n=\frac mm-\frac{n-1}{n}$, both belonging to $\mathcal Z_{d-1}$ in \eqref{eq:template-slopes}.

\begin{figure}[htbp]
\centering
\begin{tikzpicture}[x=1.05cm,y=1cm,
                    every node/.style={font=\small}]
 \draw[->,thin] (-.25,0)--(9,0) node[right] {$t$};
 \draw[densely dashed,thin] (0,-1.5)--(0,.15);
 \draw[densely dashed,thin] (3.1,-1.5)--(3.1,1.8);
 \draw[densely dashed,thin] (5.4,-1.5)--(5.4,.15);
 \draw[densely dashed,thin] (8.3,-1.5)--(8.3,.15);
 \draw[densely dashed,thin] (-.1,1.8)--(3.1,1.8);
 \draw[densely dashed,thin] (-.1,-.9)--(3.1,-.9);
 \draw[thick] (0,0)--(3.1,1.8)--(5.4,0)--(8.3,0);
 \draw[thick] (0,0)--(3.1,-.9)--(5.4,0);
 \node[left] at (-.1,1.8) {$h_k$};
 \node[left] at (-.1,-.9) {$-\dfrac{h_k}{d-1}$};
 \node[above left] at (1.55,.9) {$1/m$};
 \node[above right] at (4.25,.9) {$-1/n$};
 \node[above] at (4.6,1.55) {$f_d$};
 \node[below right] at (3.8,-.7) {$f_1=\cdots=f_{d-1}$};
 \node[above] at (6.85,.12) {$f_1=\cdots=f_d=0$};
 \node[below] at (0,-1.5) {$s_k$};
 \node[below] at (3.1,-1.5) {$s_k+mh_k$};
 \node[below] at (5.4,-1.5) {$s_k+dh_k$};
 \node[below] at (8.3,-1.5) {$s_{k+1}$};
\end{tikzpicture}
\caption{One cycle: $[s_k,s_k+dh_k]$ and the following zero interval.
}
\label{fig:sharp-template}
\end{figure}
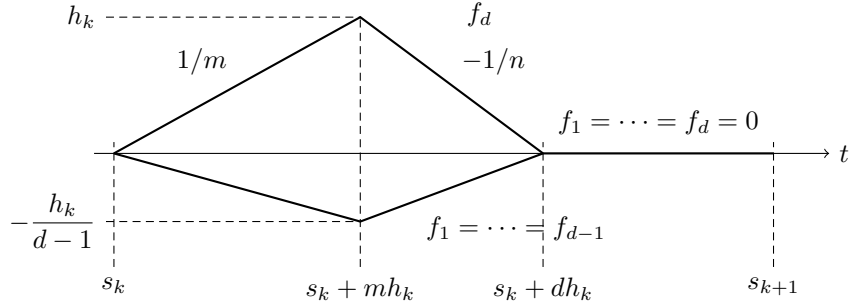

\smallskip
\noindent\emph{Step 2: compute $\underline{\Emass}(A)$ and
$\mathcal P_{A,\psi_{n/m}}(L,c)$.}
Suppose that $A$ follows this template with constant $C\ge0$
in \eqref{eq:template-correspondence}.
By \eqref{eq:sharp-height-conditions}, the number of cycles
before time $t$ is $o(t)$, and the length of an unfinished cycle
is $o(t)$. The coordinate $f_1$, which approximates
$\log\lambda_1$ within $C$ by \eqref{eq:template-correspondence},
therefore satisfies
\begin{equation}\label{eq:sharp-template-proportions}
 \lim_{t\to\infty}\frac1t
 \mathcal L^1\{u\in[0,t]:f_1'(u)=v\}
 =
 \begin{cases}
 \eta m/d,&v=-1/(m(d-1)),\\
 \eta n/d,&v=1/(n(d-1)),\\
 1-\eta,&v=0.
 \end{cases}
\end{equation}
For any fixed $K>0$, the part of $[s_k,s_k+dh_k]$ where
$b\le K$ has length at most $dK$.
Using \eqref{eq:template-correspondence}, we conclude that,
for every fixed $0<\delta<\e^{-C}$,
\[
 \lim_{t\to\infty}\frac1t
 \mathcal L^1\{u\in[0,t]:\lambda_1(u)\le\delta\}=\eta.
\]
In particular,
\begin{equation}\label{eq:sharp-template-escape}
 \underline{\Emass}(A)=\eta.
\end{equation}

Now we compute $\mathcal P_{A,\psi_{n/m}}(L,c)$: \eqref{eq:2.1} and
\eqref{eq:template-correspondence} give
\begin{equation}\label{eq:sharp-template-profile}
 \left|
 \log\mathcal N_A(R,Q)
 -\sum_{i=1}^d
  \max\left\{\frac{n\log Q+m\log R}{d}
                  -f_i({t_{R,Q}}),0\right\}
 \right|\le dC
\end{equation}
whenever ${t_{R,Q}}\ge0$.
Fix $L>0$ and $0 < c\le\e^{-L}$. For $\psi=\psi_{n/m}$, the mesh
\eqref{eq:1.11} is $T_j=\e^{mjL/n}$. Put
\[
 t_j=t_{c\psi(T_{j+1}),T_{j+1}}=m(j+1)L-\frac{mn}{d}\log c;
\]
then,
\[
 t_j-\frac{mnL}{d}
 =\frac{mn}{d}\log\frac{T_{j+1}}{c\psi(T_j)}
 =t_{c\psi(T_j),T_{j+1}}.
\]
For all sufficiently large $j$, substitution into
\eqref{eq:sharp-template-profile} gives
\begin{equation}\label{eq:sharp-template-transition}
\begin{gathered}
 \log
 \frac{\mathcal N_A(c\psi(T_j),T_{j+1})}
      {\mathcal N_A(c\psi(T_{j+1}),T_{j+1})}
 =
 \sum_{i=1}^d
 \Bigl[
 \max\left\{\tfrac md(\log c+L)
      -f_i\!\left(t_j-\tfrac{mnL}{d}\right),0\right\}\\
 -\max\{\tfrac md\log c-f_i(t_j),0\}
 \Bigr]+O_d(C).
\end{gathered}
\end{equation}

If $t_j$ lies outside neighborhoods of the points
$s_k$, $s_k+mh_k$, and $s_k+dh_k$, of sufficiently large fixed
width depending on $L$ and $c$, both times in
\eqref{eq:sharp-template-transition} belong to the same affine
interval of the template. If they lie in $(s_k,s_k+dh_k)$, by increasing these exceptional intervals if necessary, we can in addition guarantee that 
$$
f_i\!\left(t_j-\tfrac{mnL}{d}\right) < \tfrac md(\log c+L) \quad \text{and} \quad  f_i(t_j) <  \tfrac md\log c
$$
for $i = 1, \ldots, d-1$. Only the first $d-1$ coordinates contribute nonzero terms to the summation, and the sum in
\eqref{eq:sharp-template-transition} is equal to
\[
 \frac{m(d-1)}dL
 +b\!\left(t_j-\frac{mnL}{d}\right)-b(t_j)
 =
 \begin{cases}
 (m-1)L,&b\text{ is increasing},\\
 mL,&b\text{ is decreasing}.
 \end{cases}
\]
On $(s_k + dh_k, s_{k+1})$, the same sum equals $m\bigl(\max\{\log c+L,0\}-\max\{\log c,0\}\bigr) = 0$.

The exceptional indices have density zero. Indeed, each
of the points $s_k$, $s_k+mh_k$ and $s_k+dh_k$ excludes $O_{L,c}(1)$ indices, while their number up to time $t$ is $o(t)$.
Their contribution is negligible by Lemma~\ref{lem:vertical}.
Counting points of the arithmetic progression $\{t_j\}$ in
each interval changes its length divided by $mL$ by $O(1)$.
Thus the three types of indices have the same proportions
$\eta m/d$, $\eta n/d$, and $1-\eta$ as in
\eqref{eq:sharp-template-proportions}. Equation~\eqref{eq:4.19} consequently gives
\begin{equation}\label{eq:sharp-template-pressure}
\begin{gathered}
 \mathcal P_{A,\psi_{n/m}}(L,c)
 ={}\eta\,\frac{m(d-1)}d+O_d(C/L),
\end{gathered}
\end{equation}
where the implied constant is independent of $c$.
After taking the infimum over $c>0$, and then letting $L\to\infty$, Theorem~\ref{thm:main} gives
\begin{equation}\label{eq:sharp-template-dimension}
 \dim_H\Sing_A(\psi_{n/m})=\eta\,\frac{m(d-1)}d.
\end{equation}
This completes the proof.

\end{proof}

\begin{proof}[Proof of Example ~\ref{rem:aggarwal-sharp-example}]
The convergent denominators of $\alpha=[0;2,3,4,\ldots]$ satisfy
$q_0=1$, $q_1=2$, and
\[
 q_{k+1}=(k+2)q_k+q_{k-1}.
\]
Since $q_{k-1}\ge kq_{k-2}$, this gives
$$
\begin{gathered} 1\le\frac{q_k}{(k+1)!}
 \le \prod_{j=2}^k\left(1+\frac1{j(j+1)}\right)
\le \exp\left(\sum_{j=2}^k\frac1{j(j+1)}\right)
\\ = 
\exp\left(\frac12-\frac1{k+1}\right) \le e^{1/2}.
\end{gathered}
$$
Writing $a_k=\log q_k$ as in \eqref{eq:a_k_b_k}, we obtain
\begin{equation}\label{eq:sim-klog}
\lim_{k\to\infty}\frac{a_k}{k\log k}=1 \quad \text{and} \quad a_{k+1}-a_k = \log(k+2) + O(1).
\end{equation}
For $a_k\le t\le a_{k+1}$, $k\ge1$, put
\[
 b(t)=\min\{t-a_k,a_{k+1}-t\},
\]
and set $b=0$ on $[0,a_1]$.

For $a_k\le t\le a_{k+1}$, take the primitive vector
$\bv_k:=(q_k\alpha-p_k,q_k)\in u_\alpha\Z^2$.
By \eqref{eq:5.56} and the definition of $b$,
\[
 \frac12\e^{-b(t)}
 \le |a_t\bv_k|
 =\max\{\e^t|q_k\alpha|_{\Z},\e^{-t}q_k\}
 \le\e^{-b(t)}\le1.
\]

For every $\bv\in u_\alpha\Z^2$ independent of $\bv_k$,
unimodularity gives
\[
 1\le|\det(a_t\bv_k,a_t\bv)|
 \le2|a_t\bv_k|\,|a_t\bv|,
\]
so $|a_t\bv|\ge1/2$. Every nonzero collinear lattice vector
is an integer multiple of $\bv_k$ and hence is no shorter.
Therefore
\[
 \frac12\e^{-b(t)}\le\lambda_1(t)\le\e^{-b(t)}.
\]
After enlarging the constant to cover the bounded initial
interval $[0,a_1]$, we obtain
$-\log\lambda_1(t)=b(t)+O(1)$.
Minkowski's second theorem gives $1/2\le\lambda_1(t)\lambda_2(t)\le1$,
and hence $\log\lambda_2(t)=b(t)+O(1)$.

Thus \eqref{eq:template-correspondence} holds for the
functions \eqref{eq:sharp-tents}--\eqref{eq:sharp-template},
with $m=n=1$, $\eta=1$, $s_k=a_k$, and
\[
 h_k=\frac{a_{k+1}-a_k}{2}.
\]
The estimates above verify \eqref{eq:sharp-height-conditions}.
Equation~\eqref{eq:sharp-template-escape} gives
$\underline{\Emass}(\alpha)=1$, while
\eqref{eq:sharp-template-dimension} and
Theorem~\ref{thm:main} give $P_-(\alpha,\psi_1)=1/2$.

To prove that the same value holds for $\dim_H\Dhat_\alpha(c\psi_1)$, fix $\delta>0$.
By \eqref{eq:template-correspondence}, for some fixed $C\ge0$,
$r_{\alpha,1}(\e^t;\delta)=1$ whenever $b(t)>C+|\log\delta|$.
The complementary set has bounded length in each interval
$[a_k,a_{k+1}]$ and hence has density zero, due to \eqref{eq:sim-klog}.
Consequently,
\[
 \lim_{T\to\infty}\frac1{\log T}
 \int_1^T r_{\alpha,1}(S;\delta)\,\frac{dS}{S}=1.
\]
Since $\log\mathcal N_\alpha(T^{-1},T)\ge0$,
Lemma~\ref{lem:power-upper-all}, with $\delta=(c\e^L)^{1/2}$, yields
\[
 \dim_H\Dhat_\alpha(c\psi_1)
 \le \frac12+\frac{\log C_2}{L}.
\]
Letting $L \to \infty$ completes the proof.
\end{proof}

\section{Proofs of Propositions \ref{prop:transference-null},
\ref{prop:transference-full}, and \ref{prop:trivial-shifts}}\label{sec:transference}

We use the following two lemmata in the proofs of Propositions
\ref{prop:transference-null} and \ref{prop:transference-full}.

\begin{lem}
\label{lem:dual-strip}
Suppose that $A^\top$ is $\varphi$-approximable and that $\psi$ and
$\varphi$ are dual. Then, for
every $0<\delta<1/(2d)$,
\begin{equation}\label{eq:dual-strip-bound}
 \mathcal L^m(\Dhat_A(\widetilde\psi_\delta))\le2d\delta.
\end{equation}
\end{lem}

\begin{proof}
Choose a sequence $S_k\to\infty$ and
\[
 (\p_k,\q_k)
 \in(\Z^m\setminus\{\mathbf0\})\times\Z^n
\]
such that
\[
 |\p_k|\le S_k,
 \qquad
 |A^\top\p_k-\q_k|\le\varphi(S_k);
\]
it exists due to the theorem assumption. Put $T_k:=1/\varphi(S_k)$. By duality, $\psi(T_k)=1/S_k$.

Fix $\bb\in \Dhat_A(\widetilde\psi_\delta)$. Applying the defining property of $\Dhat_A(\widetilde\psi_\delta)$ at
the parameter $\delta T_k$, for all sufficiently large $k$ we can find
$\mathbf r_k\in\Z^n$ and $\mathbf a_k\in\Z^m$ such that
\[
 |\mathbf r_k|\le\delta T_k,
 \qquad
 |A\mathbf r_k-\bb-\mathbf a_k|\le\frac{\delta}{S_k}.
\]
We will use a classical transference argument. Note that $\p_k \cdot A\mathbf r_k = \mathbf r_k \cdot A^\top\p_k$. Therefore,
\[
\begin{gathered}
 |\p_k\cdot\bb|_{\Z} = |\p_k\cdot(A\mathbf r_k - \bb) -  \mathbf r_k \cdot A^\top\p_k|_{\Z}\\
 \le
 m|\p_k|\,|A\mathbf r_k-\bb-\mathbf a_k|
 +n|\mathbf r_k|\,|A^\top\p_k-\q_k|
 \le m\delta+n\delta=d\delta.
\end{gathered}
\]
Consequently,
\[
 \Dhat_A(\widetilde\psi_\delta)
 \subseteq
 \liminf_{k\to\infty}
 \left\{
 \bb\in\T^m:
 |\p_k\cdot\bb|_{\Z}\le d\delta
 \right\}.
\]
Since $\p_k\ne\mathbf0$, the homomorphism
\[
 \bb\longmapsto\p_k\cdot\bb\pmod{\Z}
\]
is surjective, thus it maps Haar measure on $\T^m$ to Haar measure on $\T$ (see e.g. \cite[\S 63, Theorem C]{Hal}). Therefore the sets under $\liminf$ above are strips of measure $2d\delta$. Since the measure of a
liminf is at most the lower limit of the measures, this proves
\eqref{eq:dual-strip-bound}.
\end{proof}

The following lemma is a classical transference statement; we will use it in the form of
\cite[Lemma~I]{MN25}.

\begin{lem}
\label{lem:inhomogeneous-transference}
Let $S,T>0$ and suppose that
\begin{equation}\label{eq:transference-hypothesis}
 |A^\top\p|_{\Z^n}\ge \frac{\kappa_d}T
 \qquad
 \text{for every }\p\in\Z^m\setminus\{\mathbf0\}
 \text{ with }|\p|\le S.
\end{equation}
Then, for every $\bb\in\T^m$, there exist $\q\in\Z^n$ and
$\mathbf a\in\Z^m$ such that
\begin{equation*}\label{eq:transference-conclusion}
 |\q|\le T,
 \qquad
 |A\q-\mathbf a-\bb|\le \frac{\kappa_d}S.
\end{equation*}
\end{lem}

\begin{proof}[Proof of Proposition \ref{prop:transference-null}]
Suppose that $A^\top$ is $\varphi$-approximable. If
$H_A\ne\T^m$, then $\Dhat_A(\widetilde\psi_\varepsilon) \subseteq H_A$ has zero
$m$-dimensional Lebesgue measure. Suppose therefore that
$H_A=\T^m$, and choose
\[
 \varepsilon<\delta<\frac1{2d}.
\]
If $\mathcal L^m(\Dhat_A(\widetilde\psi_\varepsilon))>0$, Theorem
\ref{thm:zero-full}\textup{(i)} gives
$\mathcal L^m(\Dhat_A(\widetilde\psi_\delta))=1$. This contradicts
\eqref{eq:dual-strip-bound}, since $2d\delta<1$. This proves
\eqref{eq:transference-null}.
\end{proof}

\begin{proof}[Proof of Proposition \ref{prop:transference-full}]
Suppose that $A^\top$ is not $\varphi$-approximable. For every
sufficiently large $S$ one then has
\begin{equation}\label{eq:no-dual-solution}
 |A^\top\p|_{\Z^n}>\varphi(S)
 \qquad
 \text{for every }\p\in\Z^m\setminus\{\mathbf0\}
 \text{ with }|\p|\le S.
\end{equation}
Apply Lemma
\ref{lem:inhomogeneous-transference} with $T=\frac{\kappa_d}{\varphi(S)}$. For every $\bb\in\T^m$, it gives a solution of
\[
 |A\q-\mathbf a-\bb|
 \le \frac{\kappa_d}S,
 \qquad
 |\q|\le T.
\]
By \eqref{eq:5.2},
\[
 S=\varphi^{-1}\left(\frac{\kappa_d}T\right)
 \qquad\text{and}\qquad
 \frac{\kappa_d}S
 =\kappa_d\psi\left(\frac{T}{\kappa_d}\right)
 =\widetilde\psi_{\kappa_d}(T).
\]
Since $T$ runs through all sufficiently large values, this proves
\eqref{eq:transference-full}.

To prove the additional assertion, fix $\bb\in\T^m$ and,
for sufficiently large $T$, put
\[
 \mathbf v_T:=\lfloor T/2\rfloor(1,\ldots,1)\in\Z^n,
 \qquad
 S:=\varphi^{-1}\left(\frac{3\kappa_d}{T}\right),
 \qquad Q:=\frac T3.
\]
By \eqref{eq:no-dual-solution},
\eqref{eq:transference-hypothesis} holds for this $S$ and $Q$ in place of $T$. Apply Lemma \ref{lem:inhomogeneous-transference} to
$\bb-A\mathbf v_T$. %
We obtain $\mathbf r\in\Z^n$ and $\mathbf a\in\Z^m$ such that
\[
\begin{gathered}
 |\mathbf r|\le T/3, \qquad 
 |A(\mathbf v_T+\mathbf r)-\mathbf a-\bb|
 \le \frac{\kappa_d}{S}
 =\kappa_d\psi\left(\frac{T}{3\kappa_d}\right)
 \le\widetilde\psi_{3\kappa_d}(T).
\end{gathered}
\]
Every coordinate of $\q:=\mathbf v_T+\mathbf r$ lies between
$T/6-1$ and $5T/6$. Thus, for sufficiently large $T$,
$\q\in\N^n$ and $|\q|\le T$, proving
$\Dhat_{A,\N^n}(\widetilde\psi_{3\kappa_d})=\T^m$.
\end{proof}

\begin{proof}[Proof of Proposition \ref{prop:trivial-shifts}]
Suppose that $A$ is not $\varphi$-approximable and fix
$0<c<1/2$. For every sufficiently large integer $N$ and every nonzero
$\mathbf r\in\Z^n$ with $|\mathbf r|\le2N<N/c$, one has
\begin{equation}\label{eq:lower_on_Ar}
 |A\mathbf r|_{\Z^m}>\varphi(N/c)
 >2c\varphi(N/c)=2\widetilde\varphi_c(N).
\end{equation}
Let $\mathcal Q\subseteq\Z^n$.
Let $\bb\in{\Dhat_{A,\mathcal Q}}(\widetilde\varphi_c)$; for any $T$ large enough, the system \eqref{eq:1.2} has a solution $(\p,\q)\in\Z^m\times \mathcal Q$.
Now let $T$ tend to $N+1$ from below; by the finiteness of $\{\q\in\Z^n:|\q|\le N\}$ and continuity of $\varphi$, for every sufficiently large integer $N$ we can choose $\q_N\in{\mathcal Q}$ such that
\begin{equation}\label{eq:qN_defn}
 |\q_N|\le N,
 \qquad
 |A\q_N-\bb|_{\Z^m}\le\widetilde\varphi_c(N+1).
\end{equation}
Consequently, $|\q_N-\q_{N-1}|\le2N$ and, since
$\widetilde\varphi_c$ is strictly decreasing,
\begin{equation}\label{eq:Ar_upper}
 |A(\q_N-\q_{N-1})|_{\Z^m}
\le\widetilde\varphi_c(N+1)+\widetilde\varphi_c(N)
 <2\widetilde\varphi_c(N).
\end{equation}
Combining \eqref{eq:lower_on_Ar} and \eqref{eq:Ar_upper}, we conclude that $\q_N=\q_{N-1} = \q$ for all sufficiently large $N$. Letting $N\to\infty$ in \eqref{eq:qN_defn} gives
$\bb\in A\mathcal Q\bmod\Z^m$.
The reverse inclusion is immediate, proving the additional assertion.
Taking $\mathcal Q=\Z^n$ gives \eqref{eq:5.51}.
\end{proof}

\section{Concluding remarks and open problems}\label{sec:conclusion}

We close with three directions suggested by the proof: the exact dimension in the non-$c$-invariant case, the role of lower quasimultiplicativity, and a weighted extension.

\subsection{The non-\texorpdfstring{$c$}{c}-invariant case}
\label{subsec:non-c-invariant}

For a fixed $c>0$, Theorem~\ref{thm:uniform-pressure} gives
\begin{equation}\label{eq:non-c-invariant-bounds}
 \mathcal P_{A,\psi}(L,c\e^{-L})-\varepsilon_L
 \le \dim_H\Dhat_A(c\psi)
 \le \mathcal P_{A,\psi}(L,c)+\varepsilon_L,
 \qquad \varepsilon_L=O(L^{-1}).
\end{equation}
The constants implicit in the lattice-counting and packing estimates
contribute only to $\varepsilon_L$, and disappear as $L\to\infty$.
The essential loss is the change of the second parameter from $c$
to $c\e^{-L}$. It is harmless after taking the infimum or supremum
over $c$, and also when the dimension is $c$-invariant. Otherwise,
the two sides of \eqref{eq:non-c-invariant-bounds} need not determine
the dimension for an individual value of $c$.

The covering argument may retain every approximation ball meeting a
given parent, whereas the Cantor construction must select separated
children and then enlarge the approximation function in order to
verify the uniform condition. A coarser equal-$\psi$ mesh makes the
cost of this selection negligible, but produces the factor $\e^L$.
An exact formula therefore seems to require a constant-sensitive block
construction, or a variational principle which records the relative
positions and incidences of clusters at successive scales rather than
only their cardinalities.

\medskip
\noindent\textbf{Question.}
How to determine $\dim_H\Dhat_A(c\psi)$ exactly for every $c>0$?
More specifically, is there a formula, potentially similar to those of Propositions \ref{prop:upper} and \ref{prop:lower}, for which the upper and lower constructions use
the same approximation constant?

\subsection{The role of quasimultiplicativity}
\label{subsec:quasimultiplicativity}

The upper estimate in Proposition~\ref{prop:upper} requires no
quasimultiplicativity.
The essential use of lower quasimultiplicativity occurs in
\eqref{eq:4.13}: for sufficiently large $L$, consecutive mesh
times in \eqref{eq:1.11} are separated by a fixed multiplicative
factor. Hence
$T_{j+1}-T_j$ is comparable to $T_{j+1}$, allowing
\eqref{eq:2.6} to be simplified to \eqref{eq:2.8}.
This gives the lower estimate in
Theorem~\ref{thm:uniform-pressure}; the same separation is used
in Proposition~\ref{prop:no-drop}.

Without this assumption, $T_{j+1}/T_j$ may tend to one: for instance, if $\psi(T)=\e^{-T}$ at large $T$, one has $T_j=jL$ for large $j$. The lower construction then counts increments with
bound $T_{j+1}-T_j=L$, while the upper estimate uses $T_{j+1}$.
The fixed-rescaling comparison used in \eqref{eq:2.8} is
therefore unavailable.

\medskip
\noindent\textbf{Question.}
Can lower quasimultiplicativity be weakened or removed while
retaining explicit lattice-count formulae for the dimensions
of the endpoint sets, possibly using different denominator blocks?
For general continuous, strictly decreasing $\psi\to0$, we use
the definitions \eqref{eq:1.14}, without assuming their equivalence
to \eqref{eq:classical-endpoints}.

\subsection{The weighted setting}\label{subsec:weighted}

Let $\mathbf r=(r_1,\ldots,r_m)$ and
$\mathbf s=(s_1,\ldots,s_n)$ have positive coordinates normalized by
\[
 \sum_{i=1}^m r_i=m,\qquad \sum_{j=1}^n s_j=n.
\]
Define $\Dhat_A^{\mathbf r,\mathbf s}(\psi)$ to be the set of
$\bb\in\T^m$ such that, for every sufficiently large $T$, the system
\begin{equation*}\label{eq:conclusion-weighted-system}
\begin{aligned}
 |(A\bq-\bp-\bb)_i|&\le\psi(T)^{r_i},
 &1\le i\le m,\\
 |q_j|&\le T^{s_j},
 &1\le j\le n,
\end{aligned}
\end{equation*}
has a solution $(\bp,\bq)\in\Z^m\times\Z^n$, where any
representative of $\bb$ in $\R^m$ may be used. When all weights
equal $1$, this is $\Dhat_A(\psi)$. We are interested in the Hausdorff dimension of these sets.

The target sets are now rectangles with different rates of
contraction in different directions. Their lattice-point counts
describe the available approximations, but a Hausdorff dimension
formula must also account for the directions in which these
approximations are distributed. 

\medskip

\noindent\textbf{Question.}
Let $A\in\Mat_{m,n}$, and let $\psi:\R_+\to\R_+$ be continuous,
strictly decreasing and lower quasimultiplicative, with
$\psi(T)\to0$ as $T\to\infty$. Does Theorem~\ref{thm:main} admit
a weighted analogue? In particular,
does
\begin{equation*}\label{eq:conclusion-weighted-no-drop}
 \dim_H\left(
 \bigcap_{c>0}\Dhat_A^{\mathbf r,\mathbf s}(c\psi)
 \right)
 =\inf_{c>0}\dim_H\Dhat_A^{\mathbf r,\mathbf s}(c\psi)
\end{equation*}
hold? Can the dimension in the $c$-invariant case be determined
from the weighted lattice geometry, and which information about
the distribution of approximations in different directions is needed?

\newcommand{\etalchar}[1]{$^{#1}$}
\providecommand{\bysame}{\leavevmode\hbox to3em{\hrulefill}\thinspace}
\providecommand{\MR}{\relax\ifhmode\unskip\space\fi MR }
\providecommand{\MRhref}[2]{%
  \href{http://www.ams.org/mathscinet-getitem?mr=#1}{#2}
}
\providecommand{\href}[2]{#2}

\end{document}